\documentclass[reqno]{amsart}

\usepackage{mathrsfs}
\usepackage[dvips]{graphicx}
\usepackage{amsmath,amsthm, amscd}
\usepackage[psamsfonts]{amssymb}
\usepackage[all]{xy}
\usepackage{wrapfig}

\newtheorem{Thm}{Theorem}[section]
\newtheorem{Prop}[Thm]{Proposition}
\newtheorem{Lem}[Thm]{Lemma}
\newtheorem{Cor}[Thm]{Corollary}

\newtheorem{LemD}[Thm]{Lemma--Definition}

\theoremstyle{remark}
\newtheorem{Rem}[Thm]{Remark}

\theoremstyle{definition}
\newtheorem{Def}[Thm]{Definition}
\newtheorem{Assum}[Thm]{Assumption}

\newcommand{\mysection}[2]{%
\vspace{2mm}\section{\bf #1}\label{#2}
}

\newcommand{\fig}[1]
        {\raisebox{-0.5\height}
                 {\includegraphics{#1}}
        }

\def\Z{{\mathbb Z}}
\def\R{{\mathbb R}}
\def\Q{{\mathbb Q}}

\def\calA{\mathscr{A}}

\def\calC{\mathscr{C}}
\def\calD{\mathscr{D}}
\def\calE{\mathscr{E}}

\def\calG{\mathscr{G}}
\def\calH{\mathscr{H}}

\def\calM{\mathscr{M}}

\def\calU{\mathscr{U}}

\def\calX{\mathscr{X}}

\def\deg{\mathrm{deg}}

\def\tcoprod{\textstyle\coprod}

\def\Conf{C}
\def\bConf{\overline{C}}
\def\EC{E\bConf}
\def\Tr{\mathrm{Tr}}

\def\ve{\varepsilon}
\def\bvec#1{\mbox{\boldmath{$#1$}}}
\def\tbvec#1{\mbox{\scriptsize\boldmath{$#1$}}}

\def\Diff{\mathrm{Diff}}
\def\Emb{\mathrm{Emb}}
\def\fEmb{\Emb^{\mathrm{f}}}
\def\bcalA{\overline{\calA}}
\def\bcalD{\overline{\calD}}
\def\bcalM{\overline{\calM}}

\newcommand{\acalM}{\calM^{\mathrm{Z}}}
\newcommand{\bacalM}{\bcalM\,\kern-.5mm^{\mathrm{Z}}}
\def\ev{\mathrm{ev}}

\def\bD{\mathbb{D}}
\def\bA{\mathbb{A}}

\def\fib{\mathrm{fib}}

\def\adm{\mathrm{adm}}
\def\Morse{\mathrm{Morse}}
\def\t#1{#1^T}

\def\tcoprod{\textstyle\coprod}
\def\ddx#1{\frac{\partial}{\partial #1}}
\def\bbcalM#1#2#3{{\calM'_{#2#3}(#1)}}
\def\winfty{\ell_\infty}
\def\pbcalM{\bcalM\kern-.5mm\,'}
\def\ibcalD{\bcalD\kern-.5mm\,^\infty}
\def\ibcalA{\bcalA\kern-.5mm\,^\infty}
\def\ibcalM{\bcalM\kern-.5mm\,^\infty}

\def\wDelta{\widehat{\Delta}}

\title[Exotic elements of the rational homotopy groups of $\Diff(S^4)$]{Some exotic nontrivial elements of the rational homotopy groups of $\Diff(S^4)$}
\author{Tadayuki Watanabe}
\address{Department of Mathematics, Shimane University,
1060 Nishikawatsu-cho, Matsue-shi, Shimane 690-8504, Japan}
\email{tadayuki@riko.shimane-u.ac.jp}
\date{\today}

\begin{document}

{\noindent\footnotesize {\rm Preprint (2018)}}\par\vspace{15mm}
\maketitle
\vspace{-6mm}
\setcounter{tocdepth}{2}
\numberwithin{equation}{section}
\renewcommand{\thefootnote}{\fnsymbol{footnote}}

\begin{abstract}
This paper studies the rational homotopy groups of the group $\Diff(S^4)$ of self-diffeomorphisms of $S^4$ with the $C^\infty$-topology. We present a method to prove that there are many `exotic' non-trivial elements in $\pi_*\Diff(S^4)\otimes \Q$ parametrized by trivalent graphs. As a corollary of the main result, the 4-dimensional Smale conjecture is disproved. The proof utilizes Kontsevich's characteristic classes for smooth disk bundles and a version of clasper surgery for families. In fact, these are analogues of Chern--Simons perturbation theory in 3-dimension and clasper theory due to Goussarov and Habiro. 
\end{abstract}
\par\vspace{3mm}

\def\baselinestretch{1.06}\small\normalsize

%%%%%%%%%%%%%%%%%%%%%%%%%%%%%%%%%%%%%%%%%%%%
%%%%%%%%%%%%%%%%%%%%%%%%%%%%%%%%%%%%%%%%%%%%
\mysection{Introduction}{s:intro}

The homotopy type of $\Diff(S^4)$ is an important object in topology, whereas almost nothing was known about its homotopy groups except that they include those coming from the orthogonal group $O_5$ (e.g., recent surveys in \cite{Hat2, Kup}). Let $\Diff(D^d,\partial)$ denote the group of self-diffeomorphisms of $D^d$ which fix a neighborhood of $\partial D^d$ pointwise. This is the `non-linear' part of $\Diff(S^d)$ in the sense of the well-known splitting $\Diff(S^d)\simeq O_{d+1}\times \Diff(D^d,\partial)$ (e.g., \cite{ABK}). For $d=1,2,3$, it is known that $\Diff(D^d,\partial)$ is contractible. Proof for $d=1$ is easy. The case $d=2$ is due to Smale (\cite{Sm}), and a proof for the case $d=3$ (the Smale conjecture) has been given by Hatcher (\cite{Hat}). On the other hand, for $d\geq 5$, it is known that $\Diff(D^d,\partial)$ is not contractible (e.g., \cite{Hat2}). For $d=4$, there is a conjecture which claims that $\Diff(D^4,\partial)$ is contractible, or equivalently, $\Diff(S^4)\simeq O_5$ (the 4-dimensional Smale conjecture \cite[{Problem~4.34, 4.126}]{Kir}, \cite{RS}). The following theorem, which is the main result of this paper, gives a negative answer to this conjecture.

\begin{Thm}[Theorem~\ref{thm:commute}]\label{thm:main}
For each $k\geq 1$, evaluation of Kontsevich's characteristic classes on $D^4$-bundles over $S^k$ gives an epimorphism
from $\pi_k B\Diff(D^4,\partial)\otimes\Q$ to the space $\calA_k$ of trivalent diagrams (definition in \S\ref{ss:graphs}).
\end{Thm}

\begin{Rem}
Theorem~\ref{thm:main} gives no information about $\pi_1B\Diff(D^4,\partial)\cong \pi_0\Diff(D^4,\partial)$ because $\calA_1=0$. The first nontrivial element is detected in $\calA_2\cong \Q$ (Proposition~\ref{prop:A_2}). The topological version of the 4-dimensional Smale conjecture: $\mathrm{TOP}(S^4)\simeq O_5$ has been disproved by Randall and Schweitzer in \cite{RS}. 
\end{Rem}

Kontsevich's characteristic classes, defined in \cite{Kon}, are invariants for fiber bundles with fiber a punctured homology sphere. They were defined, as an analogy to Chern--Simons perturbation theory in 3-dimension, by utilizing a graph complex and configuration space integrals, both developed by Kontsevich in \cite{Kon}. The method of this paper is essentially the same as \cite{Wa2}, where we studied the rational homotopy groups of $\Diff(D^{4k-1},\partial)$. Namely, we construct some explicit fiber bundles from trivalent graphs, by using a higher-dimensional analogue of graph-clasper surgery, developed by Goussarov and Habiro for knots and 3-manifolds (\cite{Gou, Hab}). Then we compute the values of the characteristic numbers for the bundles. 

In fact, the restriction to 4-dimensional fiber in this paper would not be essential and the results could be given for arbitrary fiber dimensions $\geq 4$. This is similar to the fact that the cocycles of $\Emb(S^1,\R^d)$ given by configuration space integrals are nontrivial and $d=4$ is not exceptional there (\cite{Kon, CCL}). This paper has some ad hoc arguments that are special in 4-dimension, in Lemmas~\ref{lem:T-bundle-II}, \ref{lem:morse-complex-invariant}, and \ref{lem:Y-restrict}.

\if0
By a theorem of Hubbuck (\cite{Hu}), which states that every connected, homotopy-commutative $H$-space of finite type has the homotopy type of a point or a torus $S^1\times\cdots\times S^1$, in particular, $\pi_i=0$, $i\geq 2$, the following holds.
\begin{Cor}
$\Diff(D^4,\partial)$ does not have the homotopy type of a finite CW complex.
\end{Cor}
This is an analogue of Antonelli--Burghelea--Kahn's theorem (\cite{ABK}), which states that $\Diff(S^d)$ does not have the homotopy type of a finite CW complex if $d\geq 7$.
\fi

%%%%%%%%%%%%%%%%%%%%%%%%%%%%5
\subsection{Some consequences of Theorem~\ref{thm:main}}

Theorem~\ref{thm:main} answers to some problems in Kirby's problem list \cite{Kir}. 

\begin{enumerate}
\item $\Diff(S^4)\not\simeq O_5$. (cf. \cite[{Problem~4.34, 4.126 (D.~Randall)}]{Kir})
\item There is a bundle over $S^2$, with a 4-manifold as fiber, which is topologically trivial but not smoothly trivial. (cf. \cite[Problem~4.122 (K.~Fukaya)]{Kir}) 
\item The space $\mathit{Sympl}$ of all standard-at-infinity symplectic structures on $\R^4$ is not contractible. (cf. \cite[Problem~4.141 (Eliashberg)]{Kir}, \cite[7.3]{El})
\end{enumerate}
Here, (2) follows from the contractibility of $\mathrm{TOP}(D^4,\partial)$, which can be shown by the Alexander trick, and (3) follows from Theorem~\ref{thm:main} and the remark given in \cite[Problem~4.141]{Kir}, which says that the evaluation map $\Diff(D^4,\partial)\to \mathit{Sympl}$ is a fibration whose fiber is the group of self-symplectomorphisms of $(D^4,\omega_0)$ fixed at the boundary, where $\omega_0$ is the standard symplectic form. This group is contractible by a deep result of Gromov based on his theory of pseudo-holomorphic curves.

As well as \cite[Appendix]{Hat}, the 4-dimensional Smale conjecture has several equivalent statements. By Morlet's equivalence $\Diff(D^d,\partial)\simeq \Omega^{d+1}\mathrm{PL}_d/O_d$ (\cite{BL}), we have the following.
\begin{enumerate}
\setcounter{enumi}{3}
\item $\mathrm{PL}_4\not\simeq O_4$.
\end{enumerate}

Let $\Emb(S^3,\R^4)_0$ denote the component of $\Emb(S^3,\R^4)$ of the standard inclusion. By the fibration sequence $\Diff(D^4,\partial)\to \Emb(D^4,\R^4)\to \Emb(S^3,\R^4)$, a generalized version of the 4-dimensional Schoenflies conjecture fails:
\begin{enumerate}
\setcounter{enumi}{4}
\item $\Emb(S^3,\R^4)_0\not\simeq \Emb(D^4,\R^4)\,\,(\simeq O_4)$\footnote{The 4-dimensional Schoenflies conjecture claims that $\pi_0\Emb(S^3,\R^4)\cong\pi_0\Emb(D^4,\R^4)\,(=\pi_0O_4)$.}.
\end{enumerate}

An element of the group $\calC(M)=\Diff(M\times I, \partial M\times I\cup M\times\{0\})$ of relative diffeomorphisms is called a {\it pseudo-isotopy}. By the fibration sequence $\Diff(D^{d+1},\partial)\to \calC(D^d) \to \Diff(D^d,\partial)$, Hatcher's theorem $\Diff(D^3,\partial)\simeq *$, and Theorem~\ref{thm:main}, we have the following.
\begin{enumerate}
\setcounter{enumi}{5}
\item $\calC(D^3)\not\simeq *$.
\end{enumerate}
This implies that pseudo-isotopy and isotopy of $D^3$ are essentially different. 

By $\pi_0\Diff(D^5,\partial)\approx \Theta_6=0$ (\cite{Ce}, \cite{KM}), $\pi_1\Diff(D^4,\partial)\otimes\Q\neq 0$, and the homotopy sequence for the fibration $\calC(D^4)\to \Diff(D^4,\partial)$, we have the following.
\begin{enumerate}
\setcounter{enumi}{6}
\item $\pi_1\calC(D^4)\otimes\Q\neq 0$.
\end{enumerate}

By considering the fibrations
$\Diff(S^3\times D^1,\partial)\to \Diff(D^4,\partial) \to \Emb(D^4,\mathrm{Int}\,D^4)$, $\Diff(S^3\times D^1,\partial)\times \Diff(S^3\times D^1,\partial)\to \Diff(S^3\times D^1,\partial)\to \Emb(S^3,S^3\times D^1)$,
we obtain the following.
\begin{enumerate}
\setcounter{enumi}{7}
\item $\Diff(S^3\times D^1,\partial)\not\simeq \Omega O(4)$.
\item $\Emb(S^3,S^3\times D^1)_0\not\simeq SO(4)$, where $\Emb(S^3,S^3\times D^1)_0$ is the component of the standard inclusion $S^3\to S^3\times \{0\}\subset S^3\times D^1$.
\end{enumerate}

In (1), (3), (4), (5), (6), (8), (9), the deficiency of being a homotopy equivalence can be measured by $\Diff(D^4,\partial)$. Other interesting statements that are equivalent to the 4-dimensional Smale conjecture are described in \cite{RS}.

%%%%%%%%%%%%%%%%%%%%%%%%%%%%5
\subsection{Content of the paper}

The rest of this paper consists of four parts.

\begin{enumerate}
\item[\bf \S\ref{s:kon}.] {\bf Kontsevich's characteristic classes} \hfill {\bf p.\pageref{s:kon}}\\
We review the definition of the invariant $\hat{Z}_k^\adm$ given by Kontsevich's characteristic classes for $D^4$-bundles. 
The main result is restated in terms of $\hat{Z}_k^\adm$.
\item[\bf \S\ref{s:GCF}.] {\bf Graph counting formula}  \hfill {\bf p.\pageref{s:GCF}}\\
We give a formula for $\hat{Z}_k^\adm$ counting flow-graphs of gradients of Morse functions. This is an analogue of the relation between Kontsevich's configuration space integral invariant of rational homology 3-spheres in \cite{Kon} and Fukaya's Morse homotopy invariant in \cite{Fu, Wa3}, proved by Shimizu in \cite{Sh} using Lescop's description \cite{Les1} of configuration space integral invariant. The formula allows us to compute the invariant by geometric arguments and hopefully makes the problem simple. 
\item[\bf \S\ref{s:cycles}.] {\bf Cycles in $B\Diff(D^4,\partial)$ associated to graphs}  \hfill {\bf p.\pageref{s:cycles}}\\
We shall construct a $D^4$-bundle $\pi^\Gamma:E^\Gamma\to B_\Gamma$ concretely from a trivalent graph $\Gamma$. This is a higher-dimensional analogue of graph-clasper surgery of \cite{Gou,Hab,GGP} and is similar to what we have given for $(4k-1)$-dimensional disk bundles in \cite{Wa2}. 
\item[\bf \S\ref{s:computation}.] {\bf Computation of the invariant}  \hfill {\bf p.\pageref{s:computation}}\\
We shall compute the value of the invariant $\hat{Z}_k^\adm$ for the $D^4$-bundles constructed in \S\ref{s:cycles}. This part is the core of the proof and is philosophically based on the computation of Kuperberg and D.~Thurston \cite{KT}. Thanks to the graph counting formula, it is enough to count only the flow-graphs that are caught in some restricted places where the gradient on fibers varies drastically in a parameter, which are highly restricted. 
\end{enumerate}

Most of \S\ref{s:kon}--\S\ref{s:cycles} consists of formal arguments such as definitions and confirmations associated with them, and there aren't major differences from known results there, though there are some new techniques to simplify arguments. In \S\ref{s:computation}, we choose Morse functions that is adapted to the surgery and consider ``coherent $v$-gradients'', and see that they make the computation into a homological one. 

The reasons that the proof is mainly given by parametrized Morse theory, unlike that of \cite{Wa2} of differential forms, are as follows.
\begin{enumerate}
\item We consider that the proof of the present paper is geometric and concrete, although the proof is a bit lengthy mainly due to some arguments of general position, compactness and orientation in finite dimensional manifolds, which are routine. 
\item We consider that giving different proofs would make the nontriviality result of Kontsevich's characteristic classes more solid. 
\item It would be interesting to compare the results of this paper with known results about stable pseudo-isotopy, some of which utilize Cerf theory (e.g., \cite{Ce, HW, Ig}). 
\item Our Z-paths, Z-graphs and geometric iterated integrals in \S\ref{s:GCF} may be of independent interest.
\end{enumerate}

%%%%%%%%%%%%%%%%%%%%%%%%%%%%5
\subsection{Notations and conventions}

For a sequence of submanifolds $A_1,A_2,\ldots,A_r\subset W$ of a smooth Riemannian manifold $W$, we say that the intersection $A_1\cap A_2\cap \cdots\cap A_r$ is {\it transversal} if for each point $x$ in the intersection, the subspace $N_xA_1+N_xA_2+\cdots+N_xA_r\subset T_xW$ spans the direct sum $N_xA_1\oplus N_xA_2\oplus\cdots\oplus N_xA_r$, where $N_xA_i$ is the orthogonal complement of $T_xA_i$ in $T_xW$ with respect to the Riemannian metric. 

For manifolds with corners and their (strata) transversality, we follow \cite[Appendix]{BT} (see also \cite[Appendix~A]{Wa3}).

As chains in a manifold $X$, we consider $\Q$-linear combinations of finitely many smooth maps from compact oriented manifolds with corners to $X$. We say that two chains $\sum n_i\sigma_i$ and $\sum m_j\tau_j$ ($n_i,m_j\in\Q$, $\sigma_i,\tau_j$: smooth maps from compact manifolds with corners) are strata transversal if for every pair $i,j$, the terms $\sigma_i$ and $\tau_j$ are strata transversal. Strata transversality among two or more chains can be defined similarly. The intersection number $\langle \sigma,\tau\rangle_X$ of strata transversal two chains $\sigma=\sum n_i\sigma_i$ and $\tau=\sum m_j\tau_j$ with $\dim{\sigma_i}+\dim{\tau_j}=\dim{X}$ is defined by $\sum_{i,j}n_im_j(\sigma_i\cdot \tau_j)$. We also consider intersection $\langle \sigma_1,\ldots,\sigma_n\rangle_X$ of strata transversal chains $\sigma_1,\ldots,\sigma_n$ for $n\geq 2$, which is defined similarly.

We denote by $|x|$ the degree of an element $x$ of a graded module. 

The diagonal $\{(x,x)\in X\times X\mid x\in X\}$ is denoted by $\Delta_X$. 

For a fiber bundle $\pi:E\to B$, we denote by $T^vE$ the (vertical) tangent bundle along the fiber $\mathrm{Ker}\,d\pi\subset TE$. Let $ST^vE$ denote the subbundle of $T^vE$ of unit spheres. Let $\partial^\fib E$ denote the fiberwise boundaries: $\bigcup_{b\in B}\partial(\pi^{-1}\{b\})$. We denote by $\pi_*$ the pushforward or integral along the fiber for $\pi$. We will use the well-known identity: $\pi_*(\pi^*\alpha\wedge \beta)=\alpha\wedge \pi_*\beta$. 

In Appendix, we describe convention for orientation (\S\ref{s:ori}).

%%%%%%%%%%%%%%%%%%%%%%%%%%%%%%%
%%%%%%%%%%%%%%%%%%%%%%%%%%%%%%%
\mysection{Kontsevich's characteristic class}{s:kon}

Kontsevich's characteristic classes are invariants for fiber bundles with fiber a punctured {\it framed} homology $d$-sphere. Here, we shall see that an invariant for {\it unframed} $(D^d,\partial)$-bundles can be obtained by adding to Kontsevich's characteristic class a correction term, in a similar way as \cite{KT, Les1}.

%%%%%%%%%%%%%%%%%%%%%%%%%%%%%%%
\subsection{Graphs}\label{ss:graphs}

In this paper, we consider connected trivalent graphs. In general, trivalent graph has even number of vertices, and if it is $2k$, then the number of edges is $3k$. Let $V(\Gamma)$ and $E(\Gamma)$ denote the sets of vertices and edges of a trivalent graph $\Gamma$, respectively. Labellings of a trivalent graph $\Gamma$ are given by bijections $V(\Gamma)\to \{1,2,\ldots,2k\}$, $E(\Gamma)\to \{1,2,\ldots,3k\}$. Let $\calG_k$ be the vector space over $\Q$ spanned by the set $\calG_k^0$ of all labelled connected trivalent graphs with $2k$ vertices. We define 
\[ \calA_k=\calG_k/{\mbox{IHX relation, label change relation}}, \]
where the IHX relation is given in Figure~\ref{fig:IHX} and the label change relation is generated by the following relations:
\[ \Gamma'\sim -\Gamma,\qquad \Gamma''\sim \Gamma. \]
Here, $\Gamma'$ is the graph obtained from $\Gamma$ by exchanging labels of two edges, $\Gamma''$ is the graph obtained from $\Gamma$ by exchanging labels of two vertices. This is the trivalent part of the cohomology of Kontsevich's graph complex which works for even dimensional manifolds \cite{Kon}.

\begin{figure}
\begin{center}
\includegraphics[height=15mm]{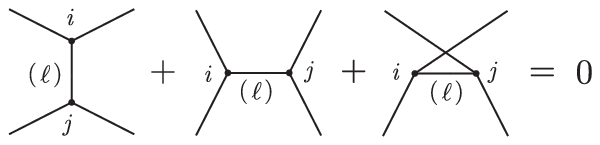}
\end{center}
\caption{IHX relation}\label{fig:IHX}
\end{figure}

\begin{Prop}\label{prop:A_2}
$\calA_1=0$, and $\calA_2$ is 1-dimensional and generated by the class of the complete graph $K_4$ on four vertices with some labels. 
\end{Prop}
\begin{proof}
By the label change relation and the IHX relation, one may see that a graph with a self-loop or multiple edges vanishes in $\calA_k$. It follows that $\calA_1=0$ and $\calA_2$ is spanned by the class of $K_4$. Since the IHX relation for $\calA_2$ imposes no restriction other than the vanishings of the graphs with a self-loop or multiple edges, we need only to check that the label change relation does not make $\calA_2$ trivial. The automorphism group of $K_4$ is isomorphic to the symmetric group $\mathfrak{S}_4$ and each automorphism changes the labels of edges by an even permutation, namely, an automorphism of $K_4$ never produces $-K_4$, which implies that $K_4\neq 0$ in $\calA_2$.
\end{proof}
\begin{Rem}
That $K_4$ represents a nontrivial class in $\calA_2$ is a special case of \cite[Example~2.5]{CGP}. 
One may easily check that $\calA_3=0$. The dimensions of $\calA_k$ for $4\leq k\leq 9$ are computed in \cite{BNM} as in the following table.
\par\medskip
\begin{center}
\begin{tabular}{c|ccccccccc}\hline
$k$ & 1 & 2 & 3 & 4 & 5 & 6 & 7 & 8 & 9\\\hline
$\dim\,\calA_k$ & 0 & 1 & 0 & 0 & 1 & 0 & 0 & 0 & 1\\\hline
\end{tabular}
\end{center}
\end{Rem}

%%%%%%%%%%%%%%%%%%%%%%%%%%%%%%%
\subsection{$(D^d,\partial)$-bundles and associated $\bConf_n(S^d,\infty)$-bundles}

A {\it $(D^d,\partial)$-bundle} is a smooth $D^d$-bundle $\pi:E\to B$ equipped with a trivialization on a neighborhood of $\partial^\fib E$. We say that a $(D^d,\partial)$-bundle is {\it pointed} if it is given a diffeomorphism between the fiber $F_0$ over the base point $b_0\in B$ and the standard unit disk $D^d$, which is compatible with the trivialization on a neighborhood of $\partial^\fib E$. A pointed $(D^d,\partial)$-bundle corresponds to a classifying map $(B,b_0)\to (B\Diff(D^d,\partial),*)$. We take a Riemannian metric on the total space $E$ of a $(D^d,\partial)$-bundle that agrees with the standard one on $D^d$ near $\partial^\fib E$ under the trivialization. 

Put $\infty=(0,\ldots,0,1)\in S^d$ and let $U_\infty\subset S^d$ denote the hemisphere $S^d_+=\{(x_1,\ldots,x_{d+1})\in S^d\mid x_{d+1}\geq 0\}$, and put $U_\infty'=U_\infty-\{\infty\}$.
The group $\Diff(D^d,\partial)$ acts on $S^d$ through the action on $S^d_-=\{(x_1,\ldots,x_{d+1})\in S^d\mid x_{d+1}\leq 0\}$. 

Let $C_n(S^d,\infty)$ denote the configuration space of ordered points in $S^d-\{\infty\}$:
\[ C_n(S^d,\infty)=\{(x_1,\ldots,x_n)\in (S^d-\{\infty\})^n\mid x_i\neq x_j \mbox{ for }i\neq j\} \]
and let $\bConf_n(S^d,\infty)$ denote its Fulton--MacPherson compactification (\cite{AS, Kon}, see also \cite[2.3]{Wa2}). Roughly, $\bConf_n(S^d,\infty)$ is a compact manifold with corners, which is obtained from $\Conf_n(S^d,\infty)$ by attaching boundary strata so that $\bConf_n(S^d,\infty)\simeq \Conf_n(S^d,\infty)$. The codimension 1 (boundary) stratum of $\bConf_n(S^d,\infty)$ consists of faces obtained by blowing-up the diagonal in $(S^d-\{\infty\})^n$ that corresponds to coincidence of points with labels in a subset $A\subset\{1,\ldots,n,\infty\}$.

These spaces admit natural diagonal $\Diff(D^d,\partial)$-actions $g\cdot(x_1,\ldots,x_n)=(g\cdot x_1,\ldots,g\cdot x_n)$. We call the $S^d$-bundle associated to a $(D^d,\partial)$-bundle $\pi:E\to B$ an {\it $(S^d,U_\infty)$-bundle} and call its associated $\R^d=S^d-\{\infty\}$-bundle an {\it $(\R^d,U_\infty')$-bundle}. Also, we consider the $\bConf_n(S^d,\infty)$-bundle
\[ \bConf_n(\pi):\EC_n(\pi)\to B \]
associated to $\pi$. 

We equip $S^d$ with the orientation induced from the unit disk in $\R^{d+1}$ with the standard orientation $dx_1\wedge\cdots\wedge dx_{d+1}$. We orient $C_n(S^d,\infty)$ by $o(C_n(S^d,\infty))_{(x_1,\ldots,x_n)}=o(S^d)_{x_1}\wedge\cdots\wedge o(S^d)_{x_n}$, where $o(X)$ denotes the orientation of an oriented manifold $X$, and the orientation of $\bConf_n(S^d,\infty)$ is defined similarly. We orient $\EC_n(\pi)=\bigcup_t \bConf_n(S^d,\infty)_t$ at $(t,y)$ by $o(B)_t\wedge o(\bConf_n(S^d,\infty)_t)_y$.

%%%%%%%%%%%%%%%%%%%%%%%%%%%%%%%
\subsection{Admissible propagator}\label{ss:adm-propagator}

The invariant will be defined via the intersection among some fundamental chains in $\EC_n(\pi)$ called admissible propagators. Admissible propagator was first considered in \cite{Les1, Les2} for rational homology 3-sphere. Let us review the definition of the Fulton--MacPherson compactification $\bConf_2(S^d,\infty)$ for two points. Let the subspaces $\Sigma_0\subset\Sigma_1$ of $S^d\times S^d$ be given as follows.
\[ \Sigma_0=\{(\infty,\infty)\},\quad \Sigma_1=\Delta_{S^d}\cup (S^d\times\{\infty\})\cup (\{\infty\}\times S^d). \]
We consider the real blow-up $B\ell(S^d\times S^d,\Sigma_0)$ of $S^d\times S^d$ along $\Sigma_0$, which is defined by replacing $\Sigma_0$ with its normal sphere $S^{2d-1}$. Then $\partial B\ell(S^d\times S^d,\Sigma_0)=S^{2d-1}$ and $\mathrm{Int}\,B\ell(S^d\times S^d,\Sigma_0)$ is naturally identified with $S^d\times S^d-\Sigma_0$. The closure $\overline{\Sigma_1-\Sigma_0}$ of $\Sigma_1-\Sigma_0$ in $B\ell(S^d\times S^d,\Sigma_0)$ is a disjoint union of three $d$-submanifolds, whose boundaries are transversal to $\partial B\ell(S^d\times S^d,\Sigma_0)$. Now $\bConf_2(S^d,\infty)$ is obtained by the real blowing-up along $\overline{\Sigma_1-\Sigma_0}$: 
\[ \bConf_2(S^d,\infty)=B\ell(B\ell(S^d\times S^d,\Sigma_0),\overline{\Sigma_1-\Sigma_0}) \]
Roughly, the real blow-up along $\overline{\Sigma_1-\Sigma_0}$ is obtained by replacing the $d$-submanifolds with their normal $S^{d-1}$-bundles. A piecewise smooth map (Gauss map)
\[ \phi_0:\partial\bConf_2(S^d,\infty)\to S^{d-1} \]
is defined as follows.
\begin{enumerate}
\item The face obtained by blowing-up along $\Sigma_0$ is canonically identified with $S^{2d-1}=\{(y_1,y_2)\in (\R^d)^2\mid \|y_1\|^2+\|y_2\|^2=1\}$, which is viewed as the `limit' of $\{(y_1,y_2)\in(\R^d)^2\mid \|y_1\|^2+\|y_2\|^2=R^2\}$ for $R\to \infty$. A map $S^{2d-1}-\overline{\Delta}_{S^d}\to S^{d-1}$ is defined by $\phi_0(y_1,y_2)=\frac{y_2-y_1}{\|y_2-y_1\|}$. 
\item The face obtained by blowing-up along $S^d\times\{\infty\}$ is identified with $B\ell(S^d,\infty)\times \partial B\ell(S^d,\infty)\cong B\ell(S^d,\infty)\times S^{d-1}$. So the projection on the second factor gives a map to $S^{d-1}$, where we consider $S^{d-1}=\{y_2\in \R^d\mid \|y_2\|^2=1\}$. The face obtained by blowing-up along $\{\infty\}\times S^d$ is similar.
\item The face obtained by blowing-up $\Delta_{S^d}$ is identified with $\{(-y_2,y_2)\in(\R^d)^2\mid 2\|y_2\|^2=1\}\times B\ell(\Delta_{S^d},(\infty,\infty))\cong S^{d-1}\times B\ell(\Delta_{S^d},(\infty,\infty))$. So the projection on the first factor gives a map to $S^{d-1}$. 
\end{enumerate}
These maps are glued together along the boundaries and define a piecewise smooth map $\phi_0$, which may be smoothly extended to a neighborhood of the boundary. 

Let $B$ be a compact oriented manifold. For an $(S^d,U_\infty)$-bundle $\pi:E\to B$, we consider the associated $S^d\times S^d$-bundle $E\times_B E\to B$. By fiberwise blowing-up $E\times_B E$ as in (3) above, we obtain $\EC_2(\pi)$. Namely, if $E=\bigcup_{t\in B}(S^d)$, then $\EC_2(\pi)$ is $\bigcup_{t\in B}\bConf_2((S^d)_t,\infty)$. We denote by $S_{\Delta_E}$ the unit $S^{d-1}$-bundle $ST^v\Delta_E$ of $T^v\Delta_E$, where $\Delta_E\to B$ is the subbundle of the associated $B\ell(S^d\times S^d,\Sigma_0)$-bundle corresponding to $\overline{\Delta_{S^d}-\Sigma_0}$. Then $S_{\Delta_E}$ can be identified with the face of $\EC_2(\pi)$ obtained by fiberwise blowing-up $\Delta_E$. The map $\phi_0:\partial\bConf_2(S^d,\infty)\to S^{d-1}$ extends naturally to $\partial^\fib\EC_2(\pi)-\mathrm{Int}\,S_{\Delta_E}$.

\begin{Def}[Admissible propagator]
We say that a piecewise smooth singular chain $\theta$ of $\EC_2(\pi)$ over $\Q$ of codimension $d-1$ is an {\it admissible propagator} if it satisfies the following.
\begin{enumerate}
\item $\partial \theta$ is a chain of $\partial \EC_2(\pi)$ and $\theta$ is strata transversal to $\bConf_2(\pi)^{-1}(\partial B)$. 
\item $\theta$ is standard on $\partial^\fib\EC_2(\pi)-\mathrm{Int}\,S_{\Delta_E}$, namely, $\partial\theta$ is the sum of three chains: a chain $\phi_0^{-1}(\{a,-a\})$ (for some $a\in S^{d-1}$) of $\partial^\fib\EC_2(\pi)-\mathrm{Int}\,S_{\Delta_E}$, a $\Z_2$-invariant chain of $S_{\Delta_E}$, and a chain of $\bConf_2(\pi)^{-1}(\partial B)$. We call each term in the sum a {\it restriction} of $\theta$ on the corresponding part. Let $\theta_{\partial B}$ denote the chain of $\bConf_2(\pi)^{-1}(\partial B)$ given by the restriction of $\partial\theta$. 
\item The chain $\partial\theta - \theta_{\partial B}$ of $\partial^\fib \EC_2(\pi)$ is invariant under the $\Z_2$-action given by swapping of two points. We call such a chain of $\partial \EC_2^\fib(\pi)$ an {\it admissible section}.
\end{enumerate}
When a chain of $ST^vE$ or of a general $S^{d-1}$-bundle is invariant under the $\Z_2$-action of the involution of the fiber $S^{d-1}$, we also call such a chain an admissible section. (Note that an admissible section is not a genuine section of $S^{d-1}$-bundle.)
\end{Def}
There exists an admissible propagator. Indeed, we will give later an explicit admissible propagator by using Morse theory.

\begin{Lem}\label{lem:propagator-unique}
\begin{enumerate}
\item There exists an admissible propagator. When $B$ is a closed manifold, the relative homology class of admissible propagator in\\ $H_{\dim{B}+d+1}(\EC_2(\pi),\partial\EC_2(\pi))$ is unique. 
\item When $B$ is a closed manifold, the homology class of admissible section of $\partial \EC_2(\pi)$ in $H_{\dim{B}+d}(\partial\EC_2(\pi))$ is unique.
\end{enumerate}
\end{Lem}
\begin{proof}
The lemma follows immediately from the proof of \cite[Lemma~3]{Wa1}.
\end{proof}

%%%%%%%%%%%%%%%%%%%%%%%%%%%%%%%
\subsection{Homotopy invariant $\hat{Z}_k^\adm$ via admissible propagator}

In the following, we assume $d=4$, $k\geq 1$. Let $B$ be a $k$-dimensional compact oriented manifold possibly with corners and let $\pi:E\to B$ be a $(D^4,\partial)$-bundle. Let $\theta^{(1)},\theta^{(2)},\ldots,\theta^{(3k)}$ be admissible propagators of $\EC_2(\pi)$. For $\Gamma\in\calG_k^0$, we define $I^\adm(\Gamma)\in\Q$ by the following.
\begin{equation}\label{eq:I^adm}
 I^\adm(\Gamma)=\langle\phi_1^{-1}\theta^{(1)},\phi_2^{-1}\theta^{(2)},\cdots, \phi_{3k}^{-1}\theta^{(3k)}\rangle_{\EC_{2k}(\pi)}
\end{equation}
Here $\phi_j:\EC_{2k}(\pi)\to \EC_2(\pi)$ is the projection which gives the endpoints of the $j$-th edge and $\phi_j^{-1}\theta^{(j)}$ denotes the chain obtained by pulling back simplices in $\EC_2(\pi)$ by $\phi_j$. We choose the order of the two points arbitrarily.
If $\vec{\theta}=(\theta^{(1)},\theta^{(2)},\ldots,\theta^{(3k)})$ is generic, the intersection of (\ref{eq:I^adm}) is strata transversal. Choosing such $\vec{\theta}$, we define
\[ Z_k^\adm(\vec{\theta})=\frac{1}{2^{3k}(2k)!(3k)!}\sum_{\Gamma\in\calG_k^0} I^\adm(\Gamma)[\Gamma]\in \calA_k.  \]

Now we let $B=I^k$, $I=[0,1]$, and suppose that $\pi$ is standard on $\partial I^k$, namely, it corresponds to a relative classifying map
$\psi_\pi:(I^k,\partial I^k) \to (B\Diff(D^4,\partial),*)$.
For this, we shall define a correction term which turns $Z_k^\adm$ into an invariant under relative homotopy of the classifying map. We consider the iteration of $\psi_\pi$:
\[ \begin{split}
  \psi_{N_k}&=\underbrace{\psi_\pi \natural \cdots \natural \psi_\pi}_{N_k}:(I^k,\partial I^k) \to (B\Diff(D^4,\partial),*),\\
  \end{split} \]
where $N_k=m(\Theta^{k+4})\,m(\pi_{k+4}SO_4)\,m(\pi_{k+3}SO_4)$, $\Theta^n$ is the group of $h$-cobordism classes of homotopy $n$-spheres (or diffeomorphism classes for $n\geq 5$ by Smale's $h$-cobordism theorem \cite{Sm2}), $m(G):=\min\{d\in\Z_{>0}\mid dx=0\mbox{ for all }x\in G\}$ for a finite abelian group $G$. Since $\pi_jSO_4$ is finite for $j\geq 4$ and so is $\Theta^n$ for $n\geq 5$, $N_k$ is finite for $k\geq 1$. Let $N_k\pi:N_kE\to I^k$ be the $(D^4,\partial)$-bundle corresponding to $\psi_{N_k}$ obtained by iteration. Since $N_k\pi$ is standard on $\partial I^k$, $\partial N_kE$ is identified with $\partial (D^4\times I^k)$ in a strata preserving way.

\begin{Lem}\label{lem:NE-trivial}
\begin{enumerate}
\item $N_kE$ is relatively diffeomorphic to $D^4\times I^k$ as manifolds with corners, in the sense of \cite[Appendix]{BT}.
\item The tangent bundle $T^v(N_kE)\to N_kE$ along the fiber is trivial and there is a canonical trivialization of $T^v(N_kE)$ up to homotopy.
\end{enumerate}
\end{Lem}
\begin{proof}
(1) holds by the multiplication by $m(\Theta^{k+4})$. For (2), the obstruction to extending the trivialization of $T^vE|_{\partial E}$ to $T^vE$ of the original bundle $\pi:E\to I^k$ belongs to $H^{k+4}(E,\partial E;\pi_{k+3}SO_4)\cong \pi_{k+3}SO_4$, and if there is an extension, then the possibility for different choices is given by
$H^{k+4}(E,\partial E;\pi_{k+4}SO_4)\cong \pi_{k+4}SO_4$.
By multiplying $m(\pi_{k+4}SO_4)\,m(\pi_{k+3}SO_4)$, the obstructions vanish.
\end{proof}

Now we shall define a correction term which turns $Z_k^\adm(\vec{\theta})$ into an invariant of $(D^4,\partial)$-bundles, in a similar manner as \cite{Wa3, Sh}. For this we shall take a cobordism $W$ between an iteration of $E$ and the trivial $D^4$-bundle over $I^k$ with a trivial rank 4 vector bundle $T^vW$ on it. Moreover, we shall take a sequence $\vec{\eta}$ of admissible sections of $ST^vW$ with a prescribed behavior near the ends.

Put $N_k'=m(\Theta^{k+5})N_k$. Let $N_k'\pi:N_k'E\to I^k$ be the iteration defined similarly to $N_k\pi$. 
We denote by $W$ the cobordism between $N_k'E$ and $D^4\times I^k$ that is obtained by identifying $D^4\times I^k\times \{1\}$ in $D^4\times I^k\times I$ with $N_k'E$ by the iteration $m(\Theta^{k+5})\iota$ of the diffeomorphism $\iota:D^4\times I^k\times \{1\}\to N_kE$ of Lemma~\ref{lem:NE-trivial} (1) formed along the $I^k$ direction. Since $\pi_0\Diff(D^n,\partial)\cong \Theta^{n+1}$ for $n\geq 5$ (\cite{Ce}) and $N_k'$ has the factor $m(\Theta^{k+5})$, the diffeomorphism $m(\Theta^{k+5})\iota$ is unique up to isotopy. Moreover, by identifying $TD^4\times I^k\times\{1\}$ in the trivial $\R^4$-bundle $TD^4\times I^k\times I$ with $T^v(N_k'E)$ by using $m(\Theta^{k+5})\iota$ and the $m(\Theta^{k+5})$ times iteration of the trivialization of Lemma~\ref{lem:NE-trivial} (2), we obtain a trivial $\R^4$-bundle $T^vW\to W$ that extends $T^v(N_k'E)$ and $TD^4\times I^k$ on the endpoints. We take a metric on $T^vW$ extending the given ones on the boundary. 

An iteration of $\vec{\theta}$ restricts to a tuple of admissible sections on a subset of $ST^vW$ and we consider its extension on $ST^vW$, as follows. The face of $\partial^\fib\EC_2(N_k'\pi)$ obtained by blowing-up along $\Delta_{N_k'E}$ is identified with $ST^v(N_k'E)$. Also, the restriction of the unit sphere bundle $ST^vW$ on $D^4\times I^k\times \{1\}$ is $ST^v(N_k'E)$:
\[ \partial^\fib\EC_2(N_k'E)\quad\supset\quad ST^v(N_k'E)\quad\subset\quad ST^vW \]
The tuple $N_k'\vec{\theta}=(N_k'\theta^{(1)},\ldots,N_k'\theta^{(3k)})$ of iterations of admissible propagators of $\EC_2(\pi)$ restricts to a tuple of admissible sections of $ST^v(N_k'E)$. The restriction of $ST^vW$ on $D^4\times I^k\times \{0\}$ is a trivial $S^3$-bundle and it admits a tuple $\vec{v}=(v^{(1)},\ldots,v^{(3k)})$ of admissible sections that are given by opposite pairs of constant sections. By Lemma~\ref{lem:propagator-unique} (2), there exists an admissible section $\eta^{(i)}$ on $ST^vW$ that extends both $N_k'\theta^{(i)}|_{ST^v(N_k'E)}$ and $v^{(i)}$ on the endpoints, where the $\Z_2$-symmetry may be assumed by taking $c\mapsto \frac{c+\tau\cdot c}{2}$ for nontrivial $\tau\in\Z_2$, for each $i$. We obtain $\vec{\eta}=(\eta^{(1)},\ldots,\eta^{(3k)})$. 

\begin{Def}\label{def:adapted}
We call the pair $(\vec{\theta},\vec{\eta})$ that can be obtained from a tuple $\vec{\theta}$ of admissible propagators and fixed $\vec{v}$ as above an {\it adapted} pair. In that case, we also say that $\vec{\eta}$ is adapted to $\vec{\theta}$.
\end{Def}

We consider the normalized configuration space
\begin{equation}\label{eq:S_n}
 S_n(\R^4)=\{(y_1,\ldots,y_n)\in (\R^4)^n\mid y_1=0,\sum_{\ell=2}^n\|y_\ell\|^2=1,y_i\neq y_j\mbox{ if }i\neq j\}. 
\end{equation}
Intuitively, this can be seen as the configuration space in the ``microscopic world'' at the limit of collapsing of $n$ points into one point. Let $\overline{S}_n(\R^4)$ be the closure of $S_n(\R^4)$ in $\bConf_n(S^4,\infty)$. In $\partial\bConf_n(S^4,\infty)$, the face obtained by blowing up along the diagonal $\{(y_1,\ldots,y_n)\in (S^4-\{\infty\})^n\mid y_1=\cdots=y_n\}$ ({\it anomalous face}), where all the $n$ points collapse into a point of $S^4-\{\infty\}$, is naturally diffeomorphic to $\overline{S}_n(\R^4)\times B\ell(S^4,\infty)$. There is an action of $SO_4$ on $\overline{S}_n(\R^4)$ by extending the diagonal action $g\cdot (y_1,\ldots,y_n)=(g\cdot y_1,\ldots,g\cdot y_n)$ on $S_n(\R^4)$. Let $\overline{S}_n(T^vW)$ denote the $\overline{S}_n(\R^4)$-bundle associated to the rank 4 vector bundle $T^vW$. The face of $\EC_n(\pi)$ obtained by blowing up along $\{\infty\}^j\times (S^4)^{n-j}$ in fiber ({\it infinite face}) is naturally diffeomorphic to $\overline{S}_{j+1}(\R^4)\times \EC_{n-j}(\pi)$. 

Note that $\eta^{(i)}$ is an admissible section of $\overline{S}_2(T^vW)=ST^vW$. By replacing $\theta^{(i)}$ with $\eta^{(i)}$ and $\EC_n(\pi)$ with $\overline{S}_n(T^vW)$ in the definition of $I^\adm(\Gamma)$ in (\ref{eq:I^adm}), we define an analogue of $I^\adm(\Gamma)$ in the microscopic configuration space. Namely, given a trivalent graph $\Gamma\in\calG_k^0$, let $\varphi_i:\overline{S}_{2k}(T^vW)\to \overline{S}_2(T^vW)=ST^vW$ be the map which gives the endpoints of the $j$-th edge. Then $\varphi_i^{-1}\eta^{(i)}$ gives a codimension 3 chain of $\overline{S}_{2k}(T^vW)$. Put
\begin{equation}\label{eq:J^adm}
 J^\adm(\Gamma)=\langle\varphi_1^{-1}\eta^{(1)},\varphi_2^{-1}\eta^{(2)},\cdots, \varphi_{3k}^{-1}\eta^{(3k)}\rangle_{\overline{S}_{2k}(T^vW)}.
\end{equation}
The intersection of the right hand side is generically of codimension $9k$, which agrees with the dimension of $\overline{S}_{2k}(T^vW)$. If the extension $\vec{\eta}=(\eta^{(1)},\eta^{(2)},\ldots,\eta^{(3k)})$ is chosen generically, the intersection in (\ref{eq:J^adm}) is strata transversal. Choosing such an $\vec{\eta}$, we define
\[ \begin{split}
  \hat{Z}_k^\adm(\vec{\theta},\vec{\eta})&=Z_k^\adm(\vec{\theta})-\alpha_k^\adm(\vec{\eta})\in\calA_k,\\
  \alpha_k^\adm(\vec{\eta})&=\frac{1}{N_k'2^{3k}(2k)!(3k)!}\sum_{\Gamma\in\calG_k^0} J^\adm(\Gamma)[\Gamma].
\end{split} \]
The following theorem is a modification of the fact that Kontsevich's characteristic classes are invariants for framed bundles (\cite{Kon}), and the proof is essentially the same. 
\begin{Thm}\label{thm:Z-adm-inv}
$\hat{Z}_k^\adm$ is independent of the choice of the adapted pair $(\vec{\theta},\vec{\eta})$ and defines a homomorphism
$\hat{Z}_k^\adm:\pi_{k}B\Diff(D^4,\partial)\to \calA_k$.
\end{Thm}
We shall prove Theorem~\ref{thm:Z-adm-inv} just below. The main Theorem~\ref{thm:main} can be restated in terms of $\hat{Z}_k^\adm$ as follows. 

\begin{Thm}\label{thm:commute}
Let $\calG_k'$ be the subspace of $\calG_k$ spanned by graphs without self-loops or multiple edges. There exists a linear map $\Psi_k:\calG_k'\to \pi_kB\Diff(D^4,\partial)\otimes\Q$ such that the composition $\hat{Z}_k^\adm\circ \Psi_k:\calG_k'\to \calA_k$ 
%\[ \calG_k \stackrel{\Psi_k}{\longrightarrow}  \pi_kB\Diff(D^4,\partial)\otimes\Q \stackrel{\hat{Z}_k^\adm}{\longrightarrow} \calA_k\]
agrees with the projection.
\end{Thm}
Note that the projection $\calG_k'\to \calA_k$ is surjective. Theorem~\ref{thm:commute} will be proved in \S\ref{s:cycles} and \S\ref{s:computation}. More precise statement is given in Theorem~\ref{thm:Z(G)}.

\begin{Cor} 
$\dim \pi_kB\Diff(D^4,\partial)\otimes\Q \geq \dim \calA_k$.
\end{Cor}

\subsection{Proof of Theorem~\ref{thm:Z-adm-inv}}
We take two adapted pairs $(\vec{\theta}_0,\vec{\eta}_0)$, $(\vec{\theta}_1,\vec{\eta}_1)$. Since $\theta_0^{(i)}$ and $\theta_1^{(i)}$ are relatively homologous by Lemma~\ref{lem:propagator-unique} (1), one can find a tuple $\vec{\omega}=(\omega^{(1)},\ldots,\omega^{(3k)})$ of admissible propagators of the $\bConf_2(S^4,\infty)$-bundle $\EC_{2}(\pi)\times I$ over $B\times I$ such that $\omega^{(i)}$ extends given ones $\theta_0^{(i)},\theta_1^{(i)}$ on $\EC_2(\pi)\times \{0\}, \EC_2(\pi)\times \{1\}$ respectively. If we let $\vec{\eta}_{\widetilde{E}}$ be the tuple of restriction of the iteration $N_k'\vec{\omega}$ on $ST^v(N_k'E\times I)\subset \partial^\fib\EC_2(N_k'\pi)\times I$, then this is a symmetric chain that extends $N_k'\vec{\theta}_0|_{ST^v(N_k'E)\times\{0\}}$ and $N_k'\vec{\theta}_1|_{ST^v(N_k'E)\times\{1\}}$, and agrees with the restrictions of $\vec{\eta}_0$ and $\vec{\eta}_1$ on the endpoints. Then the following identities hold.
\begin{eqnarray}
 & Z_k^\adm(\vec{\theta}_0)-Z_k^\adm(\vec{\theta}_1)=\alpha_k^\adm(\vec{\eta}_{\widetilde{E}}) \label{eq:Z-Z}\\
 & \alpha_k^\adm(\vec{\eta}_0)-\alpha_k^\adm(\vec{\eta}_1)=\alpha_k^\adm(\vec{\eta}_{\widetilde{E}}) \label{eq:a-a}
\end{eqnarray}
The invariance of $\hat{Z}_k^\adm$ follows immediately from these identities.
\par\medskip

\noindent{\it Proof of (\ref{eq:Z-Z}):}  Similarly to the definition of (\ref{eq:I^adm}), let $\phi_i:\EC_{2k}(\pi)\times I\to \EC_2(\pi)\times I$ be the map which gives the endpoints of the $j$-th edge. Choose $\vec{\omega}$ so that the piecewise intersection $I^\adm(\Gamma;\vec{\omega})=\langle \phi_1^{-1}\omega^{(1)},\ldots,\phi_{3k}^{-1}\omega^{(3k)}\rangle_{\EC_{2k}(\pi)\times I}$ is strata transversal, and put 
\[ Z_k^\adm(\vec{\omega})=\frac{1}{2^{3k}(2k)!(3k)!}\sum_{\Gamma\in\calG_k^0}I^\adm(\Gamma;\vec{\omega})[\Gamma]. \]
Since $\partial I^\adm(\Gamma;\vec{\omega})$ is a 0-boundary, we have $\# \partial I^\adm(\Gamma;\vec{\omega})=0$ and hence $\#\partial Z_k^\adm(\vec{\omega})=0$. The boundary of $I^\adm(\Gamma;\vec{\omega})$ comes from the strata obtained by replacing at least one $\omega^{(i)}$ with $\partial \omega^{(i)}$. The boundary of $\omega^{(i)}$ can be described as follows.
\[ \partial \omega^{(i)}=\theta_1^{(i)}-\theta_0^{(i)}+\partial^\fib\omega^{(i)} \]
Here, we consider $\theta_0^{(i)}$ as a chain of $\EC_2(\pi)\times \{0\}$, $\theta_1^{(i)}$ as a chain of $\EC_2(\pi)\times\{1\}$, and $\partial^\fib\omega^{(i)}$ is a chain of $\partial^\fib\EC_2(\pi)\times I$.

We shall prove $\#\partial Z_k^\adm(\vec{\omega})=Z_k^\adm(\vec{\theta}_1)-Z_k^\adm(\vec{\theta}_0)+\alpha_k^\adm(\vec{\eta}_{\widetilde{E}})$. In $\#\partial Z_k^\adm(\vec{\omega})$, the sum of terms of boundary points of $\partial I^\adm(\Gamma;\vec{\omega})$ from $\theta_1^{(i)}-\theta_0^{(i)}$ is $Z_k^\adm(\vec{\theta}_1)-Z_k^\adm(\vec{\theta}_0)$. The boundary points of $\partial I^\adm(\Gamma;\vec{\omega})$ from $\partial^\fib\omega^{(i)}$ is on a stratum $S_A$ of $\partial^\fib\EC_{2k}(\pi)\times I$ that corresponds to the collapse of a full subgraph $\Gamma_A$ of $\Gamma$ on a vertex set $A\subset \{1,2,\ldots,2k\}$, $|A|\geq 2$, into a point on a fiber. The sum of terms of boundary points on $S_A$ with $|A|=2$ is proved to be cancelled each other by the IHX relation, as in \cite{KT, Les1}. The sum of terms of boundary points on $S_A$ with $3\leq |A|<2k$ is proved to vanish: When $\Gamma_A$ has a univalent vertex, the intersection must be empty by a dimensional reason (\cite[Lemma~A.8]{CCL}). When $\Gamma_A$ has a bivalent vertex, the sum of the terms of the boundary points vanishes by a symmetry of $S_A$ (\cite[Lemma~2.1]{Kon}, \cite[Lemma~A.9]{CCL}), which is allowed by the symmetry condition for an admissible section. The sum of terms of boundary points on $S_A$ with $|A|=2k$ (anomalous face) is $\alpha_k^\adm(\vec{\eta}_{\widetilde{E}})$. Further, we need to check that the boundary points on an infinite face $S_A$ for $|A|=j$, where $\infty\in A\subset \{1,2,\ldots,2k,\infty\}$, does not contribute to $\#\partial Z_k^\adm(\vec{\omega})$. By the direct product structure $\overline{S}_{j+1}(\R^4)\times \EC_{2k-j}(\pi)$ of $S_A$, it follows that the graphs for $S_A$ can be counted as the product of the numbers of graphs in the two factors. However, the count of graphs in $\overline{S}_{j+1}(\R^4)$ is zero since $\dim\overline{S}_{j+1}(\R^4)=4j-1$ is less than the codimension of the space of graphs in $\overline{S}_{j+1}(\R^4)$ that is $3\cdot\frac{3j+m}{2}=\frac{9j+3m}{2}$, where $m$ is the number of edges in $\Gamma$ that connect vertices in $A$ and those outside $A$.

From $\#\partial Z_k^\adm(\vec{\omega})=0$, the identity (\ref{eq:Z-Z}) follows.
\par\medskip

\noindent{\it Proof of (\ref{eq:a-a}):} Choose a tuple $\vec{\eta}_{W\times I}=(\eta_{W\times I}^{(1)},\ldots,\eta_{W\times I}^{(3k)})$ of admissible sections of $ST^v(W\times I)$ which extends $\vec{\eta}_0,\vec{\eta}_1,\vec{\eta}_{\widetilde{E}}$ and the tuple $\vec{v}$ of constant sections on $(D^4\times I^k\times \{0\})\times I\subset W\times I$. We choose $\vec{\eta}_{W\times I}$ generically so that the intersection $J^\adm(\Gamma;\vec{\eta}_{W\times I})=\langle \varphi_1^{-1}\eta_{W\times I}^{(1)},\ldots, \varphi_{3k}^{-1}\eta_{W\times I}^{(3k)}\rangle$ in the associated $\overline{S}_{2k}(\R^4)$-bundle $\overline{S}_{2k}(W\times I)$ to $ST^v(W\times I)$ is strata transversal. Then $J^\adm(\Gamma;\vec{\eta}_{W\times I})$ is a 1-chain and we put
\[  \alpha_k^\adm(\vec{\eta}_{W\times I})=\frac{1}{N_k' 2^{3k}(2k)!(3k)!}\sum_{\Gamma\in\calG_k^0}J^\adm(\Gamma;\vec{\eta}_{W\times I})[\Gamma].
\]
We have $\#\partial\alpha_k^\adm(\vec{\eta}_{W\times I})=0$. 

In $\#\partial\alpha_k^\adm(\vec{\eta}_{W\times I})$, the sum of terms of boundary points from the configuration space $\overline{S}_{2k}(\R^4)$ is proved to vanish by the IHX relation, a dimensional reason, and symmetry, as above. Note that there is no face in $\overline{S}_{2k}(\R^4)$ corresponding to the bifurcation of collapse of all the $2k$ points into a point, by the norm square sum condition (\ref{eq:S_n}). What remains is the contribution of boundary points on $\partial (W\times I)$, which gives $\alpha_k^\adm(\vec{\eta}_1)-\alpha_k^\adm(\vec{\eta}_0)$ on 
$W\times \{0,1\}$ and $\alpha_k^\adm(\vec{\eta}_{\widetilde{E}})$ on $N_k'E\times I$. On $(D^4\times I^k\times \{0\})\times I$, the restriction of $\vec{\eta}_{W\times I}$ is a pair of constant sections, the value of $\alpha_k^\adm$ is zero by a dimensional reason. Thus we have
\[ \#\partial\alpha_k^\adm(\vec{\eta}_{W\times I})=\alpha_k^\adm(\vec{\eta}_1)-\alpha_k^\adm(\vec{\eta}_0)+\alpha_k^\adm(\vec{\eta}_{\widetilde{E}}). \]
This together with $\#\partial\alpha_k^\adm(\vec{\eta}_{W\times I})=0$ implies (\ref{eq:a-a}). 

Proof of the homotopy invariance is the same as above. 
\qed

%%%%%%%%%%%%%%%%%%%%%%%%%%%%%%%
\subsection{Bundle bordism invariance of $\hat{Z}_k^\adm$ on the image from $\pi_k$}\label{ss:bordism-invariance}

We consider a bordism invariant defined on the image $\mathrm{Im}\,H$ of the natural homomorphism $H:\pi_{k}B\Diff(D^4,\partial)\to \Omega_{k}(B\Diff(D^4,\partial))$. There is a natural isomorphism 
\[ \Omega_k(B\Diff(D^4,\partial))\otimes \Q\cong \bigoplus_{p+q=k} H_p(B\Diff(D^4,\partial);\Q)\otimes\Omega_q \]
(see \cite{CF}). By the Milnor--Moore theorem (\cite[Appendix]{MM}) on path-connected homotopy associative $H$-spaces, $H_*(B\Diff(D^4,\partial);\Q)$ is the commutative polynomial algebra with primitive part $\pi_*B\Diff(D^4,\partial)\otimes\Q$. Thus $H$ is injective over $\Q$, and the homomorphism
\[ \hat{Z}_k^\adm\circ H^{-1}:\mathrm{Im}\,H\otimes \Q\to \calA_k \]
is defined. We shall give below a direct definition of this homomorphism from a $(D^4,\partial)$-bundle whose base is not necessarily a sphere. Such a definition will be needed because the $(D^4,\partial)$-bundles constructed later by surgery are not over $S^k$ (Proposition~\ref{prop:primitive}). 

Let $\pi:E\to B$ be a $(D^4,\partial)$-bundle over a closed oriented $k$-manifold $B$ that is bundle bordant to a pointed $(D^4,\partial)$-bundle $\pi_0:E_0\to S^{k}$.

\subsubsection{\bf Invariant for $\pi_0:E_0\to S^k$:}
By closing $I^k$ suitably, we may take the $N_k'$-fold disjoint iteration $N_k'\pi_0:N_k'E_0\to S^{k}$ and an adapted pair $(\vec{\theta}_0,\vec{\eta}_0)$. Then $\hat{Z}_k^\adm(\vec{\theta}_0,\vec{\eta}_0)$ is defined as follows. By Lemma~\ref{lem:NE-trivial}, $N_k'E_0$ is relatively diffeomorphic to $D^4\times S^{k}$, and one can find a compact manifold $W_0$ with corners having $N_k'E_0 \cup (S^3\times S^{k}\times I)\cup (-D^4\times S^{k})$ as the boundary, which satisfies the following.
\begin{itemize}
\item $W_0$ is relatively diffeomorphic to $D^4\times S^{k}\times I$.
\item There is a trivial rank 4 vector bundle $T^vW_0\to W_0$ that extends the trivial rank 4 vector bundles $T^v(N_k'E_0)$ and $TD^4\times S^{k}$ on the endpoints of $I$.
\end{itemize}
Choosing an adapted pair $(\vec{\theta}_0,\vec{\eta}_0)$ for $W_0$, which is defined similarly as Definition~\ref{def:adapted}, $\hat{Z}_k^\adm(\vec{\theta}_0,\vec{\eta}_0)$ is defined.

\subsubsection{\bf Invariant for $\pi:E\to B$:}
We consider the $N_k'$-fold iteration of $\pi$:
\[ N_k'\pi:N_k'E\to N_k'B=B\tcoprod\cdots\tcoprod B\quad\mbox{($N_k'$ times)}.\]
Let $\widetilde{\pi}:\widetilde{E}\to \widetilde{B}$ be a $(D^4,\partial)$-bundle bordism between $N_k'\pi$ and $N_k'\pi_0$, which exists by the choice of $\pi$.

We take a tuple $\vec{\theta}=(\theta^{(1)},\theta^{(2)},\ldots,\theta^{(3k)})$ of admissible propagators of $\EC_2(\pi)$ and a tuple $\vec{\eta}_{\widetilde{E}}$ of generic admissible sections of $T^v\widetilde{E}$ that agrees with $\vec{\eta}_0$ on $ST^v(N_k'E_0)$ and with the restriction of $N_k'\vec{\theta}$ on $ST^v(N_k'E)$. By using this, we define
\[ \hat{Z}_k^\adm(\vec{\theta},\vec{\eta}_{\widetilde{E}}\cup \vec{\eta}_0)=Z_k^\adm(\vec{\theta})-\alpha_k^\adm(\vec{\eta}_{\widetilde{E}})-\alpha_k^\adm(\vec{\eta}_0). \]

\begin{Prop}\label{prop:bordism-inv}
For generic $\vec{\theta},\vec{\eta}_{\widetilde{E}}$, the identity
\begin{equation}\label{eq:bordism-inv}
 \hat{Z}_k^\adm(\vec{\theta},\vec{\eta}_{\widetilde{E}}\cup \vec{\eta}_0)
=\hat{Z}_k^\adm(\vec{\theta}_0,\vec{\eta}_0) 
\end{equation}
holds. Hence the left hand side does not depend on the choices of $\widetilde{\pi},\vec{\theta},\vec{\eta}_{\widetilde{E}}$ and defines a homomorphism $\mathrm{Im}\,H\to \calA_k$.
\end{Prop}
\begin{proof}
Similarly to (\ref{eq:Z-Z}) in the proof of Theorem~\ref{thm:Z-adm-inv}, the identity $Z_k^\adm(\vec{\theta})-Z_k^\adm(\vec{\theta}_0)-\alpha_k^\adm(\vec{\eta}_{\widetilde{E}})=0$ holds. This together with the two definitions of $\hat{Z}_k^\adm$ on both sides proves the identity (\ref{eq:bordism-inv}).
\end{proof}

%%%%%%%%%%%%%%%%%%%%%%%%%%%%%%%
\subsection{Change of cobordism}\label{ss:ch-cobordism}

In the definition of the homomorphism $\hat{Z}_k^\adm:\mathrm{Im}\,H\to \calA_k$, we took a $(k+5)$-cobordism $W$ with corners of special type. Here, we shall see that the same homomorphism can be defined by replacing the cobordism with more general one. This fact will be used later in computing the value of the invariant where we change $W$ into one that is adapted to surgery (Lemma~\ref{lem:occupied}). 

\begin{Assum}\label{hyp:W}
We assume that a $(D^4,\partial)$-bundle $\pi:E\to B$ over a closed oriented $k$-manifold $B$ satisfies the following.
\begin{enumerate}
\item The isomorphism class of $\pi$ corresponds to an element of $\mathrm{Im}\,H$. In particular, $B$ is oriented cobordant to $S^k$. Let $\widetilde{B}$ be an oriented cobordism between them.
\item There is a compact oriented manifold $W$ with corners bounded by $O_E=E\cup (\partial D^4\times \widetilde{B})\cup -(D^4\times S^k)$. Here, $W$ need not be the total space of a $D^4$-bundle over $\widetilde{B}$. Let $\widetilde{\pi}:O_E\to \widetilde{B}$ be the projection that is the gluing of the natural ones $\partial D^4\times \widetilde{B}\to \widetilde{B}$, $\pi:E\to B$, and $D^4\times S^k\to S^k$.
\item $T^vE$ has a vertical framing (trivialization) $\tau_E:T^vE\to \R^4\times E$, that is standard near $\partial^\fib E$.
\end{enumerate}
\end{Assum}
Under this assumption, we may construct a trivial bundle $\ve^4(W)\to W$, by identifying the sides (over $\partial \widetilde{B}=B\tcoprod (-S^k)$) of the trivial bundle $\R^4\times W\to W$ with $T^vE$ and $T^v(D^4\times S^k)=TD^4\times S^k$ by using $\tau_E$ and the standard framing respectively. Let $\vec{\theta}$ be a tuple of admissible propagators of $\EC_2(\pi)$ and we take a tuple $\vec{\eta}_W$ of generic admissible sections of $\ve^4(W)$ that is adapted to $\vec{\theta}$ without iteration with respect to $\vec{v}$ fixed above. With this data, $Z_k^\adm(\vec{\theta})$ and $\alpha_k^\adm(\vec{\eta}_W)$ are defined. Here, $\vec{\eta}_W$ depends on the choice of $\ve^4(W)$, which depends on the choice of $\tau_E$.

\begin{Lem}\label{lem:a-inv-eta}
For fixed $W,\tau_E$, the value of $\alpha_k^\adm(\vec{\eta}_W)$ does not depend on the choice of $\vec{\eta}_W$ that is adapted to $\vec{\theta}$.
\end{Lem}
Proof of Lemma~\ref{lem:a-inv-eta} is the same as the proof of Theorem~\ref{thm:Z-adm-inv}.

\begin{Lem}\label{lem:a-inv-cob}
For a fixed $\tau_E$, the value of $\alpha_k^\adm(\vec{\eta}_W)$ is invariant under a relative cobordism of $W$. Namely, if we replace $W$ with another compact oriented manifold $W'$ that is related to $W$ by a $(k+6)$-cobordism $V$ with corners such that $\partial V=W\cup (\partial W\times I)\cup (-W')$, then we have $\alpha_k^\adm(\vec{\eta}_{W'})=\alpha_k^\adm(\vec{\eta}_W)$ for any choice of $\vec{\eta}_{W'}$.
\end{Lem}
\begin{proof}
 By gluing $\R^4\times V$ to $T^v(E\times I)=T^v(\partial W\times I)$ by using the trivialization $\tau_E\times \mathrm{id}_I$, we obtain a trivial $\R^4$-bundle $\ve^4(V)\to V$. The trivial 1-parameter family of the restriction of $\vec{\theta}$ on $ST^vE$ gives a tuple of admissible sections of the unit $S^3$-bundle $ST^v(E\times I)$. By extending this and $\vec{\eta}_W$, we take a tuple $\vec{\eta}_V$ of generic admissible sections of the unit sphere bundle $S\ve^4(V)$. Let $\vec{\eta}_{W'}$ be the restriction of $\vec{\eta}_V$ on $W'$. As in the proof of Theorem~\ref{thm:Z-adm-inv}, the 1-chain $\alpha_k^\adm(\vec{\eta}_V)$ is defined and it follows from the identity $\#\partial\alpha_k^\adm(\vec{\eta}_V)=0$ that $\alpha_k^\adm(\vec{\eta}_W)-\alpha_k^\adm(\vec{\eta}_{W'})=0$. This proves that $\alpha_k^\adm$ is invariant under a relative cobordism of $W$. Moreover, by Lemma~\ref{lem:a-inv-eta}, $\vec{\eta}_{W'}$ can be altered without changing the value of $\alpha_k^\adm$.
\end{proof}
In order to make $\alpha_k^\adm$ invariant under changes of $\tau_E$ and $W$ (to one not necessarily relative cobordant), we add more terms. $\alpha_k^\adm$ can be defined for a general oriented closed $(k+5)$-manifold $X$ and a trivial rank 4 vector bundle $\ve^4(X)$ over it by using admissible sections. By the same argument as the proof of Lemma~\ref{lem:a-inv-cob}, it induces a homomorphism $\alpha_k^\adm:\Omega_{k+5}\to \calA_k$. If $k+5\not\equiv 0$ (mod 4), then $\Omega_{k+5}\otimes \Q=0$ by Thom's results on the cobordism groups (\cite{T}) and hence $\alpha_k^\adm=0$ on $\Omega_{k+5}$. 

From now on, we assume $k+5\equiv 0$ (mod 4)\footnote{For only a disproof of the 4-dimensional Smale conjecture, this case is not necessary.}. In this case, the rational cobordism invariants are given by homogeneous polynomials of degree $k+5$ in the Pontrjagin classes, by Thom's result again. Namely, there exists a universal homogeneous polynomial $P_k(p_1,\ldots,p_{\frac{k+5}{4}})$ of degree $k+5$ in $H^*(BSO_{k+5};\Q)\otimes\calA_k=\Q[p_1,\ldots,p_{\frac{k+3}{2}},e]\otimes\calA_k$, $|p_j|=4j$, $|e|=\frac{k+5}{2}$, such that 
\begin{equation}\label{eq:a-P}
 \alpha_k^\adm(X)=\langle P_k(p_1(TX),\ldots,p_{\frac{k+5}{4}}(TX)),[X]\rangle.
\end{equation}

We shall extend the definition of the right hand side of (\ref{eq:a-P}) for the manifold with corners $W$ by an integral of a pullback of a characteristic form on the classifying space. Since $TW|_{\partial W}=T^vO_E$ is the Whitney sum of vertical and horizontal subbundles, we may assume that the classifying map for $TW|_{\partial W}$ factors through the classifying space $BSO_4\times BSO_{k+1}$ for Whitney sums. We shall deform the characteristic form on $W$ according to this boundary condition. We consider the natural map $\sigma: BSO_4\times BSO_{k+1}\to BSO_{k+5}$. The cohomology ring of $BSO_4\times BSO_{k+1}$ is $H^*(BSO_4\times BSO_{k+1};\Q)= \Q[p_1',e']\otimes \Q[p_1'',\ldots,p_{\frac{k-1}{2}}'',e'']$,
with $|p_1'|=4$, $|e'|=4$, $|p_j''|=4j$, $|e''|=\frac{k+1}{2}$. Then by the Whitney sum formula for the Pontrjagin class, we have $\sigma^* p_n = \sum_{i+j=n} p_i'p_j''=p_n'' + p_1'p_{n-1}''$. 
Let $P_k'(p_1',p_1'',p_2'',\ldots, p_m'')$, $m=\frac{k+5}{4}$, be the polynomial in $H^*(BSO_4\times BSO_{k+1};\Q)\otimes\calA_k$ given by
\begin{equation}\label{eq:P'}
 \begin{split}
  P_k'(p_1',p_1'',p_2'',\ldots,p_m'')&=P_k(\sigma^* p_1,\sigma^* p_2,\ldots, \sigma^* p_m)\\
  &=P_k(p_1''+p_1', p_2''+p_1'p_1'',\ldots,p_m''+p_1'p_{m-1}''),
\end{split}
\end{equation}
where $P_k$ is the polynomial defined in (\ref{eq:a-P}).

In the following we regard $BSO_n$ as the oriented Grassmannian $\widetilde{G}_n(\R^\infty)$. Let $x_1,x_2,\ldots\in\Omega^*(BSO_{k+5})$ be closed forms on $BSO_{k+5}$ such that $[x_j]=p_j$ for each $j$. Let $u_1\in\Omega^4(BSO_4)$,  $v_1,v_2,\ldots\in\Omega^*(BSO_{k+1})$ be closed forms such that $[u_1]=p_1'$, $[v_j]=p_j''$ for each $j$. By (\ref{eq:P'}), there is a form $\eta\in\Omega^*(BSO_4\times BSO_{k+1})\otimes \calA_k$ such that the form level identity $P_k'(u_1,v_1,v_2,\ldots,v_m) = \sigma^* P_k(x_1,x_2,\ldots,x_m) + d\eta$ holds. Thus $P_k'(u_1,v_1,v_2,\ldots,v_m)$ and $\sigma^* P_k(x_1,x_2,\ldots,x_m)$ can be smoothly connected by an exact form. Namely, let $\rho:BSO_4\times BSO_{k+1}\times [0,\ve)\to BSO_4\times BSO_{k+1}$ be the projection. Let $\psi: [0,\ve)\to [0,1]$ be a smooth function that is 0 near $0\in [0,\ve)$ and 1 near $\ve$. Then the form
\[ P_k''=\rho^*\sigma^* P_k(x_1,\ldots,x_m) + d(\psi \rho^*\eta) \]
on $BSO_4\times BSO_{k+1}\times [0,\ve)$ agrees with $\rho^*P_k'(u_1,v_1,\ldots,v_m)$ near 0 and with $\rho^*\sigma^*P_k(x_1,\ldots,x_m)$ near $\ve$. 

In order to get a characteristic form on $W$, we shall take a classifying map $W\to BSO_{k+5}$ for $TW$ that factors through $BSO_4\times BSO_{k+1}$ on $\partial W=O_E$. By the decomposition $TW|_{E}=TE\oplus \gamma^1=T^vE\oplus \pi^* TB\oplus\gamma^1$, where $\gamma^1$ denotes a rank 1 trivial real line bundle, and $TW|_{\partial D^4\times\widetilde{B}}=TD^4|_{\partial D^4}\oplus \widetilde{\pi}^*\widetilde{B}$, 
there are classifying maps $\varphi_\partial:\partial W\to BSO_4\times BSO_{k+1}$ and $\varphi:W\to BSO_{k+5}$ for $TW|_{\partial}$ and $TW$ respectively such that $\varphi|_{\partial W}=\sigma\circ \varphi_\partial$. Further, we assume that the restriction of $\varphi$ on a collar neighborhood $\partial W\times [0,\ve)$ factors through $BSO_4\times BSO_{k+1}$. More explicitly, let $\widetilde\varphi_\partial:\partial W\times [0,\ve)\to BSO_4\times BSO_{k+1}\times [0,\ve)$ be the map defined by $\widetilde{\varphi}_\partial(x,t)=(\varphi(x,t),t)$ and we assume $\varphi=\sigma\circ \rho\circ\widetilde{\varphi}_\partial$ on $\partial W\times [0,\ve)$. 
Since $T^v\partial W$ has a trivialization that extends to that of $T^v\partial W\times [0,\ve)$, we assume further that $\widetilde{\varphi}_\partial$ is a map to $\{*\}\times BSO_{k+1}\times [0,\ve)$, $\varphi_\partial$ is a map to $\{*\}\times BSO_{k+1}$, and that $\varphi|_{-(D^4\times S^k)}$ is a map to $\{*\}\subset BSO_{k+5}$. 

We consider the form 
\[ \widetilde{\varphi}_\partial^* P_k'' = \varphi^* P_k(x_1,\ldots,x_m) + d(\widetilde{\varphi}_\partial^*\psi \rho^*\eta) \]
on $E\times [0,\ve)$. Since this agrees with $\widetilde{\varphi}_\partial^* \rho^* P_k'(u_1,v_1,\ldots,v_m)$ near 0 and with $\varphi^* P_k(x_1,\ldots,x_m)$ near $\ve$, $\widetilde{\varphi}_\partial^* P_k''$ can be extended by $\varphi^* P_k(x_1,\ldots,x_m)$ to the whole of $W$. We denote the resulting closed form on $W$ by $P_k(TW)$ and define
\[ P_k(TW;\tau_E)=\int_W P_k(TW). \]
\begin{Lem}\label{lem:a-inv-tau-W}
Suppose $k+5\equiv 0$ (mod 4).
\begin{enumerate}
\item $P_k(W;\tau_E)$ does not depend on the choices of $\varphi$, $\varphi_\partial$, $x_1,\ldots,x_m,u_1,v_1,\ldots,v_m$.
\item $P_k(W;\tau_E)$ does not depend on $\tau_E$. Namely, if $\tau_E'$ is another choice, then $P_k(W;\tau_E)=P_k(W;\tau_E')$.
\item $\alpha_k^\adm(\vec{\eta}_W)-P_k(W;\tau_E)$ does not depend on $W$ (and on $\tau_E$ and $\vec{\eta}_W$ that is adapted to $\vec{\theta}$).
\end{enumerate}
\end{Lem}
\begin{proof}
(1) A change of $\varphi$ by a relative homotopy that fixes a collar neighborhood $\partial W\times [0,\ve)$ does not affect the integral $P_k(W;\tau_E)$ since the difference is given by an integral of a compact support exact form, which vanishes. By the definition of $P_k(TW)$, a change of $x_1,\ldots,x_m$ may change $P_k(TW)$ by a compact support exact form and does not affect the result, either. Next, we consider a change of $\varphi_\partial$ (and thus of its extension $\varphi$) by a homotopy in $BSO_4\times BSO_{k+1}$, which may change $P_k(W;\tau_E)$ by an integral over $\partial W\times I$. The homotopy gives a classifying map $\partial W\times I\to BSO_4\times BSO_{k+1}$ which can be factored as $\partial W\times I\to B\times I\to \{*\}\times BSO_{k+1}$ thanks to the presence of the framing. Hence the integral can be written as that of a $(k+5)$-form over the $(k+1)$-manifold $B\times I$, which vanishes. The effect of changing $u_1,v_1,\ldots,v_m$ is similar to this, where the homotopy is replaced by an exact form on $B\times I$. 

    (2) A change of $\tau_E$ corresponds to gluing a map $E\times I\to BSO_4\times\{*\}\to BSO_{k+5}$ which maps $E\times\{0,1\}$ to the base point, to $\varphi:W\to BSO_{k+5}$ along the face $E=E\times\{0\}$. Hence $P_k(W;\tau_E')-P_k(W;\tau_E)$ can be written as an integral over $E\times I$. Namely, a change of framing corresponds to a classifying map $\varphi_E:E\times I\to BSO_4\times BSO_{k+1}$ given by $\varphi_E=(\varphi_1, \varphi_2 \pi r)$, where $r:E\times I\to E$ is the projection, $\varphi_1:(E\times I,E\times\{0,1\})\to (BSO_4,*)$, and $\varphi_2:B\to BSO_{k+1}$. Hence $P_k(W;\tau_E')-P_k(W;\tau_E)$ is a linear combination of integrals of the following form:
\[ \int_{E\times I} \varphi_1^*\alpha\wedge r^*\pi^*\varphi_2^*\beta=\int_E r_*\varphi_1^*\alpha\wedge \pi^*\varphi_2^*\beta, \]
where $\alpha\in\Q[u_1]\subset H^*(BSO_4;\Q)$ and $\beta\in\Q[v_1,\ldots,v_m]\subset H^*(BSO_{k+1};\Q)$ are monomials.  
The form $r_*\varphi_1^*\alpha$ is closed since by the Stokes formula for pushforward (e.g., \cite[VII. Problem~4]{GHV}), one has $dr_*\varphi_1^*\alpha=r_*d\varphi_1^*\alpha\pm r_*^{\partial}\varphi_1^*\alpha=0$, and it has compact support along the fiber. Let $T[\alpha]\in H^*(SO_4;\Q)\otimes\calA_k$ be the image of the Chern--Simons transgression $H^*(BSO_4,*;\Q)\to H^*(ESO_4,SO_4;\Q)\leftarrow H^{*-1}(SO_4;\Q)$ for the Pontrjagin class $[\alpha]$. 
Then there exists a continuous map $g:(E,\partial E)\to (SO_4,\mathbf{1})$ such that $[r_*\varphi_1^*\alpha]=g^*T[\alpha]$ in $H^*_c(E;\Q)$. Since $T$ annihilates decomposable monomials (\cite[Proposition~16.19]{Sw}), the integral above vanishes unless $T[\alpha]$ is a multiple of $Tp_1'=T[u_1]$ and thus is of degree 3. We consider the integral along the fiber of $\pi$ first and we have
\[ \int_E r_*\varphi_1^*\alpha\wedge \pi^*\varphi_2^*\beta=\int_B \pi_*r_*\varphi_1^*\alpha\wedge \varphi_2^*\beta=0 \]
by a dimensional reason.

(3) That $\alpha_k^\adm(\vec{\eta}_W)-P_k(W;\tau_E)$ does not depend on $W$ follows from (\ref{eq:a-P}). 
\end{proof}

\begin{Prop} Let $(\vec{\theta}_0,\vec{\eta}_0)$ be the adapted pair for $N_k'\pi_0:N_k'E_0\to S^k$ and $W_0$, that was taken in \S\ref{ss:bordism-invariance}. 
For a $(D^4,\partial)$-bundle $\pi:E\to B$ satisfying Assumption~\ref{hyp:W}, we define
\[ \hat{Z}_k^\adm(\vec{\theta},\vec{\eta}_W)=\left\{\begin{array}{ll}
Z_k^\adm(\vec{\theta})-\alpha_k^\adm(\vec{\eta}_W)+P_k(W;\tau_E) & \mbox{if $k+5\equiv 0$ (mod 4)}\\
Z_k^\adm(\vec{\theta})-\alpha_k^\adm(\vec{\eta}_W) & \mbox{otherwise}
\end{array}\right. \]
Then the following hold.
\begin{enumerate}
\item $\hat{Z}_k^\adm(\vec{\theta},\vec{\eta}_W)$ does not depend on the choices of $\tau_E, W, \vec{\eta}_W$.
\item The identity $\hat{Z}_k^\adm(\vec{\theta},\vec{\eta}_W)=\hat{Z}_k^\adm(\vec{\theta}_0,\vec{\eta}_0)$ holds.
\end{enumerate}
\end{Prop}
\begin{proof}
(1) follows from Lemmas~\ref{lem:a-inv-eta}, \ref{lem:a-inv-cob}, \ref{lem:a-inv-tau-W}. 
(2) holds since $W$ may be the total space of a bundle bordism as in \S\ref{ss:bordism-invariance}. If $k+5\equiv 0$ (mod 4), then we need to check that $P_k(W;\tau_E)=0$ for the bundle bordism $W$. This can be proved by the same manner as Lemma~\ref{lem:a-inv-tau-W} (2). 
\end{proof}

%%%%%%%%%%%%%%%%%%%%%%%%%%%%%%%
%%%%%%%%%%%%%%%%%%%%%%%%%%%%%%%
\mysection{Graph counting formula}{s:GCF}

We shall see that under some assumptions $\hat{Z}_k^\adm$ can be obtained by counting some graphs in a bundle whose edges are trajectories of Morse gradients. This makes it possible to compute the exact value of $\hat{Z}_k^\adm$ in some cases directly and geometrically by counting graphs. The main idea of the proof of the counting formula is to construct an explicit admissible propagator by using Morse trajectory spaces, as in \cite{Sh}.

%%%%%%%%%%%%%%%%%%%%%%%%%%%%%%%
\subsection{Fiberwise Morse functions}\label{ss:FMF}

Let $\pi:E\to B$ be a $(\R^4,U_\infty')$-bundle over an oriented closed Riemannian $k$-manifold $B$, and write $F_b=\pi^{-1}(\{b\})$. We also equip $E$ with a Riemannian metric. A $C^\infty$ map $f:E\to \R$ is said to be a {\it fiberwise Morse function} if its restriction $f_b=f|_{F_b}:F_b\to \R$ ($b\in B$) to each fiber is Morse. The union of critical points of $f_b$ over $b\in B$ forms a $k$-submanifold of $E$. We call its connected component a {\it critical locus}. 

Let $\xi$ be a fiberwise gradient-like vector field for $f$ along the fiber, namely, the vector field on $E$ whose restriction $\xi_b=\xi|_{F_b}$ is gradient-like for $f_b$. We call such a vector field a {\it $v$-gradient}. For a critical locus $p$ of a fiberwise Morse function $f:E\to \R$, its {\it descending manifold} and {\it ascending manifold} are defined by
\[ \begin{split}
	\calD_p(\xi)&=\{x\in E\mid \lim_{t\to -\infty}\Phi_{-\xi}^t(x)\in p\},\quad \calA_p(\xi)=\{x\in E\mid \lim_{t\to \infty}\Phi_{-\xi}^t(x)\in p\},
\end{split} \] 
where $\Phi_{-\xi}^t:E\to E$, $t\in\R$, is the flow of $-\xi$. For a pair $p,q$ of critical loci with $|p|=i$, $|q|=i+\ell$, we may assume that $\calD_p(\xi)$ and $\calA_q(\xi)$ intersect transversally in $E$ by choosing the $v$-gradient $\xi$ generically within the space of $v$-gradients. In such a case, the intersection consists of integral curves of $\xi$ between $p$ and $q$. There is a free $\R$-action on $\calD_p(\xi)\cap\calA_q(\xi)$ defined by $x\mapsto \Phi_{-\xi}^T(x)$ ($T\in \R$). We put
\[ \calM_{pq}'=\calM_{pq}'(\xi)=(\calD_p(\xi)\cap\calA_q(\xi))/\R. \]
This space is locally parametrized as the intersection of $\calD_p(\xi)\cap\calA_q(\xi)$ with a level surface of $f$, as in \S\ref{ss:morse-complex}. The dimension of the manifold $\calM_{pq}'$ is $|p|-|q|-1+\dim{B}$. We call the intersection $\calD_p(\xi)\cap\calA_q(\xi)$ an {\it $i/i+\ell$-intersection} (\cite{HW})\footnote{In \cite{HW}, an $i/j$-intersection was considered modulo change of vertical parameter on a trajectory.}. 

We take a Morse function $h:B\to \R$ on the base space $B$ such that the numbers of critical points of index $k$ and 0 are both one on each path-component of $B$. Let $\eta$ be its gradient-like vector field. We call such an $\eta$ a {\it $h$-gradient}. We take a base point $b_0\in B$ of a path-component of $B$ to be the maximal point of $h$ and put $F_0=F_{b_0}$, $f_0=f_{b_0}$. In the following, we assume the following for $f,\xi$.
\begin{Assum}\label{hyp:eta} 
\begin{enumerate}
\item The $v$-gradient $\xi$ satisfies the parametrized Morse--Smale condition. Namely, all the descending and ascending manifolds are mutually transversal in $E$.
\item The descending manifolds $\calD_p(\xi)$ for every critical loci $p$ of $f$ are fiberwise orientable, i.e., the vertical vector bundle $T^v\calD_p(\xi)$ restricted on the locus $p$ is orientable.
\item For $\ell\geq 1$, there are no $i/i+\ell$-intersections.
\item $f$ is standard outside the unit disk $D^4$, namely, it agrees with a fixed linear function outside $D^4$ in each fiber, and $\xi$ is constant there.
\end{enumerate}
\end{Assum}

\begin{Lem}\label{lem:i/i-str}
Let $\xi$ be a generic $v$-gradient satisfying Assumption~\ref{hyp:eta}. The parameters in $B$ of $i/i$-intersections equip $B$ with a conic stratification (in the sense of \emph{\cite[I.1.3]{Ce}}). Let $B^{(\ell)}$ ($\ell=0,1,2,\ldots$) denote the codimension $\ell$ stratum in the conic stratification of $B$. 
\end{Lem}
\begin{proof}[Proof (sketch)] An $i/i$-intersection can be locally described as the transversal intersection between families of submanifolds and it turns out that a single $i/i$-intersection is a codimension 1 bifurcation. Below we shall describe a concrete shape of the stratification near a codimension $r$ stratum through a basic example. 

Let $f_b:F_b\to \R$ ($b\in \R^r$) be a family of Morse functions and let $p,q_0,q_1,\ldots,q_{r-1}$ be critical points of $f_0:F_0\to \R$ of index $i$ such that $f_0(p)>f_0(q_0)>f_0(q_1)>\cdots>f_0(q_{r-1})$. We also denote by $p,q_0,q_1,\ldots,q_{r-1}$ their loci for the family $f_b$. Now we assume that $i/i$-intersections between $p$ and $q_0$, $q_0$ and $q_1$, $q_1$ and $q_2$, $\ldots$, $q_{r-2}$ and $q_{r-1}$ occur simultaneously at $0\in\R^r$. A standard model for this bifurcation can be described as follows. Take $\ve>0$ and a neighborhood $U$ of $0\in\R^r$ both small and put $L=\bigcup_{b\in U} f_b^{-1}(f(p)-\ve)$. $L$ is locally a $\R^3$-bundle $L'\to U$. $\calD_p(\xi)$ intersects each fiber $L_b'$ of $L'$ in a $(i-1)$-disk, and $\calA_{q_j}(\xi)$ intersects each fiber $L_b'$ in a $(3-i)$-disk. In the case $i=2$, we describe the model $K_p(s), M_{q_0}(0),M_{q_1}(t_1),\ldots,M_{q_{r-1}}(t_{r-1})\subset\R^3$ ($s\in\R, 0\leq t_1\leq\cdots\leq t_{r-1}$) for the disks in $L_b'$ by a local coordinate (Figure~\ref{fig:conic-stratification} (1)). 
\[ \begin{split}
  K_p(s) &= \{(s,x_2,0)\mid x_2\in\R\},\qquad M_{q_0}(0)=\{(0,0,x_3)\mid x_3\in\R\}\\
  M_{q_j}(t_j)&=\{(t_j,0,x_3)\mid x_3\in\R\}\quad (1\leq j\leq r-1).\\
\end{split}\]
For simplicity, we assume the following (Figure~\ref{fig:conic-stratification} (2)).
\begin{itemize}
\item Let $t_0=0$. Suppose that at $t_j=t_{j+1}$, the disks $M_{q_j}(t_j)$ and $M_{q_{j+1}}(t_{j+1})$ coincide, and at the same time the $i/i$-intersection between $q_j$ and $q_{j+1}$ occurs. On $t_{j+1}<t_j$, $\calA_{q_{j+1}}(\xi)$ slides under $\calA_{q_j}(\xi)$, and $M_{q_{j+1}}(t_{j+1})$ disappears from $L'$. 
\item At the moment the family of $K_p(s)$ intersects $M_{q_j}(t_j)$, the $i/i$-intersection between $p$ and $q_j$ occurs, and $\calD_p(\xi)$ slides over $\calD_{q_j}(\xi)$.
\end{itemize}
\begin{figure}
\[ \includegraphics[height=40mm]{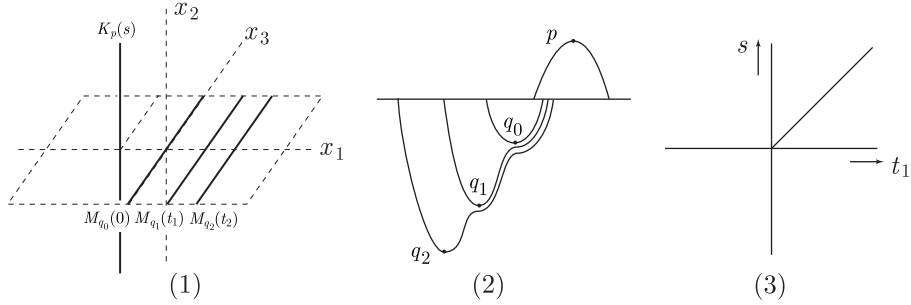} \]
\caption{Describing handle-slides by intersections in a level surface.}\label{fig:conic-stratification}
\end{figure}

We shall prove that the $i/i$-intersections in the standard model give a conic stratification of $\R^r$ by induction. Suppose that the $i/i$-intersections among $q_0,\ldots,q_{r-1}$ gives a conic stratification $A^{(\ell)}$ ($\ell=0,1,2,\ldots,r-1$) of the parameter space $\{(t_1,\ldots,t_{r-1})\mid t_1,\ldots,t_{r-1}\in\R\}=\R^{r-1}$ centered at the origin. We consider the product stratification $\widehat{A}^{(\ell)}=A^{(\ell)}\times \R$ and add to this the strata of $i/i$-intersections including $p$. We see that the result is again a conic stratification of $\R^r$. 

Let $S_{pq_j}$ denote the set of parameters of $i/i$-intersections between $p$ and $q_j$. Then 
$S_{pq_0}=\{(s,t_1,\ldots,t_{r-1})\mid s=0\}$, 
$S_{pq_j}=\{(s,t_1,\ldots,t_{r-1})\mid s=t_j, \,\,0\leq t_1\leq\cdots\leq t_j \}$ $(1\leq j\leq r-1)$.
It is easy to see that these sets form conic strata whose boundaries are on the strata of $\widehat{A}^{(\ell)}$ of codimension $\geq 1$. Thus the addition of these sets gives another conic stratification. For example, when $r=2$, the sets $S_{q_0q_1}=\{(s,t_1)\mid t_1=0\}, S_{pq_0}=\{(s,t_1)\mid s=0\}, S_{pq_1}=\{(s,t_1)\mid s=t_1, t_1\geq 0\}$ are shown in Figure~\ref{fig:conic-stratification} (3). The cases where $\calA_{q_0}(\xi),\ldots,\calA_{q_{r-1}}(\xi)$ are located in a different way is essentially the same as the example given here. 

Finally, we remark that the bundle projection $\pi$ induces immersion of $\calM_{pq}'(\xi)$ into $B$ since $\mathrm{Ker}\,d\pi\cap T(\calD_p(\xi)\cap\calA_q(\xi))$ agrees with the line subbundle of $T(\calD_p(\xi)\cap\calA_q(\xi))$ generated by the differential of the free $\R$-action, which is mapped to zero in $T\calM_{pq}'(\xi)$.
\end{proof}

%%%%%%%%%%%%%%%%%%%%%%%%%%%%%%%
\subsection{Morse complex}\label{ss:morse-complex}

Let $\xi_0$ be a gradient-like vector field for a Morse function $\mu:\R^4\to \R$ that is standard outside $D^4$. Suppose that $\xi_0$ is Morse--Smale, namely, all the intersections between the descending manifolds and the ascending manifolds of critical points of $\xi_0$ are transversal.
%Moreover, we assume that the restrictions of the Riemannian metric on $\R^4$ on a neighborhood of each critical point of $\mu$ is Euclidean with respect to the local coordinate of the Morse lemma.
Let $C_k=C_k(\xi_0)$ be the free $\Z$-module generated by the set $P_k$ of critical points of $\mu$ of index $k$ and $\partial:C_{k+1}\to C_k$ is defined for $p\in P_{k+1}$ by
\[\begin{split}
	\partial p=\sum_{q\in P_k}\#\bbcalM{\xi_0}{p}{q}\cdot q,\quad \bbcalM{\xi_0}{p}{q}=(\calD_p(\xi_0)\cap \calA_q(\xi_0))\cap Q_p,\\
\end{split}\]
where $Q_p$ is a level surface of $\mu$ that lies just below $p$ and $\bbcalM{\xi_0}{p}{q}$ is an oriented 0-manifold whose orientation is derived from those of $\calD_p(\xi_0)$ and $\calA_q(\xi_0)$. More precisely, $\calD_p(\xi_0)\cap \calA_q(\xi_0)$ is a union of finitely many integral curves of $-\xi_0$. We orient $\calD_p(\xi_0)$ and $\calA_p(\xi_0)$ so that $o(\calD_p(\xi_0))_p\wedge o(\calA_p(\xi_0))_p \sim o(\R^4)_p$. Especially, if $|p|=4$, we set $o(\calD_p(\xi_0))=o(\R^4)$, $o(\calA_p(\xi_0))=1$, and if $|p|=0$, we set $o(\calA_p(\xi_0))=o(\R^4)$, $o(\calD_p(\xi_0))=1$. Let $o^*_{\R^4}(\calD_p(\xi_0))$ and $o^*_{\R^4}(\calA_p(\xi_0))$ be the coorientations of $\calD_p(\xi_0)$ and $\calA_p(\xi_0)$ in $\R^4$ respectively, as in \S\ref{s:ori}. At each point $b\in \bbcalM{\xi_0}{p}{q}$, the wedge $o^*_{\R^4}(\calD_p(\xi_0))_b\wedge o^*_{\R^4}(\calA_q(\xi_0))_b\in \bigwedge^{d-1}T^*_b \R^4$ defines a coorientation of $\calD_p(\xi_0)\cap \calA_q(\xi_0)$ passing through $b$. Hence there exists a sign $\ve(p,q)_b=\pm 1$ such that
\[ o^*_{\R^4}(\calD_p(\xi_0))_b\wedge o^*_{\R^4}(\calA_q(\xi_0))_b \sim \ve(p,q)_b \,\, \iota(-\xi_0)o(\R^4)_b. \]
The sign $\ve(p,q)_b$ does not depend on the choice of $Q_p$. Then 
\[ \#\bbcalM{\xi_0}{p}{q}=\sum_{b\in \bbcalM{\xi_0}{p}{q}} \ve(p,q)_b. \]
It is known that $(C_*,\partial)$ above is a chain complex (e.g., \cite{Bo, AD}) called the Morse complex\footnote{It should probably be called the ``Morse--Thom--Smale--Witten complex''.}. For a Morse function $\mu:\R^4\to \R$ as above, the complex $(C_*,\partial)$ is acyclic. So there exists a $\Z$-linear map $g:C_*(\xi_0)\to C_{*+1}(\xi_0)$ such that $\partial g+g\partial=\mathrm{id}$. Such a $g$ is called a chain contraction or a {\it combinatorial propagator}.

%%%%%%%%%%%%%%%%%%%%%%%%%%%%%%%
\subsection{Z-paths}

A family version of the Morse complex is obtained by counting ``Z-paths'' defined below (\cite{Wa4}, see also \cite{Hu}). Take $f,\xi,\eta$ as in \S\ref{ss:FMF}. 
We say that a piecewise smooth embedding $\sigma:[\mu,\nu]\to E$ is {\it vertical} if $\mathrm{Im}\,\sigma$ is included in a single fiber of $\pi$, and say that $\sigma$ is {\it horizontal} if $\mathrm{Im}\,\sigma$ is included in a critical locus of $f$. We say that a vertical embedding (resp. horizontal embedding) $\sigma:[\mu,\nu]\to E$ is {\it descending} if $f(\sigma(\mu))\geq  f(\sigma(\nu))$ (resp. $h\circ\pi(\sigma(\mu))\leq h\circ\pi(\sigma(\nu))$). 

A {\it flow-line of $-\xi$} is a vertical smooth embedding $\sigma:[\mu,\nu]\to E$ such that for each $T\in (\mu,\nu)$, there exists a positive real number $C_T$ such that
\[ d\sigma_T\Bigl(\frac{\partial}{\partial T}\Bigr)=-C_T\,\xi_{\sigma(T)}. \]

\begin{Def}\label{def:al-path}
Let $x,y$ be two points of $E$ such that $h\circ\pi(x)\geq h\circ\pi(y)$. A {\it Z-path for $(\xi,\eta)$ from $x$ to $y$} is a sequence $\gamma=(\sigma_1,\sigma_2,\ldots,\sigma_n)$, $n\geq 1$, where
\begin{enumerate}
\item For each $i$, $\sigma_i$ is either a vertical or horizontal embedding $[\mu_i,\nu_i]\to E$ for some real numbers $\mu_i,\nu_i$. For each $i$, $\sigma_i$ is descending.
\item $\sigma_i(\nu_i)=\sigma_{i+1}(\mu_{i+1})$ for $1\leq i<n$. $\sigma_1(\mu_1)=x$, $\sigma_n(\nu_n)=y$.
\item If $\sigma_i$ is vertical (resp. horizontal) and if $i<n$, then $\sigma_{i+1}$ is horizontal (resp. vertical).
\item If $\sigma_i$ is vertical, then $\sigma_i$ is a flow-line of $-\xi$. If moreover $i\neq 1,n$, then $\mu_i<\nu_i$. If $\sigma_i$ is horizontal, then $\mu_i<\nu_i$ and $\pi\circ \sigma_i:[\mu_i,\nu_i]\to B$ is a flow-line of $-\eta$ in $B$.
\item If $n=1$, then $\mu_1<\nu_1$.
\end{enumerate}
\begin{figure}
\includegraphics[height=25mm]{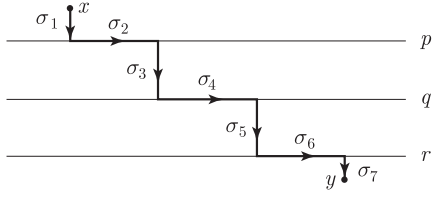}
\caption{A Z-path from $x$ to $y$ passing through three critical loci $p,q,r$. Each vertical segment between critical loci is in $i/i+\ell$-intersection for some $\ell$. This is a point of $\int_B X\,\Omega\,\Omega\,Y^T$ (\S\ref{ss:iterated-int-Q}). }\label{fig:z-path}
\end{figure}
(See Figure~\ref{fig:z-path} for an example of a Z-path.) We say that two Z-paths are {\it equivalent} if they differ only by orientation-preserving reparametrizations on segments. For a Z-path $\gamma=(\sigma_1,\ldots,\sigma_n)$ for $(-\xi,\eta)$, the {\it inverse Z-path} is the sequence $\gamma'=(\sigma_n',\ldots,\sigma_1')$, where $\sigma_i'$ is the inverse path of $\sigma_i$: $\sigma_i'(T)=\sigma_i(\mu_i+\nu_i-T)$.
\end{Def}

%%%%%%%%%%%%%%%%%%%%%%%%%%%%%%%
\subsection{Counting Z-paths over a path in $B$}\label{ss:count-Z-path}

We assume Assumption~\ref{hyp:eta} and that $\eta$ is Morse--Smale. Let $\alpha:I\to B$ be a flow-line of $-\eta$ between two points $a,b\in B$. Let $x\in F_a,y\in F_b$ be critical points of $f_a,f_b$ respectively with $|x|=|y|=i$. By choosing $\xi$ generically, we may assume that $\alpha$ intersects $B^{(1)}$ (Lemma~\ref{lem:i/i-str}) transversally at finitely many points and does not intersect $B^{(k)}$, $k\geq 2$. 

Under the assumptions above, $i/i$-intersections may occur finitely many times in the bundle $\alpha^*\pi:\alpha^* E\to I$ restricted over $\alpha$, and hence there may be only finitely many Z-paths from $x$ to $y$. We count these Z-paths with signs that are determined as follows. For an $i/i$-intersection $\sigma$ between critical loci $p,q$ of $f$, its sign $\ve(\sigma)$ is determined by the orientations of the ascending and descending manifolds. Namely, we choose coorientations $o^*_E(\calD_p(\xi))$, $o^*_E(\calA_q(\xi))$ of $\calD_p(\xi)$, $\calA_q(\xi)$ in $E$, respectively. Let $o^*_E(\widetilde{L}_z)$ be the coorientation of the level surface locus $\widetilde{L}_z$ of $f$ that passes through $z\in\sigma$, which restricts to $\iota(-\xi_z)\,o(F_{\pi(z)})_z$ on a fiber. Then $o^*_E(\calD_p(\xi))_z\wedge o^*_E(\calA_q(\xi))_z\wedge o^*_E(\widetilde{L}_z)_z$ is a 5-form in $\bigwedge^*T_zE$. We consider the pullback of this 5-form to $\alpha^*E$ by the natural bundle map $\hat{\alpha}:\alpha^*E\to E$. If this is equivalent to the original orientation of $\alpha^*E$ given by the rule (\ref{eq:o(E)}), then set $\ve(\gamma)=1$ and otherwise set $\ve(\gamma)=-1$. Let $\sigma_1,\sigma_2,\ldots,\sigma_r$ be all the $i/i$-intersections included in a Z-path $\gamma$ from $x$ to $y$ in the $(\R^4,U_\infty')$-bundle $\alpha^*\pi:\alpha^*E\to I$. Then we define the sign of $\gamma$ by 
\begin{equation}\label{eq:e(gamma)}
  \ve(\gamma)=\ve(\sigma_1)\ve(\sigma_2)\cdots \ve(\sigma_r).
\end{equation}

We define a $\Z$-linear map $\Phi_\alpha:C_*(\xi_a)\to C_*(\xi_b)$ by letting
\begin{equation}\label{eq:n(x,y)}
\Phi_\alpha(x)=\sum_{y\in P_*(\xi_b)}n_\alpha(x,y)y,\quad n_\alpha(x,y)=\sum_{\gamma}\ve(\gamma)
\end{equation}
for each $x\in P_*(\xi_a)$, where the sum for $n_\alpha(x,y)$ is taken for all Z-paths in $\alpha^*E$ that goes from $x$ to $y$.
\begin{Lem}\label{lem:Phi-chain}
$\Phi_\alpha$ is a chain map, namely, $\Phi_\alpha\circ \partial^a=\partial^b\circ \Phi_\alpha$ for the boundary operators $\partial^a,\partial^b$ of $C_*(\xi_a),C_*(\xi_b)$ respectively.
\end{Lem}
Lemma~\ref{lem:Phi-chain} will be proved later in \S\ref{ss:parallel}.
Under the identification $C_*(\xi_a)=C_*(\xi_b)=C_*(\xi_0)$ as $\Z$-modules induced by critical loci, where $\xi_0=\xi_{b_0}$ is a gradient-like vector field for the Morse function $f_0:F_0=F_{b_0}\to \R$ on the base fiber, 
we may consider $\Phi_\alpha$ as a $\Z$-linear endomorphism $C_*(\xi_0)\to C_*(\xi_0)$. 

Here, we make the following additional assumption on $f,\xi,\eta$, for simplicity, which is enough for our purpose.
\begin{Assum}\label{hyp:cv-const}
\begin{enumerate}
\item The critical values of $f_b$ are constants over $B$.
\item For any pair $a,b\in P_*(\eta)$ with $|a|=|b|+1$ and for a flow-line $\alpha:I\to B$ of $-\eta$ between $a$ and $b$, the path $\alpha$ intersects $B^{(1)}$ transversally and does not intersect $B^{(k)}$ for $k\geq 2$. Under the identification $C_*(\xi_a)=C_*(\xi_b)=C_*(\xi_0)$ by critical loci, $\Phi_\alpha=1:C_*(\xi_0)\to C_*(\xi_0)$.
\end{enumerate}
\end{Assum}

%%%%%%%%%%%%%%%%%%%%%%%%%%%%%%%
\subsection{Graph counting formula}\label{ss:gcf}

Let $f^{(1)},f^{(2)},\ldots,f^{(3k)}:E\to \R$ be a sequence of fiberwise Morse functions and let $\xi^{(i)}$ be a $v$-gradient for $f^{(i)}$. We assume that $(f^{(i)},\xi^{(i)})$ satisfies Assumption~\ref{hyp:eta} for each $i$. Let $f_0^{(i)}:F_0\to \R$ be the restriction of $f^{(i)}$ on the base fiber $F_0=F_{b_0}$ and let $\xi_0^{(i)}=\xi_{b_0}^{(i)}$. We consider a \emph{connected} edge-oriented trivalent graph with its sets of vertices and edges labelled and with $2k$ vertices and $3k$ edges. Choose some of the edges and split each chosen edge into two arcs. We attach elements of $P_*(\xi_0^{(i)})$ on the two univalent vertices ({\it white vertices}) that appear after the splitting of the $i$-th edge. We call such obtained graph a {\it $\vec{C}$-graph} ($\vec{C}=(C_*(\xi_0^{(1)}),\ldots,C_*(\xi_0^{(3k)}))$, see Figure~\ref{fig:cgraph}). A $\vec{C}$-graph has two kinds of (possibly split) ``edges'': a {\it compact edge}, which is connected,  and a {\it separated edge}, which consists of two arcs. We call vertices that are not white vertices {\it black vertices}. If $p_i$ (resp. $q_i$) is the critical point of $f_0^{(i)}$ attached on the input (resp. output) white vertex of a separated edge $i$, we define the degree of $i$ by $\deg(i)=|p_i|-|q_i|$. We define the degree of a compact edge $i$ by $\deg(i)=1$. We define the degree of a $\vec{C}$-graph by $\deg(\Gamma)=(\deg(1),\deg(2),\ldots,\deg(3k))$. 
\begin{figure}
\fig{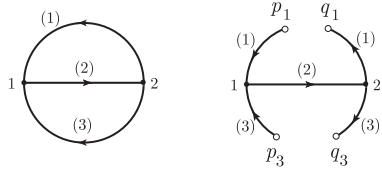}
\caption{A trivalent graph (left) and a $\vec{C}$-graph (right)}\label{fig:cgraph}
\end{figure}

We say that a continuous map $I$ from a $\vec{C}$-graph $\Gamma$ to $E$ is a {\it Z-graph} for the sequence $\vec{\xi}=(\xi^{(1)},\xi^{(2)},\ldots,\xi^{(3k)})$ of $v$-gradients and the $h$-gradient $\eta$ if it satisfies the following conditions (see Figure~\ref{fig:flow-graph}).
\begin{figure}
\fig{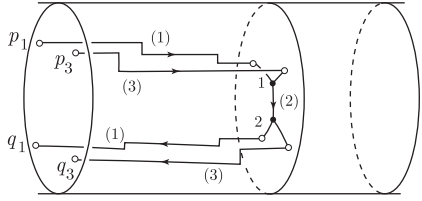}
\caption{Z-graph for $\vec{\xi}=(\xi^{(1)},\xi^{(2)},\xi^{(3)})$}\label{fig:flow-graph}
\end{figure}

\begin{enumerate}
\item Every white vertex equipped with $p_i$ is mapped by $I$ to the corresponding critical point $p_i$ in $F_0$.
\item The restriction of $I$ to the $i$-th edge of $\Gamma$ is either a vertical flow-line of $-\xi^{(i)}$ or a pair $(\gamma,\delta)$ consisting of a Z-path $\gamma$ from a critical point of $f_0^{(i)}$ and an inverse Z-path $\delta$ to a critical point of $f_0^{(i)}$ both for $(\xi^{(i)},\eta)$ such that the terminal endpoint of $\gamma$ and the initial endpoint of $\delta$ lie in the same fiber of $\pi$ and the projections of $\gamma$ and $\delta$ on $B$ give equivalent (piecewise smooth) flow-lines of $-\eta$.
\end{enumerate}
For a $\vec{C}$-graph $\Gamma$, let $\acalM_\Gamma(\vec{\xi},\eta)$ be the set of equivalence classes of all Z-graphs for $(\vec{\xi},\eta)$ from $\Gamma$ to $E$, where we say that two Z-graphs are {\it equivalent} if they differ only by orientation-preserving reparametrizations on segments. By definition, all the black vertices (or trivalent vertices) of a Z-graph must be contained in a single fiber of $\pi$. Hence a Z-graph consists of a uni-trivalent graph $V$ in a single fiber with some Z-paths or inverse Z-paths from/to the base fiber $F_0$ attached to the univalent vertices of $V$. The following lemma is a straightforward analogue of \cite{Fu, Wa3}.
 
\begin{Lem}\label{lem:0-mfd}
Suppose that $(f,\vec{\xi},\eta)$ satisfies Assumptions~\ref{hyp:eta} and \ref{hyp:cv-const}. If $\dim{B}=k$ and if $\vec{f}=(f^{(1)},f^{(2)},\ldots,f^{(3k)})$ and $\vec{\xi}$ are generic, then for every $\vec{C}$-graph $\Gamma$ with $2k$ black vertices and $\deg(\Gamma)=(1,1,\ldots,1)$, the space $\acalM_\Gamma(\vec{\xi},\eta)$ is a compact 0-dimensional manifold. 
\end{Lem}

We give a proof of Lemma~\ref{lem:0-mfd} in \S\ref{s:transversality}. From Lemma~\ref{lem:0-mfd}, it follows that there are only finitely many points in $B$ on which uni-trivalent graphs of Z-graphs are included.

When the assumption of Lemma~\ref{lem:0-mfd} is satisfied, we may define an orientation of $\acalM_\Gamma(\vec{\xi},\eta)$ in a similar way as \cite{Wa3}. Roughly, an orientation of $\acalM_\Gamma(\vec{\xi},\eta)$ is defined as follows. The space $\acalM_\Gamma(\vec{\xi},\eta)$ can be considered as the intersection of several smooth manifold strata in $EC_{2k}(\pi)$ each corresponds to the space of an edge of $\Gamma$. We orient $\acalM_\Gamma(\vec{\xi},\eta)$ by the coorientation $\bigwedge_{e}v_e$ of $\acalM_\Gamma(\vec{\xi},\eta)$ in $EC_{2k}(\pi)$ for some coorientations $v_e$ of the strata for each compact or separated edge $e$. Each compact or separated edge $e$ has two black vertices and $v_e$ is a vector in $\bigwedge^3T_{(x,y)}EC_2(\pi)$, where $x,y$ are the images from the black vertices of $e$.  Then $\#\acalM_\Gamma(\vec{\xi},\eta)\in \Z$ is defined as the count of the Z-graphs with orientations.

Let $\vec{g}=(g^{(1)},g^{(2)},\ldots,g^{(3k)})$ be a sequence of combinatorial propagators for $\vec{C}=(C_*(\xi_0^{(1)}),\ldots,C_*(\xi_0^{(3k)}))$. Then we define
\begin{equation}\label{eq:Z_k}
Z_k^\Morse(\vec{\xi},\eta)=\frac{1}{2^{3k}(2k)!(3k)!}\Tr_{\vec{g}}\Bigl(\sum_\Gamma \#\acalM_\Gamma(\vec{\xi},\eta)\, \Gamma\Bigr)\in\calA_k. 
\end{equation}
Here, the sum is taken over all possible $\vec{C}$-graphs $\Gamma$ with $2k$ black vertices and $\deg(\Gamma)=(1,1,\ldots,1)$, and $\Tr_{\vec{g}}$ is defined as follows. For simplicity, we assume that the labels for the separated edges in a $\vec{C}$-graph $\Gamma$ is $1,2,\ldots,a$. Let $p_i,q_i$ be the critical points of $f_0^{(i)}$ on the input and output of the $i$-th edge of $\Gamma$, respectively and let $g_{q_ip_i}^{(i)}\in\Q$ be the coefficient of $p_i$ in the expansion of $g^{(i)}(q_i)$. Then $\Tr_{\vec{g}}(\Gamma)$ is defined by the following formula.
\begin{center}
 $\Tr_{\vec{g}}\,\, $\raisebox{-0.4\height}{\includegraphics{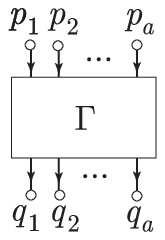}}
  $=\displaystyle \prod_{i=1}^a(-g_{q_ip_i}^{(i)})\times$ \raisebox{-0.4\height}{\includegraphics{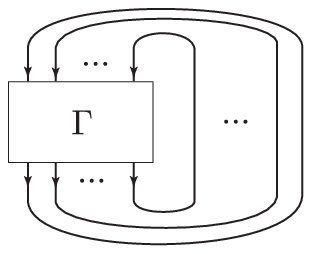}}
\end{center}
The definition of $\Tr$ can be generalized to graphs with other degrees in the same manner. The following theorem, which gives an analogue of Shimizu's identity in \cite{Sh}, is the main result of this section.

\begin{Thm}\label{thm:GCF}
For $(f,\vec{\xi},\eta)$ satisfying Assumptions~\ref{hyp:eta} and \ref{hyp:cv-const}, there exists a tuple of admissible propagators $\vec{\theta}=(\theta^{(1)},\ldots,\theta^{(3k)})$ such that the following identity holds.
\[ Z_k^\adm(\vec{\theta})=\frac{(-1)^{3k}}{2^{3k}}\sum_{\ve_i=\pm 1}Z_k^\Morse((\ve_1\xi^{(1)},\ldots,\ve_{3k}\xi^{(3k)}),\eta) \]
In such a case, we will denote $\hat{Z}_k^\adm(\vec{\theta})$ by $\hat{Z}_k^\Morse(\vec{\xi},\eta)$. 
\end{Thm}

%%%%%%%%%%%%%%%%%%%%%%%%%%%%%%
\subsection{Moduli space of vertical flow-lines}\label{ss:M2}

Here we give a preliminary for an iterated integral description of spaces of Z-paths, which is used to describe their boundaries. 
Let $\pi:E\to B$ be an $(\R^4,U_\infty')$-bundle over a closed oriented manifold $B$ and $\xi$ be a $v$-gradient for a fiberwise Morse function $f:E\to \R$, as above. For simplicity, we assume that $B$ is path connected. We define
\[ \calM_2(\xi)=\{(x,y)\in E\times E\mid \pi(x)=\pi(y),\,y=\Phi_{-\xi}^T(x)\mbox{ for some $T>0$}\}. \]
Let $\Sigma(\xi)$ denote the union of all the critical loci of $\xi$. 
For a pair $(x,y)$ of distinct points of $E-\Sigma(\xi)$ such that $\pi(x)=\pi(y)=s$, a {\it ($r$ times) broken flow-line between $x$ and $y$} is a sequence $\gamma_0,\gamma_1,\ldots,\gamma_r$ ($r\geq 1$) of integral curves of $-\xi$ in the fiber $\pi^{-1}(s)$ satisfying the following conditions:
\begin{enumerate}
\item The domain of $\gamma_0$ is $[0,\infty)$, the domain of $\gamma_r$ is $(-\infty,0]$ and the domain of $\gamma_i$, $1\leq i\leq r-1$, is $\R$.
\item $\gamma_0(0)=x$, $\gamma_r(0)=y$.
\item There is a sequence $q_1,q_2,\ldots,q_r$ of distinct critical loci of $\xi$ such that $\lim_{T\to \infty}\gamma_{i-1}(T)=\lim_{T\to -\infty}\gamma_i(T)\in q_i$ ($1\leq i\leq r$).
\end{enumerate}
%A broken flow-line $(\gamma_0,\gamma_1,\ldots,\gamma_r)$ between $x$ and $y$ is determined uniquely up to reparametrizations by the boundary points $x,y$ and intersection points of $\gamma_i$ with level surfaces that lie between $q_i$ and $q_{i+1}$ ($i=1,2,\ldots,r-1$). 

A {\it fiberwise space over a space $X$} is a pair of a space $Y$ and a continuous map $\phi:Y\to X$. A {\it fiber} over a point $s\in X$ is $Y(s)=\phi^{-1}(s)$ (\cite{CJ}). For two fiberwise spaces $Y_1=(Y_1,\phi_1)$ and $Y_2=(Y_2,\phi_2)$ over $X$, a {\it fiberwise product} $Y_1\times_X Y_2$ is defined as the following subspace of $Y_1\times Y_2$:
\[ Y_1\times_X Y_2=\int_{s\in X}Y_1(s)\times Y_2(s), \]
where $\int_{s\in X}$ means $\bigcup_{s\in X}$. Namely, $Y_1\times_X Y_2$ is the pullback of $Y_1\stackrel{\phi_1}{\to} X \stackrel{\phi_2}{\leftarrow} Y_2$. 

Let $\winfty$ denote the image of the $\infty$-section of the associated $(S^4,U_\infty)$-bundle $\pi^\infty:E^\infty\to B$ to $\pi:E\to B$ and let $\wDelta_E=\Delta_E\cup (E^\infty\times_B \ell_\infty)\cup (\ell_\infty\times_B E^\infty)$. Let $C(\xi)$ be the set of path-components in $\Sigma(\xi)$. The following proposition is a straightforward analogue of \cite[Proposition~3.4, 8.4]{Wa3}.
\begin{Prop}\label{prop:bM2_mfd}
If $\xi$ is generic, then there is a natural compactification $\bcalM_2(\xi)$ of $\calM_2(\xi)$ into a stratified space satisfying the following conditions.
\begin{enumerate}
\item Let $\ev:\bcalM_2(\xi)\to E^\infty\times E^\infty$ be the natural map, which assigns the endpoints. Then $\bcalM_2(\xi)-\ev^{-1}(\wDelta_{E})$ is a manifold with corners.
\item $\ev$ induces a diffeomorphism $\mathrm{Int}\,\bcalM_2(\xi)\to \calM_2(\xi)$, where $\mathrm{Int}$ denotes the codimension 0 stratum.
\item The codimension $r$ stratum of $\bcalM_2(\xi)-\ev^{-1}(\wDelta_{E})$ consists of $r$ times broken flow-lines. The codimension $r$ stratum of $\bcalM_2(\xi)-\ev^{-1}(\wDelta_{E})$ for $r\geq 1$ is canonically diffeomorphic to
\[ \left\{\begin{array}{ll}
\displaystyle\int_{s\in B}\sum_{q_1\in C(\xi)}\calA_{q_1}(\xi_{s})\times \calD_{q_1}(\xi_{s})-\sum_{q_1\in C(\xi)}\Delta_{q_1} & (\mbox{if $r=1$})\\
\displaystyle\int_{s\in B}\sum_{{{q_1,\ldots,q_r\in C(\xi)}\atop{q_1,\ldots,q_r\,\mathrm{distinct}}}}\calA_{q_1}(\xi_{s})\times \bbcalM{\xi_{s}}{q_1}{q_2}\times\cdots\times\bbcalM{\xi_{s}}{q_{r-1}}{q_r}\times\calD_{q_r}(\xi_{s}) & (\mbox{if $r\geq 2$})
\end{array}\right.
\]
\end{enumerate}
\end{Prop}

Let $\{p_1,\ldots,p_N\}$ be the set of all critical loci of $f$ numbered so that $f(p_1(b_0))>f(p_2(b_0))>\cdots>f(p_N(b_0))$. The formula of the codimension $r$ stratum of $\bcalM_2(\xi)-\ev^{-1}(\wDelta_{E})$ for $r\geq 2$ in Proposition~\ref{prop:bM2_mfd} can be abbreviated as 
\[ \int_{s\in B}X_0^\circ(s)\times \underbrace{\Omega_0^\circ(s)\times \cdots\times \Omega_0^\circ(s)}_{r-1}\times Y_0^{\circ T}(s),\quad\mbox{where} \]
$X_0^\circ(s)=(\calA_{p_1}(\xi_{s})\ \ \calA_{p_2}(\xi_{s})\ \ \cdots\ \ \calA_{p_N}(\xi_{s}))$, $Y_0^\circ(s)=(\calD_{p_1}(\xi_{s})\ \ \calD_{p_2}(\xi_{s})\ \ \cdots\ \ \calD_{p_N}(\xi_{s}))$, 
\[   \Omega_0^\circ(s)=((1-\delta_{ij})\calM'_{p_ip_j}(\xi_s))=\left(\begin{array}{cccc}
    \emptyset & \calM'_{p_1p_2}(\xi_{s}) &  \cdots & \calM'_{p_1p_N}(\xi_{s})\\
    \calM'_{p_2p_1}(\xi_{s}) & \emptyset &  \cdots & \calM'_{p_2p_N}(\xi_{s})\\
    \vdots & \vdots & \ddots & \vdots\\
    \calM'_{p_Np_1}(\xi_{s}) & \calM'_{p_Np_2}(\xi_{s}) & \cdots & \emptyset
    \end{array}\right),
\]
and the direct product of matrices is given by matrix multiplication with the multiplications and the sums given respectively by direct products and disjoint unions (\S\ref{ss:iterated-integral}). The codimension $r$ stratum of a matrix of stratified spaces will denote the matrix whose entries are the codimension $r$ strata of the given matrix. The following two propositions are the restrictions of Proposition~\ref{prop:bM2_mfd}.

\begin{Prop}\label{prop:bD}
Let $p$ be a critical locus of $f$ and let $\bcalD_p(\xi)=\ev^{-1}(p\times E^\infty)$, $\bcalA_p(\xi)=\ev^{-1}(E^\infty\times p)$. If $\xi$ is generic, then
\begin{enumerate}
\item $\bcalD_p(\xi)$ (resp. $\bcalA_p(\xi)$) is a compact manifold with corners.
\item $\ev$ induces a diffeomorphism $\mathrm{Int}\,\bcalD_p(\xi)\to \calD_p(\xi)$ (resp. $\mathrm{Int}\,\bcalA_p(\xi)\to \calA_p(\xi)$).
\item The codimension $r$ stratum of $\t{Y}_0=(\bcalD_{p_1}(\xi)\ \ \bcalD_{p_2}(\xi)\ \ \cdots\ \ \bcalD_{p_N}(\xi))^T$ (resp. $X_0=(\bcalA_{p_1}(\xi)\ \ \bcalA_{p_2}(\xi)\ \ \cdots\ \ \bcalA_{p_N}(\xi))$) for $r\geq 1$ is canonically diffeomorphic to
$\displaystyle\int_{s\in B}\underbrace{\Omega_0^\circ(s)\times \cdots \times\Omega_0^\circ(s)}_{r}\times Y_0^{\circ T}(s)$ $
\mbox{(resp. } \displaystyle\int_{s\in B}X_0^\circ(s)\times\underbrace{\Omega_0^\circ(s)\times \cdots \times\Omega_0^\circ(s)}_{r}\mbox{)}$.
\end{enumerate}
\end{Prop}

\begin{Prop}\label{prop:bMpq}
Let $p,q$ be critical loci of $f$ and let $\pbcalM_{pq}(\xi)=\ev^{-1}(p\times q)$. If $\xi$ is generic, then
\begin{enumerate}
\item $\pbcalM_{pq}(\xi)$ is a compact manifold with corners.
\item There is a natural diffeomorphism $\mathrm{Int}\,\pbcalM_{pq}(\xi)\to \bbcalM{\xi}{p}{q}$.
\item The codimension $r$ stratum of $\Omega_0=((1-\delta_{ij})\pbcalM_{p_ip_j}(\xi))$ for $r\geq 1$ is canonically diffeomorphic to
$\displaystyle\int_{s\in B}\underbrace{\Omega_0^\circ(s)\times \cdots \times \Omega_0^\circ(s)}_{r+1}$.
\end{enumerate}
\end{Prop}

\begin{Rem}\label{rem:M-infty}
Let 
\[\begin{split}
\calD_\infty(\xi)&=\{x\in E\mid \lim_{t\to -\infty}\Phi_{-\xi}^t(x)\in\winfty\}, \quad \calA_\infty(\xi)=\{x\in E\mid \lim_{t\to \infty}\Phi_{-\xi}^t(x)\in\winfty\},\\
\calM_{p\infty}'(\xi)&=(\calD_p(\xi)\cap\calA_\infty(\xi))/\R, \qquad \calM_{\infty q}'(\xi)=(\calD_\infty(\xi)\cap\calA_q(\xi))/\R.
\end{split} \] 
Although $\calD_\infty(\xi)$ (resp. $\calA_\infty(\xi)$) is similar to $\calD$ (resp. $\calA$) of critical locus of index 4 (resp. 0), we define its coorientation by $o^*_{E^\infty}(\calD_\infty(\xi))=-1$ (resp. $o^*_{E^\infty}(\calA_\infty(\xi))=-1$), which is opposite to that of usual critical loci given in \S\ref{ss:morse-complex}. The reason for the minus sign is the orientation convention for the infinite 3-sphere on the boundary of $\bConf_2(S^4,\infty)$. 

There are natural compactifications 
\[\begin{split}
\bcalD_\infty(\xi)&=\ev^{-1}(\ell_\infty\times E^\infty),\qquad \bcalA_\infty(\xi)=\ev^{-1}(E^\infty\times \ell_\infty),\\
\pbcalM_{p\infty}(\xi)&=\ev^{-1}(p\times \ell_\infty),\qquad \pbcalM_{\infty q}(\xi)=\ev^{-1}(\ell_\infty\times p)
\end{split} \]
of $\calD_\infty(\xi)$, $\calA_\infty(\xi)$, $\calM_{p\infty}'(\xi)$ and $\calM_{\infty q}'(\xi)$ respectively into smooth compact manifolds with corners, as analogues of Propositions~\ref{prop:bD} and \ref{prop:bMpq}. We consider $\winfty\subset E^\infty$ as a critical locus of $\xi$ and will allow Z-paths in $E^\infty$ to pass through $\winfty$ for the compactification of the moduli space of Z-paths.
\end{Rem}

%%%%%%%%%%%%%%%%%%%%%%%%%%%%%%%
\subsection{Admissible propagator from Z-paths}

Let $\acalM_2(\xi,\eta)$ be the set of equivalence classes of all Z-paths in $E$ for $(\xi,\eta)$. It will turn out that there is a natural structure of non-compact manifold on $\acalM_2(\xi,\eta)$. Roughly, since an equivalence class of $\gamma=(\sigma_1,\ldots,\sigma_n)$ in $\acalM_2(\xi,\eta)$ may be described by a sequence of vertical flow-lines, it may be locally described as a subset of the direct product of spaces of (vertical) flow-lines of $-\xi$.

\begin{Prop}[Proof in \S\ref{ss:iterated-int-Q}]\label{prop:bM}
Suppose that $B$ is path-connected. There is a natural path-connected compactification $\bacalM_2(\xi,\eta)$ of $\acalM_2(\xi,\eta)$ that has the structure of a sum of finitely many smooth compact manifolds with corners.
\end{Prop}

Now we shall define a chain $\bar{s}_\beta$ in $\partial^\fib\EC_2(\pi)$ and chains $M(\xi)$, $H_{pq}(\xi,\eta)$, $\theta_\mathrm{Z}(\xi,\eta)$ and $\theta_\mathrm{Z}^*(\xi,\eta)$ in $\EC_2(\pi)$. 

\subsubsection{$\bar{s}_\beta$} 
For a Morse $v$-gradient $\beta$ of $T^v(E)$ that is constant near $\ell_\infty$ with respect to the partial trivialization of $\pi$ near $\ell_\infty$, let $\Sigma(\beta)=\{x\in E\mid \beta_x=0\}$ and let $s_{\beta}:E-\Sigma(\beta)\to ST^v(E)$ be the pointwise normalization $\beta_x/\|\beta_x\|$ of $\beta$. By the local formula of the $v$-gradient near critical points, we see that $s_\beta$ has a natural smooth extension
\[ s_\beta:B\ell^{\,\fib}(E,\Sigma(\beta))\to ST^v(E),\]
where $B\ell^\fib(X,Y)$ for subbundles $X,Y$ of a fiber bundle is the fiberwise blow-up: $\bigcup_b B\ell(X_b,X_b\cap Y_b)$. 
We identify $ST^v(\Delta_E)$ with $ST^v(E)$, and we denote by $\overline{s}_\beta$ the chain of $\partial^\fib\EC_2(\pi)$ obtained by extending $s_\beta$ by $\phi_0^{-1}(\{a\})$ for some $a\in S^3$ (see \S\ref{ss:adm-propagator} for the definition of $\phi_0$). We orient $\overline{s}_\beta$ by the coorientation that extends $\phi_0^*\omega_{S^3}$ on $\partial^\fib\EC_2(\pi)-\mathrm{Int}\,S_{\Delta_E}$, where $\omega_{S^3}$ is the $SO_4$-invariant unit volume form on $S^3$. 

\subsubsection{$M(\xi)$} 
The natural map $i:\bcalM_2(\xi)-\ev^{-1}(\wDelta_E)\to \mathrm{Int}\,\EC_2(\pi)$ restricts to an embedding on the preimage of a neighborhood of $\partial^\fib\EC_2(\pi)$ and the closure of the partial image in $\EC_2(\pi)$ is a manifold with corners. By extending $\bcalM_2(\xi)-\ev^{-1}(\wDelta_E)$ by attaching the corners now obtained, a compact manifold with corners is obtained and we denote it by $\ibcalM_2(\xi)$. The map $i$ is extended to a smooth map 
\[ M(\xi):\ibcalM_2(\xi)\to \EC_2(\pi), \]
which gives a $(k+5)$-chain if an orientation of $\ibcalM_2(\xi)$ is given. We orient $\calM_2(\xi)$ as follows. For a generic parameter $t\in B$ such that $\xi_t$ is Morse--Smale, we give $\calM_2(\xi_t)$ a coorientation in $F_t\times F_t$ and extend it to that of $\calM_2(\xi)$ in $E\times_B E$. More precisely, since $\calM_2(\xi_t)$ is the image of the embedding $\varphi:F_t\times (0,\infty)\to F_t\times F_t$, $\varphi(u,T)=(u,\Phi_{-\xi_t}^T(u))$, we may define
\[\begin{split}
o(\calM_2(\xi_t))_{(u,v)}&=d\varphi_*(o(F_t)_u\wedge dT), \quad o^*_{F_t\times F_t}(\calM_2(\xi_t))_{(u,v)}=*\, o(\calM_2(\xi_t))_{(u,v)}, 
\end{split}\]
where $d\varphi_*$ is the map induced by $d\varphi$ under the identifications $T(F_t\times (0,\infty))=T^*(F_t\times (0,\infty))$ and $T(F_t\times F_t)=T^*(F_t\times F_t)$ by orthonormal bases and $*$ is the Hodge star operator in $F_t\times F_t$. We choose $o^*_{E\times_B E}(\calM_2(\xi))$ so that its restriction to $F_t\times F_t$ is equivalent to $o^*_{F_t\times F_t}(\calM_2(\xi_t))_{(u,v)}$. Then we choose the orientation of $\ibcalM_2(\xi)$ that is compatible with $o(\calM_2(\xi))$. 

\subsubsection{$H_{pq}(\xi,\eta)$} 
Let $\acalM_2(\xi,\eta)'$ denote the space of equivalence classes of inverse Z-paths and let $\bacalM_2(\xi,\eta)'$ be its compactification defined similarly as $\bacalM_2(\xi,\eta)$. Let $\ev_1,\ev_2:\bacalM_2(\xi,\eta)\to E^\infty$ be the maps giving the initial and terminal endpoints, respectively. Similarly, let $\overline{\ev}_1,\overline{\ev}_2:\bacalM_2(\xi,\eta)'\to E^\infty$ be the maps giving the initial and terminal endpoints, respectively. For critical loci $p,q$ of $\xi$, we define $\bD_p(\xi,\eta), \bA_q(\xi,\eta)$ as
\[ \bD_p(\xi,\eta)=\ev_1^{-1}(p\cap F_0),\quad \bA_q(\xi,\eta)=\overline{\ev}_2^{-1}(q\cap F_0). \]
These are the space of Z-paths from $p$ and that of inverse Z-paths to $q$, respectively. We will define the orientations of $\bD_p(\xi,\eta)$ and $\bA_q(\xi,\eta)$ later in \S\ref{ss:ori}.
We define the space of separated paths $\calH_{pq}^0(\xi,\eta)=\bA_q(\xi,\eta)\times_B\bD_p(\xi,\eta)$ by the pullback of $\pi\circ\overline{\ev}_1:\bA_q(\xi,\eta)\to B$ by $\pi\circ\ev_2: \bD_p(\xi,\eta)\to B$ (\S\ref{ss:iterated-integral}). The map $\ev_1\times \ev_2:\bA_q(\xi,\eta)\times \bD_p(\xi,\eta)\to E^\infty\times E^\infty$ induces a map $\ev:\calH_{pq}^0(\xi,\eta)\to E^\infty\times_B E^\infty$, which lifts to a smooth map
\[ H_{pq}(\xi,\eta):B\ell^\fib(B\ell^\fib(\calH_{pq}^0(\xi,\eta),\Sigma_0),\overline{\Sigma_1-\Sigma_0})\to  \EC_2(\pi), \]
where $\Sigma_0, \Sigma_1$ are the subbundles of $E^\infty\times_B E^\infty$ obtained by restricting the fiber to $\Sigma_0, \Sigma_1$ of \S\ref{ss:adm-propagator}, abusing the notation.

\subsubsection{$\theta_\mathrm{Z}(\xi,\eta)$ and $\theta_\mathrm{Z}^*(\xi,\eta)$} 
Let $g:C_*(\xi_0)\to C_{*+1}(\xi_0)$ be a combinatorial propagator for $C_*(\xi_0)$. Then 
\[ \theta_\mathrm{Z}(\xi,\eta)=M(\xi,\eta)-\sum_{p,q\in P_*(\xi_0)}g_{qp}H_{pq}(\xi,\eta) \]
defines a $(k+5)$-chain of $\EC_2(\pi)$. Let $\partial^*=-\partial^T:C_{*+1}(-\xi_0)\to C_*(-\xi_0)$, where $\partial^T$ is defined by the matrix transpose of $\partial$. Then $\partial^*$ is the boundary operator of the Morse complex for $-\xi_0$, and $g^*=-g^T:C_*(-\xi_0)\to C_{*+1}(-\xi_0)$ is a combinatorial propagator for $\partial^*$. We define
\[ \theta_\mathrm{Z}^*(\xi,\eta)=M(-\xi,\eta)-\sum_{p,q}g^*_{qp}H_{pq}(-\xi,\eta). \]

\begin{Thm}\label{thm:P-propagator}
Suppose $(f,\xi,\eta)$ satisfies Assumption~\ref{hyp:cv-const}. Then
\[ \hat\theta_\mathrm{Z}(\xi,\eta)=\theta_\mathrm{Z}(\xi,\eta)+\theta_\mathrm{Z}^*(\xi,\eta)\]
is a relative $(k+5)$-cycle of $(\EC_2(\pi),\partial \EC_2(\pi))$ such that
\[\partial \hat\theta_\mathrm{Z}(\xi,\eta)=-\overline{s}_\xi-\overline{s}_{-\xi}.\]
Hence $-\hat\theta_\mathrm{Z}(\xi,\eta)$ is an admissible propagator. 
\end{Thm}
The minus sign in the formula of $\partial \hat\theta_\mathrm{Z}(\xi,\eta)$ is due to the outward-normal-first convention for the boundary orientation (Appendix~\ref{s:ori} (\ref{eq:inward_first})). We shall prove Theorem~\ref{thm:P-propagator} in the rest of this section. 

\begin{proof}[Proof of Theorem~\ref{thm:GCF} assuming Theorem~\ref{thm:P-propagator}] 
If we take $-\hat{\theta}_\mathrm{Z}(\xi^{(j)},\eta)$ as an admissible propagator, the configuration of each intersection point of (\ref{eq:I^adm}) consists of the vertices in a mapping from a $\vec{C}$-graph to $E$ such that each edge is either a vertical flow-line of $M(\pm\xi^{(j)},\eta)$ or of a pair of Z-paths in $H_{pq}(\pm\xi^{(j)},\eta)$ for some $p,q$. This is precisely a Z-graph and then the right hand side is obtained immediately.
\end{proof}

%%%%%%%%%%%%%%%%%%%%%%%%%%%%%%%
\subsection{Iterated integrals of fiberwise spaces}\label{ss:iterated-integral}

The spaces of Z-paths given in the definition of $\theta_\mathrm{Z}$ can be described by a geometric analogue of K.~T.~Chen's iterated integrals (\cite{Ch}). Let $b_0$ be a base point of $B$. Let $PB$ denote the space of piecewise smooth paths $\gamma:[0,1]\to B$ such that $\gamma(0)=b_0$. Let $\sigma:C\to PB$ be a chain from a compact oriented manifold $C$ with corners. Let $\phi_i:A_i\to B$ ($i=1,\ldots,k$) be fiberwise spaces over $B$ (\S\ref{ss:M2}). We define the map $\hat\sigma:C\times \Delta^{k-1}\to B^k$ by 
\[ (\gamma,s_1,\ldots,s_{k-1}) \mapsto (\bar\gamma(s_1),\cdots,\bar\gamma(s_{k-1}),\bar\gamma(1)), \]
where $\bar\gamma=\sigma(\gamma)$, $\Delta^{k-1}=\{(s_1,\ldots,s_{k-1})\in \R^{k-1}\mid 0\leq s_1\leq\cdots\leq s_{k-1}\leq 1\}$. Then the iterated integral is defined by  $\displaystyle \int_\sigma A_1\cdots A_k=\hat\sigma^*(A_1\times\cdots\times A_k)$, or
\[ \begin{split}
  \int_{(\gamma,s_1,\ldots,s_{k-1})\in C\times \Delta^{k-1}}A_1(\bar{\gamma}(s_1))\times\cdots\times A_{k-1}(\bar{\gamma}(s_{k-1}))\times A_{k}(\bar{\gamma}(1)).
\end{split} \]
Here, $\hat\sigma^*(A_1\times\cdots\times A_k)$ is the pullback of the fiberwise space $\phi_1\times\cdots\times\phi_k:A_1\times\cdots\times A_k\to B^k$ by $\hat\sigma$. Then the iterated integral is a fiberwise space over $C\times\Delta^{k-1}$.

Iterated integrals can be extended to matrices whose entries are fiberwise spaces over $B$. For matrices $A_1, A_2$ of fiberwise spaces over $B$, its direct product $A_1\times A_2$ is defined by the following matrix of fiberwise spaces over $B\times B$:
\[ A_1(b_1)\times A_2(b_2)=\Bigl(\sum_k (A_1(b_1))_{ik}\times (A_2(b_2))_{kj} \Bigr) \qquad ((b_1,b_2)\in B\times B).\]
$A_1\times\cdots\times A_k$ etc. can be defined similarly. Iterated integrals of matrices can be defined by the formula above with this convention.

%%%%%%%%%%%%%%%%%%%%%%%%%%%%%%%
\subsection{Parallel transport}\label{ss:parallel}

As a 0-chain $\sigma$ of $PB$, we take a map $\sigma:\{*\}\to \{\alpha\}\subset PB$ that assigns to $*$ a flow-line $\alpha:I\to B$ of an $h$-gradient $-\eta$ such that $\alpha(0)=b_0$. Let $\pbcalM_{p_ip_j}(\xi)_\alpha$ denote the pullback of $\pbcalM_{p_ip_j}(\xi)\to B$ by $\alpha$. Put $\Omega_\alpha=((1-\delta_{ij})\pbcalM_{p_ip_j}(\xi)_\alpha)$ and $\bvec{e}_{p_i}=(\emptyset\ \ \cdots\ \ \emptyset\ \ \{p_i\}\ \ \emptyset\ \ \cdots\ \ \emptyset)$. For critical loci $p,q$ of $\xi$, the set of Z-paths from $p_{\alpha(0)}=p\cap \pi^{-1}(\alpha(0))$ to $q_{\alpha(1)}=q\cap \pi^{-1}(\alpha(1))$ can be written as
\begin{equation}\label{eq:eooe}
 \sum_{k=0}^\infty \int_\sigma \bvec{e}_p\underbrace{\Omega_\alpha\cdots\Omega_\alpha}_k\bvec{e}_q^T. 
\end{equation}
If $|p_{\alpha(0)}|=|q_{\alpha(1)}|$ and $\alpha,\xi$ are generic, then as shown in the following lemma, this is a finite set and can be counted with signs as in (\ref{eq:e(gamma)}).

\begin{Lem}\label{lem:parallel} If $|p_{\alpha(0)}|=|q_{\alpha(1)}|$ and $\alpha,\xi$ are generic, then the following identity holds.
\begin{equation}\label{eq:e0101e}
 n_\alpha(p_{\alpha(0)},q_{\alpha(1)})=\#\sum_{k=0}^\infty\int_\sigma \bvec{e}_p\underbrace{\Omega_\alpha^1\cdots\Omega_\alpha^1}_k\bvec{e}_q^T
\end{equation}
Here, $\Omega_\alpha^1$ is the matrix obtained from $\Omega_\alpha$ by replacing all the $(p_i,p_j)$-entries for $|p_i|\neq |p_j|$ with $\emptyset$. Moreover, $\Phi_\alpha:C_*(\xi_{\alpha(0)})\to C_*(\xi_{\alpha(1)})$ is a chain map.
\end{Lem}
\begin{proof}
In the case $|p_{\alpha(0)}|=|q_{\alpha(1)}|$, the vertical segments of a Z-path are only $i/i$-intersections. Thus the count of (\ref{eq:eooe}) is equal to the right hand side of (\ref{eq:e0101e}). Since there are finitely many parameters at which $i/i$-intersections occur, each of the integrals on the right hand side is a finite sum and we have
\[ \begin{split}
  \mathrm{RHS}&=\bvec{e}_p(\alpha(0))\times\Bigl(\sum_{k=0}^\infty
\sum_{s_1<s_2<\cdots<s_k}\Omega_\alpha^1(\alpha(s_1))\times\cdots\times\Omega_\alpha^1(\alpha(s_k))\Bigr)\times \bvec{e}_q(\alpha(1))^T\\
&\cong\bvec{e}_p(\alpha(0))\times(\bvec{1}+\Omega_\alpha^1(\alpha(t_1)))\times \cdots \times (\bvec{1}+\Omega_\alpha^1(\alpha(t_n)))\times \bvec{e}_q(\alpha(1))^T.
\end{split} \]
Here, $t_1,\ldots,t_n$ are the times at which $i/i$-intersections occur, and $0<t_1<\cdots<t_n<1$. Since the count of each term $\bvec{1}+\Omega_\alpha^1(\alpha(t_j))$ is the elementary matrix corresponding to a handle-slide, the left hand side of the identity is obtained. That each term $\bvec{1}+\Omega_\alpha^1(\alpha(t_j))$ gives a chain map follows from \cite[Lemma~9.3]{Wa2}.
\end{proof}

%%%%%%%%%%%%%%%%%%%%%%%%%%%%%%%
\subsection{Iterated integral expressions for $\bD_p(\xi,\eta)$ and $\bA_q(\xi,\eta)$}\label{ss:iterated-int-Q}

Let $\{p_1,\ldots,p_N\}$ be the set of all critical loci of $f$ numbered so that $f(p_1(b_0))>\cdots>f(p_N(b_0))$ and put $p_0=p_{N+1}=\ell_\infty$. We put 
$X=(-\ell_\infty\ \ \bcalA_{p_1}(\xi)\ \ \cdots\ \ \bcalA_{p_N}(\xi)\ \ \bcalA_\infty(\xi))$, $Y=(\bcalD_\infty(\xi)\ \ \bcalD_{p_1}(\xi)\ \ \cdots\ \ \bcalD_{p_N}(\xi)\ \ -\ell_\infty)$, $\Omega=((1-\delta_{ij})\pbcalM_{p_ip_j}(\xi))$, 
whose entries are compact in the sense of Propositions~\ref{prop:bD}, \ref{prop:bMpq} and Remark~\ref{rem:M-infty}. For a critical point $c$ of $\eta$, let $\sigma:\bcalD_c(\eta)\to B$ be the chain given by the natural map extending the inclusion $\calD_c(\eta)\to B$. Moreover, by parametrizing a flow-line from $c$ to a terminal point by a parameter that is proportional to the height of $h$, we may consider a point of $\bcalD_c(\eta)$ as a point of $PB$ base pointed at $c$. Then we may consider $\sigma:\bcalD_c(\eta)\to B$ as a chain $\bcalD_c(\eta)\to PB$. 

Let $P'B$ denote the space of piecewise smooth paths $\gamma:[0,1]\to B$ such that $\gamma(1)=c$. There is a map $\iota:PB\to P'B$ induced by reversing paths. For the chain $\sigma$ of $PB$, we write $\sigma'=\iota\circ \sigma$. We put
\[ 
   \bD_p^0(\xi,\eta)_\sigma=\sum_{k=0}^\infty \int_\sigma \bvec{e}_p\underbrace{\Omega\cdots\Omega}_kY^T,\qquad \bA_q^0(\xi,\eta)_\sigma=\sum_{k=0}^\infty \int_{\sigma'} X\underbrace{\Omega\cdots\Omega}_k\bvec{e}_q^T.
\]
The sums are disjoint unions and finite. A generic point of $\bD_p^0(\xi,\eta)_\sigma$ represents a Z-path from $p\cap F_c$ over a flow-line of $-\eta$ from $c$. A generic point of $\bA_q^0(\xi,\eta)_\sigma$ represents a Z-path to $q\cap F_c$ over a flow-line of $\eta$ to $c$. 

\begin{LemD}\label{lem:dxooy}
Let $n\geq 1$. For a generic $\xi$, the space $\displaystyle\int_\sigma X\underbrace{\Omega\cdots\Omega}_n\,\t{Y}$ is a disjoint union of finitely many manifolds with corners, and the closure of its codimension 1 stratum is the sum of $S_n,T_n,U_n,V_n$ given as follows.
\begin{equation}\label{eq:dxooy}\small
\begin{split}
  S_n=&\int_\sigma (\partial X)\underbrace{\Omega\cdots\Omega}_{n}\,\t{Y}
  +\sum_{i=1}^n\int_\sigma X\underbrace{\Omega\cdots\Omega}_{i-1}(\partial\Omega)\underbrace{\Omega\cdots\Omega}_{n-i}\,\t{Y}+\int_\sigma X\underbrace{\Omega\cdots\Omega}_n(\partial\t{Y})\\
  T_n=&\int_\sigma (X\times_B\Omega)\underbrace{\Omega\cdots\Omega}_{n-1}\,\t{Y}
  +\sum_{i=1}^{n-1}\int_\sigma X\underbrace{\Omega\cdots\Omega}_{i-1}(\Omega\times_B\Omega)\underbrace{\Omega\cdots\Omega}_{n-i-1}\,\t{Y}\\
  &+\int_\sigma X\underbrace{\Omega\cdots\Omega}_{n-1}(\Omega\times_B\t{Y})\\
  U_n=&X(c)\times \int_\sigma \underbrace{\Omega\cdots\Omega}_n\,Y^T,\qquad V_n=\int_{\partial \sigma}X\underbrace{\Omega\cdots\Omega}_n\,\t{Y}
\end{split}
\end{equation}
\end{LemD}
We will prove Lemma~\ref{lem:dxooy} in \S\ref{ss:transversality}.

Proofs of the following two lemmas are similar to Lemma~\ref{lem:dxooy}.
\begin{LemD}\label{lem:dooy}
Let $n\geq 1$. For a generic $\xi$, the space $\displaystyle\int_\sigma \bvec{e}_p\underbrace{\Omega\cdots\Omega}_n\,\t{Y}$ is a disjoint union of finitely many manifolds with corners, and the closure of its codimension 1 stratum is given by the following formula.
\[ \small\begin{split}
  &\sum_{i=1}^n\int_\sigma \bvec{e}_p\underbrace{\Omega\cdots\Omega}_{i-1}(\partial\Omega)\underbrace{\Omega\cdots\Omega}_{n-i}\,\t{Y}
  +\int_\sigma \bvec{e}_p\underbrace{\Omega\cdots\Omega}_n(\partial\t{Y})+\int_\sigma (\bvec{e}_p\times_c \Omega)\underbrace{\Omega\cdots\Omega}_{n-1}Y^T\\
  &+\sum_{i=1}^{n-1}\int_\sigma \bvec{e}_p\underbrace{\Omega\cdots\Omega}_{i-1}(\Omega\times_B\Omega)\underbrace{\Omega\cdots\Omega}_{n-i-1}\,\t{Y}
  +\int_\sigma \bvec{e}_p\underbrace{\Omega\cdots\Omega}_{n-1}(\Omega\times_B\t{Y})\\
  &+\int_{\partial \sigma}\bvec{e}_p\underbrace{\Omega\cdots\Omega}_n\,\t{Y}.
\end{split}\]
Let $S_n'$ be the first row, and let $T_n'$ be the second row.
\end{LemD}

\begin{LemD}\label{lem:dxoo}
Let $n\geq 1$. For a generic $\xi$, the space $\displaystyle\int_{\sigma'} X\underbrace{\Omega\cdots\Omega}_n\,\t{\bvec{e}_q}$ is a disjoint union of finitely many manifolds with corners, and the closure of its codimension 1 stratum is given by the following formula.
\[ \small\begin{split}
  &\int_{\sigma'}(\partial X)\underbrace{\Omega\cdots\Omega}_n\t{\bvec{e}_q}
  +\sum_{i=1}^n\int_{\sigma'} X\underbrace{\Omega\cdots\Omega}_{i-1}(\partial\Omega)\underbrace{\Omega\cdots\Omega}_{n-i}\,\t{\bvec{e}_q}\\
  &+\int_{\sigma'}(X\times_B\Omega)\underbrace{\Omega\cdots\Omega}_{n-1}\t{\bvec{e}_q}
  +\sum_{i=1}^{n-1}\int_{\sigma'} X\underbrace{\Omega\cdots\Omega}_{i-1}(\Omega\times_B\Omega)\underbrace{\Omega\cdots\Omega}_{n-i-1}\,\t{\bvec{e}_q}\\
  &+\int_{\sigma'}X\underbrace{\Omega\cdots\Omega}_{n-1}(\Omega\times_c\t{\bvec{e}_q})+\int_{\partial \sigma'}X\underbrace{\Omega\cdots\Omega}_n\,\t{\bvec{e}_q}.
\end{split}\]
Let $S_n''$ be the first row, and let $T_n''$ be the second row.
\end{LemD}

\begin{LemD}\label{lem:dxy}
For a generic $\xi$, the spaces $\displaystyle\int_\sigma XY^T$, $\displaystyle\int_\sigma \bvec{e}_pY^T$, $\displaystyle\int_{\sigma'} X\t{\bvec{e}_q}$ are disjoint unions of finitely many manifolds with corners, and the closure of their codimension 1 strata are given by the following formulas.
\[ \small\begin{split}
  \partial\int_\sigma XY^T&=\int_\sigma ((\partial X)Y^T + X(\partial Y^T)) + \int_\sigma X\times_B Y^T + X(c)\times\int_\sigma Y^T + \int_{\partial \sigma} XY^T \\
  \partial\int_\sigma \bvec{e}_pY^T&=\int_\sigma\bvec{e}_p(\partial Y^T) + \bvec{e}_p\times_c Y^T + \int_{\partial \sigma}\bvec{e}_p Y^T\\
  \partial\int_{\sigma'} X\t{\bvec{e}_q}&=\int_{\sigma'}(\partial X)\bvec{e}_q^T + X\times_c\bvec{e}_q^T + \int_{\partial \sigma'}X\bvec{e}_q^T
\end{split}\]
We denote the four terms in the first row by $S_0, T_0, U_0, V_0$, the first two terms in the second row by $S_0',T_0'$, and the first two terms in the third row by $S_0'',T_0''$.
\end{LemD}

The following proposition follows from Propositions~\ref{prop:bM2_mfd}, \ref{prop:bD}, \ref{prop:bMpq}.

\begin{Prop}\label{prop:dw=ww}
There are natural stratification preserving diffeomorphisms
\[ \begin{split}
&\partial X\cong  X\times_B \Omega,\quad 
\partial\,Y^T\cong  \Omega\times_B Y^T,\\ 
&\partial \Omega\cong  \Omega\times_B \Omega,\quad
\partial \ibcalM_2(\xi)\cong \overline{s}_\xi + X_0\times_B Y_0^T.
\end{split}\]
They induce strata preserving diffeomorphisms $S_n\cong T_{n+1}$, $S_n'\cong T_{n+1}'$, $S_n''\cong T_{n+1}''$ for $n\geq 0$. 
\end{Prop}

\begin{proof}[Proof of Proposition~\ref{prop:bM} assuming Lemmas~\ref{lem:dxooy}--\ref{lem:dxy} and Proposition~\ref{prop:dw=ww}] 
When $\sigma$ is the fundamental cycle of $B$, put
\[ \bacalM_2(\xi,\eta)_\sigma=\Bigl\{\ibcalM_2(\xi)+\sum_{n=0}^\infty \int_\sigma X \underbrace{\Omega\cdots\Omega}_n\,\t{Y}\Bigr\}\Bigl/\sim. \]
Here, we identify the strata by the diffeomorphisms $S_n\cong T_{n+1}$ for $n\geq 0$ of Proposition~\ref{prop:dw=ww}. Since any Z-path can be shrinked to a point by reducing the number of vertical segments, this gives a natural path-connected compactification of $\acalM_2(\xi,\eta)$. 
\end{proof}

Let $\bD_p(\xi,\eta)_\sigma$ (resp. $\bA_q(\xi,\eta)_\sigma$) denote the pullback of the fiberwise spaces $\bD_p(\xi,\eta)$ (resp. $\bA_q(\xi,\eta)$) over $B$ by $\sigma$.
The following proposition can be proved by an argument similar to the proof of Proposition~\ref{prop:bM}.
\begin{Prop}\label{prop:Dp}
There are canonical bijections $\bD_p(\xi,\eta)_\sigma\approx\bD_p^0(\xi,\eta)_\sigma/{\sim}$ and $\bA_q(\xi,\eta)_\sigma\approx\bA_q^0(\xi,\eta)_\sigma/{\sim}$. 
Here, we identify the strata by the diffeomorphisms $S_n'\cong T_{n+1}'$, $S_n''\cong T_{n+1}''$ for $n\geq 0$ of Proposition~\ref{prop:dw=ww}. 
\end{Prop}

%%%%%%%%%%%%%%%%%%%%%%%%%%%%%%%
\subsection{Orientations of $\bD_p(\xi,\eta)_\sigma$ and $\bA_q(\xi,\eta)_\sigma$}\label{ss:ori}

We shall give an orientation convention for the spaces $\bD_p(\xi,\eta)_\sigma$, $\bA_q(\xi,\eta)_\sigma$ of Z-paths on the chain $\sigma:\bcalD_c(\eta)\to B$. 
Let $\gamma,\gamma'$ be points of codimension 0 strata of $\bD_p(\xi,\eta)_\sigma,\bA_q(\xi,\eta)_\sigma$ respectively.

\subsubsection{Nondegenerate strata of $\bD_p(\xi,\eta)_\sigma$ and $\bA_q(\xi,\eta)_\sigma$}
Suppose that the vertical segments of $\gamma$ between critical loci are only $i/i$-intersections.
If $\sigma_1,\ldots,\sigma_r$ are the $i/i$-intersections included in $\gamma$ and if the terminal vertical segment of $\gamma$ is of $\calD_{p'}(\xi)$, then we define the orientation of $\bD_p(\xi,\eta)_\sigma$ at $\gamma$ as follows. 
\[ o(\bD_p(\xi,\eta)_\sigma)_\gamma = \ve(\sigma_1)\ve(\sigma_2)\cdots\ve(\sigma_r)\,o(\calD_{p'}(\xi))_{\bar{\gamma}(1)}, \]
where $\ve(\sigma_i)$ is the same as that given in \S\ref{ss:count-Z-path}.
Suppose that the vertical segments of $\gamma'$ between critical loci are only $i/i$-intersections. If $\sigma_1,\ldots,\sigma_r$ are the $i/i$-intersections in $\gamma'$ and the initial vertical segment of $\gamma'$ is of $\calA_{q'}(\xi)$, then we define the orientation of $\bA_q(\xi,\eta)_\sigma$ at $\gamma'$ as follows.
\[ o(\bA_q(\xi,\eta)_\sigma)_{\gamma'} = \ve(\sigma_1)\ve(\sigma_2)\cdots\ve(\sigma_r)\,o(\calA_{q'}(\xi))_{\bar{\gamma}'(0)}.\]
With the conventions given here, one can check that the orientations on $S_n'\cong T_{n+1}'$ and $S_n''\cong T_{n+1}''$ in the gluings in Proposition~\ref{prop:Dp} are consistent. 

\subsubsection{Degenerate strata of $\bD_p(\xi,\eta)_\sigma$ and $\bA_q(\xi,\eta)_\sigma$}
Suppose that the vertical segments of $\gamma$ between critical loci consist of one $i+1/i$-intersection $\tau$ and $i/i$-intersections $\sigma_1,\ldots,\sigma_r$. Suppose that the terminal vertical segment of $\gamma$ is of $\tau'=\calD_{p'}(\xi)$. Then we define the orientation of $\bD_p(\xi,\eta)_\sigma$ at $\gamma$ by
\begin{equation}\label{eq:o(degenerate)}
 o(\bD_p(\xi,\eta)_\sigma)_\gamma = \ve(\sigma_1)\ve(\sigma_2)\cdots\ve(\sigma_r)\,o(\tau, \tau') 
\end{equation}
for some orientation $o(\tau,\tau')$ determined by the pair $\tau$, $\calD_{p'}(\xi)$. We define $o(\tau,\tau')$ as follows.

Suppose that $\gamma$ goes within $\alpha^*E$ for a flow-line $\alpha:I\to B$ of $-\eta$ with $\alpha(0)=c$, and that the vertical segments $\tau, \tau'$ are located in the fibers over $u_0,v_0\in \mathrm{Im}\,\alpha$. We consider that $\tau$ depends smoothly on a parameter $u$ in a neighborhood $U$ of $u_0\in B$ and write $\tau=\tau(u)$. Then $K=\bigcup_{u\in U}\tau(u)$ is a subbundle of $\pi^{-1}U\to U$, which has a local parametrization $(u,z)\mapsto K(u,z)$, $(u,z)\in U\times (-\ve,\ve)$. Similarly, we write $\tau'=\tau'(v)$ for a parameter $v$ in a neighborhood $V$ of $v_0\in B$, and we obtain a subbundle $L=\bigcup_{v\in V}\tau'(v)\subset \pi^{-1}V$, which has a local parametrization $(v,w)\mapsto L(v,w)$, $(v,w)\in V\times (-\ve,\ve)$. Since $\gamma$ is a point of a codimension 0 stratum of $\bD_p(\xi,\eta)_\sigma$, two generic points $u,v$ on $\mathrm{Im}\,\alpha$ are related by $v=\Phi_{-\eta}^T(u)$ for $T>0$. We can take, as a neighborhood of $\gamma$ in $\bD_p(\xi,\eta)_\sigma$, the image of the embedding $\mu: U\times (-\ve,\ve)\times (T_0-\ve,T_0+\ve)\times (-\ve,\ve)\to \pi^{-1}U\times \pi^{-1}V$ given by
\[ \mu(u,z,T,w) = (u,K(u,z))\times (\Phi_{-\eta}^T(u), L(\Phi_{-\eta}^T(u),w)), \]
and the orientation of $\mathrm{Im}\,\mu$ gives $o(\tau,\tau')$. The Jacobian matrix $J\mu$ of $\mu$ can be modified by elementary column operations into the following form.
\[ \left(\begin{smallmatrix}
  \mathbf{1} & 0& O &  O\\
  O & 0 & \ddx{z}K(u,z)  & O\\
  \ddx{u}\Phi_{-\eta}^T(u) & -\eta & O &  O\\
  O & 0 & O & \ddx{w}L(\Phi_{-\eta}^T(u),w)
  \end{smallmatrix}\right) \]
where $\ddx{u}K(u,z)$ etc. denotes the Jacobian matrix of $K(u,z)$ etc. with respect to $u$. The column vectors of the matrix correspond to tangent vectors of $\mathrm{Im}\,\mu$. The left half of the matrix corresponds to a basis of a tangent space of $\calM_2(\eta)$, and the right half corresponds to bases of tangent spaces of $K,L$. This gives a direct sum decomposition of $T_{(u,z,T,w)}\mathrm{Im}\,\mu$. 

Based on this observation, we define the coorientation $o^*(\tau,\tau')=*\,o(\tau,\tau')$ of $\mathrm{Im}\,\mu$ in $\pi^{-1}U\times \pi^{-1}V\approx (U\times F)\times (V\times F)=(U\times V)\times (F\times F)$ by the tensor product
\begin{equation}\label{eq:o(M_2)}
 o^*(\tau,\tau')=o^*_{U\times V}(\calM_2(\eta)) \otimes (o^*_E(K)\wedge o^*_E(L)).
\end{equation}
Since $\calM_2(\eta)$ is the image of the embedding $\varphi:B\times (0,\infty)\to B\times B$, $\varphi(u,T)=(u,\Phi_{-\eta}^T(u))$, we may define
\[ \begin{split}
  o(\calM_2(\eta))_{(u,v)}&=d\varphi_*(o(B)_u\wedge dT),\quad o^*_{U\times V}(\calM_2(\eta))_{(u,v)}=*\,o(\calM_2(\eta))_{(u,v)},
\end{split} \]
where $*$ is the Hodge star operator and $d\varphi_*$ is the map induced by $d\varphi$ under the identification $T(B\times (0,\infty))=T^*(B\times (0,\infty))$ by an orthonormal basis.
If $\tau$ is a flow-line between critical loci $t,t'$, then $o^*_E(K), o^*_E(L)$ can be given as follows.
\[ o^*_E(K)_b=o^*_E(\calD_t(\xi))_b\wedge o^*_E(\calA_{t'}(\xi))_b, \quad 
o^*_E(L)=o^*_E(\calD_{p'}(\xi)). \]
Similarly, $o(\bA_q(\xi,\eta)_\sigma)_{\gamma'}$, $o(\calH_{pq}^0(\xi,\eta)_\sigma)_\gamma$ can be defined by the same formula as (\ref{eq:o(M_2)}), letting $\tau'=\calA_{q'}(\xi)$ for $\bA_q(\xi,\eta)_\sigma$, $\tau'=\calA_{q'}(\xi)\times_B\calD_{p'}(\xi)$, $o^*_E(L)_{(x,y)}=o^*_E(\calA_{q'}(\xi))_x\wedge o^*_E(\calD_{p'}(\xi))_y$ for $\calH_{pq}^0(\xi,\eta)_\sigma$, respectively.

\subsubsection{Preliminary lemmas concerning the orientation of $\partial\bD_p(\xi,\eta)_\sigma$ and $\partial\bA_q(\xi,\eta)_\sigma$}
According to (\ref{eq:o(M_2)}), the coorientation of the stratum of Z-paths induced on a path in $\partial \bcalM_2(\eta)$ is given by
the tensor product of the coorientation of $\partial\bcalM_2(\eta)$ and $o^*_E(K)\wedge o^*_E(L)$. Since the natural map $\partial\bcalM_2(\eta)\to B\times B$ is locally an immersion near a generic point, the coorientation of $\partial\bcalM_2(\eta)$ in $B\times B$ at a generic point makes sense and it is induced from that of $\bcalM_2(\eta)$ as follows.

\begin{Lem}[{\cite[Lemma~5.4]{Wa2}}]\label{lem:o(dM_2)}
Suppose that a flow-line of $\partial\bcalM_2(\eta)$ between $u,v\in B$ passes through one critical point $r\in P_*(\eta)$. Then we have
\[ o^*_{B\times B}(\partial\bcalM_2(\eta))_{(u,v)} = (-1)^{(|r|+1)\dim{B}}o^*_B(\calA_r(\eta))_u\wedge o^*_B(\calD_r(\eta))_v. \]
\end{Lem}
In particular, when $|r|=\dim{B}$, the sign is $(-1)^{(|r|+1)\dim{B}}=1$. For a critical point $m$ of $\eta$, let $\bcalD_m(\eta),\bcalA_m(\eta)$ denote the compactifications of $\calD_m(\eta), \calA_m(\eta)$ obtained by adding singular flow-lines passing through several critical points (e.g., \cite{BH, Wa3}). If a flow-line $\gamma\in \bcalD_m(\eta)$ of $-\eta$ that starts from $m$ has one singular point at a critical point $r$ ($|r|=|m|-1$), then $\gamma$ is determined uniquely by a point $b$ on a flow-line between $m$ and $r$, and the terminal point $a=\gamma(1)$. Similarly, if a flow-line $\gamma\in \bcalA_m(\eta)$ of $-\eta$ that ends at $m$ has one singular point at a critical point $r$ ($|r|=|m|+1$), then $\gamma$ is determined uniquely by a point $b$ on a flow-line between $r$ and $m$, and the initial point $a=\gamma(0)$. The orientations induced on the strata of such singular flow-lines are as follows.
\begin{Lem}[{\cite[Lemma~5.1, 5.2]{Wa2}}]\label{lem:o(dD)}
\[ \begin{split}
  o(\partial \bcalD_m(\eta))_{\gamma}&=(-1)^{|r|}\ve(m,r)_b\,o(\calD_r(\eta))_a,\quad o(\partial \bcalA_m(\eta))_{\gamma}=-\ve(r,m)_b\,o(\calA_r(\eta))_a.
\end{split} \]
\end{Lem}

%%%%%%%%%%%%%%%%%%%%%%%%%%%%%%%
\subsection{Boundaries of chains of Z-paths}\label{ss:boudary}

In the following, we assume that $(f,\xi,\eta)$ satisfies Assumptions~\ref{hyp:eta} and \ref{hyp:cv-const}. Here, we let $c$ to be the maximal point of the Morse function $h:B\to \R$. Here, we denote the chains $\ev_2:\bD_p(\xi,\eta)\to E^\infty$, $\overline{\ev}_1:\bA_q(\xi,\eta)\to E^\infty$ by $\bD_p(\xi,\eta)$, $\bA_q(\xi,\eta)$, abusing the notation. Let $C_*=C_*(\xi_c)=\Z^{P_*}$ be the Morse complex for the $v$-gradient $\xi_c$ and let $g:C_*\to C_{*+1}$ be a combinatorial propagator for $C_*$. Let $\overline{C}_*$ be the graded $\Z$-module defined by $\overline{C}_k=C_k$ for $k\neq 0,4$, $\overline{C}_4=C_4\oplus\langle \ell_\infty^+\rangle$, and $\overline{C}_0=C_0\oplus \langle \ell_\infty^-\rangle$. We define a $\Z$-linear map $\overline{\partial}:\overline{C}_*\to \overline{C}_{*-1}$ by 
\[ \overline{\partial}p=\left\{\begin{array}{ll}
\partial p & \mbox{if $p\in P_*$ and $|p|\neq 1$},\\
\partial p+\#\calM'_{p\infty}(\xi_c)\,\ell_\infty^- & \mbox{if $|p|=1$}\\
\sum_{q\in P_3}\#\calM'_{\infty q}(\xi_c)\,q & \mbox{if $p=\ell_\infty^+$}\\
0 & \mbox{if $p=\ell_\infty^-$}
\end{array}\right. \] 
Then $(\overline{C}_*,\overline{\partial})$ can be considered as the Morse complex of a `singular' gradient on $S^4$ and one can show that $H_*(\overline{C}_*,\overline{\partial})\cong H_*(S^4;\Z)\cong \Z\oplus \Z$ by perturbing the singular gradient on $S^4$ slightly. We put $\overline{P}_*=P_*\cup\{\ell_\infty^-,\ell_\infty^+\}$, $\bD_{\ell_\infty^+}(\xi,\eta)_\sigma=\bcalD_\infty(\xi)$, $\bA_{\ell_\infty^-}(\xi,\eta)_\sigma=\bcalA_\infty(\xi)$\footnote{Modulo degenerate chains, this definition of $\bD_{\ell_\infty^+}(\xi,\eta)_\sigma$ (resp. $\bA_{\ell_\infty^-}(\xi,\eta)_\sigma$) is the same as the space of Z-paths from $\ell^+_\infty$ (resp. the space of inverse Z-paths to $\ell^-_\infty$).}, and $\bD_{\ell_\infty^-}(\xi,\eta)_\sigma=\bA_{\ell_\infty^+}(\xi,\eta)_\sigma=\ell_\infty$. The following proposition will be used in \S\ref{ss:dg}.

\begin{Prop}\label{prop:dD}
Let $\sigma$ be the fundamental cycle of $B$ and let $p$ be an element of $\overline{P}_*$. Then the following identities for chains in $E^\infty$ hold modulo degenerate chains. 
\[
\partial\bD_p(\xi,\eta)_\sigma=\displaystyle\sum_{r\in \overline{P}_{|p|-1}} \langle \overline{\partial} p,r\rangle\, \bD_r(\xi,\eta)_\sigma,\quad 
\partial\bA_q(\xi,\eta)_\sigma=\displaystyle\sum_{r\in \overline{P}_{|q|+1}} \bA_r(\xi,\eta)_\sigma \langle\overline{\partial} r,q\rangle
\]
\end{Prop}
\begin{proof}
In the iterated integral description of $\bD_p$ of Proposition~\ref{prop:Dp}, the first two terms in $S_n'$ (Lemma~\ref{lem:dooy}) and the two terms in $T_{n+1}'$ cancel each other by Proposition~\ref{prop:dw=ww}. The following holds.
\[ \partial\bD_p(\xi,\eta)_\sigma
=\sum_{k=0}^\infty\int_\sigma (\bvec{e}_p\times_c \Omega)\underbrace{\Omega\cdots\Omega}_{k-1}Y^T
+\sum_{k=0}^\infty \int_{\partial\sigma}\bvec{e}_p\underbrace{\Omega\cdots\Omega}_kY^T
\]
Here, the signs are determined by using (\ref{eq:o(degenerate)}), (\ref{eq:o(M_2)}), and Lemma~\ref{lem:o(dM_2)} (the case $|r|=\dim{B}$ in the notation of Lemma~\ref{lem:o(dM_2)}). The first term in the right hand side is $\displaystyle\sum_r \langle \overline{\partial} p,r\rangle\, \bD_r(\xi,\eta)_\sigma$, and by Lemma~\ref{lem:o(dD)}, the second term is fibered over $\displaystyle\partial\bcalD_c(\eta)=(-1)^{|c|}\sum_{c'} \calM'_{cc'}(\eta)\times \bcalD_{c'}(\eta)$. Thus, considered modulo degenerate chains, 
\[ \begin{split}
  &\sum_{k=0}^\infty \int_{\partial\sigma}\bvec{e}_p\underbrace{\Omega\cdots\Omega}_kY^T
  =(-1)^{|c|} \sum_{c'} \sum_{k_1,k_2\geq 0} \int_{\calM'_{cc'}}\bvec{e}_p\underbrace{\Omega\cdots\Omega}_{k_1}\times \int_{\bcalD_{c'}}\underbrace{\Omega\cdots\Omega}_{k_2}Y^T\\
  &=(-1)^{|c|} \sum_{c'}\sum_{{p'}\atop{|p'|=|p|}} \Bigl(\sum_{k_1\geq 0}\int_{\calM'_{cc'}}\bvec{e}_p\underbrace{\Omega\cdots\Omega}_{k_1}\bvec{e}_{p'}^T\Bigr)\times \Bigl(\sum_{k_2\geq 0}\int_{\bcalD_{c'}}\bvec{e}_{p'}\underbrace{\Omega\cdots\Omega}_{k_2}Y^T\Bigr)\\
  &=(-1)^{|c|} \sum_{c'}\langle \partial_\eta c,c'\rangle  \Bigl(\sum_{k_2\geq 0}\int_{\bcalD_{c'}}\bvec{e}_p\underbrace{\Omega\cdots\Omega}_{k_2}Y^T\Bigr) \\
  &=(-1)^{|c|} \sum_{c'}\langle \partial_\eta c,c'\rangle  \bD_p(\xi,\eta)_{\bcalD_{c'}}=(-1)^{|c|}\bD_p(\xi,\eta)_{\partial \sigma}=0.
\end{split} \]
Here, we used Lemma~\ref{lem:parallel} and $\Phi_\alpha=1$ of Assumption~\ref{hyp:cv-const} (2) in the third equality. The proof for $\partial\bA_q$ is similar.
\end{proof}

Let $\calH_{pq}^0(\xi,\eta)_\sigma=\bA_q(\xi,\eta)_\sigma\times_B \bD_p(\xi,\eta)_\sigma$ be the pullback of the fiberwise space $\calH_{pq}^0(\xi,\eta)$ over $B$ by $\sigma:\bcalD_c(\eta)\to B$.
By an argument similar to the proof of Proposition~\ref{prop:dD}, the following proposition is obtained.
\begin{Prop}\label{prop:dH}
If $\partial \sigma=0$, then $\partial \calH_{pq}^0(\xi,\eta)_\sigma$, considered as a chain of $E^\infty\times_B E^\infty$, is given by the following formula modulo degenerate chains.
\[\begin{split}
 &\partial\bA_q(\xi,\eta)_\sigma\times_B \bD_p(\xi,\eta)_\sigma+\bA_q(\xi,\eta)_\sigma\times_B \partial\bD_p(\xi,\eta)_\sigma\\
 &=\sum_{{r\in \overline{P}_*}\atop{|r|=|q|+1}}\langle \overline{\partial} r,q\rangle \bA_r(\xi,\eta)_\sigma\times_B \bD_p(\xi,\eta)_\sigma
 +\sum_{{s\in \overline{P}_*}\atop{|s|=|p|-1}}\langle \overline{\partial} p,s\rangle \bA_q(\xi,\eta)_\sigma\times_B \bD_s(\xi,\eta)_\sigma
\end{split}
\]
\end{Prop}

Let $g$ be a combinatorial propagator for $C_*(\xi_c)$ and let $\overline{g}:\overline{C}_*\to \overline{C}_{*+1}$ be the $\Z$-linear map defined by $\overline{g}(p)=g(p)$ for $p\in P_*$ and $\overline{g}(\ell_\infty^-)=0$. Then $\overline{\partial}\overline{g}+\overline{g}\overline{\partial}$ is not the identity. More precisely, for $|p|=1,2,3$, $(\overline{\partial}\overline{g}+\overline{g}\overline{\partial})(p)=(\partial g+g\partial)(p)=p$, and for $|p|=0$, 
\[ (\overline{\partial}\overline{g}+\overline{g}\overline{\partial})(p)=\overline{\partial}\overline{g}(p)
 =\left\{\begin{array}{ll}
 p+\ell_\infty^- & (p\neq \ell_\infty^-)\\
 0 & (p=\ell_\infty^-)
 \end{array}\right.
\]
Indeed, for $p\neq \ell_\infty^-$, $|p|=0$, we have $\overline{\partial}\overline{g}(p)=\overline{\partial}g(p)$, where $g(p)$ is a 1-chain with $\partial g(p)=p$. Hence $g(p)$ has another end at $\ell_\infty^-$ and $\overline{\partial}g(p)=p-(-\ell_\infty^-)=p+\ell_\infty^-$ by the orientation convention of Remark~\ref{rem:M-infty}. For $|p|=4$, 
\[ (\overline{\partial}\overline{g}+\overline{g}\overline{\partial})(p)=g\overline{\partial}(p)
 =\left\{\begin{array}{ll}
 p & (p\neq \ell_\infty^+)\\
 \sum_{p'\in P_4}p' & (p=\ell_\infty^+)
\end{array}\right.
 \]
where the last identity for $p=\ell_\infty^+$ can be obtained by the identities $\overline{\partial}(-\ell_\infty^+ + \sum_{p'\in P_4}p')=0$ and $g\overline{\partial}(\ell_\infty^+)=\sum_{p'\in P_4}g\partial(p')=\sum_{p'\in P_4} p'$, by the orientation convention of Remark~\ref{rem:M-infty} again. 

Let $\Tr_g:S_*(E^\infty\times_B E^\infty)\otimes C_*(\xi_c)^{\otimes 2}\to S_*(E^\infty\times_B E^\infty)$ be defined for $\tau\in S_*(E^\infty\times_B E^\infty)$, $x,y\in P_*(\xi_c)$ by 
\[ \Tr_g(\tau\otimes x\otimes y)=-\langle g(x),y\rangle\tau. \]
\begin{Lem}\label{lem:dtheta-0} Let $\sigma$ be the fundamental cycle of $B$ and let $\theta^0$ be the chain of $E^\infty\times_BE^\infty$ defined by
\[ \theta^0(\xi,\eta)=\displaystyle M_0(\xi)+\Tr_g\Bigl(\sum_{p,q\in P_*(\xi_c)}\calH_{pq}^0(\xi,\eta)_\sigma\, q\otimes p\Bigr),\] 
 where $M_0(\xi)$ is the natural map $\ibcalM_2(\xi)\to E^\infty\times_B E^\infty$ and the sum is taken for $p,q\in P_*(\xi_c)$ such that $|p|=|q|+1$. Then the following identity holds modulo degenerate chains.
\[ \partial \theta^0(\xi,\eta)=\pm \Delta_E - E^\infty\times_B \ell_\infty - \ell_\infty\times_B E^\infty. \]
\end{Lem}
\begin{proof} By Proposition~\ref{prop:dH}, $\Tr_g\Bigl(\sum_{p,q}\partial \calH_{pq}^0(\xi,\eta)_\sigma\, q\otimes p\Bigr)$ can be rewritten as
\begin{equation}\label{eq:Tr(dH)}
 \begin{split}
  &\Tr_g\Bigl(\sum_{p,q\in P_*}\sum_{r\in \overline{P}_*}\bA_r\times_B\bD_p\,\langle\overline{\partial} r,q\rangle q\otimes p\Bigr)
  + \Tr_g\Bigl(\sum_{p,q\in P_*}\sum_{s\in \overline{P}_*}\bA_q\times_B\bD_s\,\langle\overline{\partial} p,s\rangle q\otimes p \Bigr)\\
  =&\Tr_{\overline{g}}\Bigl(\sum_{p,q\in \overline{P}_*}\sum_{r\in \overline{P}_*}\bA_r\times_B\bD_p\,\langle\overline{\partial} r,q\rangle q\otimes p\Bigr)
  + \Tr_{\overline{g}}\Bigl(\sum_{p,q\in \overline{P}_*}\sum_{s\in \overline{P}_*}\bA_q\times_B\bD_s\,\langle\overline{\partial} p,s\rangle q\otimes p \Bigr)\\
  =&\Tr_{\overline{g}\overline{\partial}+\overline{\partial} \overline{g}}\Bigl(\sum_{p,r\in\overline{P}_*}\bA_r\times_B\bD_p\,r\otimes p\Bigr)
  =-\sum_{p,r\in\overline{P}_*}\langle (\overline{g}\overline{\partial}+\overline{\partial} \overline{g})(r),p\rangle\,\bA_r\times_B\bD_p,
\end{split} 
\end{equation}
where $\bA_*=\bA_*(\xi,\eta)_\sigma$ and $\bD_*=\bD_*(\xi,\eta)_\sigma$. For $1\leq |p|=|r|\leq 3$, we have $\langle (\overline{g}\overline{\partial}+\overline{\partial} \overline{g})(r),p\rangle=\langle r,p\rangle$. For $|p|=|r|=0$, we have $(\overline{g}\overline{\partial}+\overline{\partial} \overline{g})(r)=r+\ell_\infty^-$ if $r\neq \ell_\infty^-$, and $(\overline{g}\overline{\partial}+\overline{\partial} \overline{g})(r)=0$ if $r=\ell_\infty^-$. Hence the sum of terms for $|p|=|r|=0$ in (\ref{eq:Tr(dH)}) is
\[ \begin{split}
  &-\sum_{r\in P_0}\sum_{p\in \overline{P}_0}\langle r+\ell_\infty^-,\,p\rangle\, \bA_r\times_B \bD_p
  =-\sum_{r\in P_0}(\bA_r\times_B\bD_r + \bA_r\times_B \bD_{\ell_\infty^-})\\
  =&-\sum_{r\in P_0}\bA_r\times_B\bD_r - \Bigl(\sum_{r\in P_0}\bA_r\Bigr)\times_B \ell_\infty.
\end{split} \]
For $|p|=|r|=4$, we have $(\overline{g}\overline{\partial}+\overline{\partial} \overline{g})(r)=r$ if $r\neq \ell_\infty^+$, and $(\overline{g}\overline{\partial}+\overline{\partial} \overline{g})(r)=\sum_{p'\in P_4}p'$ if $r=\ell_\infty^+$. Hence the sum of terms for $|p|=|r|=4$ in (\ref{eq:Tr(dH)}) is
\[ \begin{split}
  &-\sum_{r\in P_4}\sum_{p\in \overline{P}_4}\langle r,\,p\rangle\, \bA_r\times_B \bD_p - \sum_{p\in\overline{P}_4}\Bigl\langle \sum_{p'\in P_4} p',\,p\Bigr\rangle\,\bA_{\ell_\infty^+}\times_B \bD_p\\
  =&-\sum_{r\in P_4}\bA_r\times_B \bD_r - \ell_\infty\times_B\Bigl(\sum_{p\in P_4}\bD_p\Bigr).
\end{split} \]
Hence (\ref{eq:Tr(dH)}) is equal to
\begin{equation}\label{eq:XY}
 -\sum_{r\in P_*}\bA_r\times_B \bD_r - \Bigl(\sum_{r\in P_0}\bA_r\Bigr)\times_B \ell_\infty - \ell_\infty\times_B\Bigl(\sum_{p\in P_4}\bD_p\Bigr). 
\end{equation}
Here, we decompose as $\sigma=\sum_i\sigma_i$, where $\sigma_i$ is the restriction of $\sigma$ on the closure of a codimension 0 stratum in the conic stratification of $\bcalD_{b_0}(\eta)$ of Lemma~\ref{lem:i/i-str}. Let $b_i$ be a base point of $\sigma_i$ in the interior of the image, and let $\psi_i:C_*(\xi_0)\to C_*(\xi_{b_i})$ denote the parallel transport defined by Z-paths on a flow-line of $-\eta$ between $b_0$ and $b_i$. Let $(X_0)_{\sigma_i},(Y_0)_{\sigma_i}$ be the pullbacks of the fiberwise spaces $X_0, Y_0$ over $B$ respectively by $\sigma_i$. Then the first term in (\ref{eq:XY}) above can be rewritten as follows.
\[ \begin{split}
  &\sum_i\Tr_{\psi_i\circ\mathrm{id}\circ\psi_i^{-1}}\Bigl(\sum_{p_i,q_i}\bcalA_{q_i}(\xi,\eta)_{\sigma_i}\times_B\bcalD_{p_i}(\xi,\eta)_{\sigma_i}\,q_i\otimes p_i\Bigr)\\
  &= \sum_i\Tr_{\mathrm{id}}\Bigl(\sum_{p_i,q_i}\bcalA_{q_i}(\xi,\eta)_{\sigma_i}\times_B\bcalD_{p_i}(\xi,\eta)_{\sigma_i}\,q_i\otimes p_i\Bigr)\\
  &=-\sum_i (X_0)_{\sigma_i}\times_B (Y_0)_{\sigma_i}^T=-X_0\times_B Y_0^{T}.
\end{split} \]
Then we use the identity 
$ \partial \ibcalM_2(\xi)\cong -s_\xi + X_0\times_B Y_0^{T} - \bcalA_\infty(\xi)\times_B\ell_\infty - \ell_\infty\times_B \bcalD_\infty(\xi)$
of Proposition~\ref{prop:dw=ww}, where the signs are correct by Lemma~\ref{lem:o(dM_2)} and the convention (see also the proof of \cite[Proposition~5.5]{Wa3}), and obtain
$\partial\theta^0(\xi,\eta)=\pm\Delta_E - \Bigl(\sum_{r\in \overline{P}_0}\bA_r\Bigr)\times_B \ell_\infty - \ell_\infty\times_B\Bigl(\sum_{p\in \overline{P}_4}\bD_p\Bigr)=\pm\Delta_E - E^\infty\times_B \ell_\infty - \ell_\infty\times_B E^\infty$.
This completes the proof.
\end{proof}

%%%%%%%%%%%%%%%%%%%%%%%%%%%%%%%5
\subsection{Proof of Theorem~\ref{thm:P-propagator}}

By the definition of $\theta_\mathrm{Z}(\xi,\eta)$, we have
\[ \partial \theta_\mathrm{Z}(\xi,\eta)=\partial M(\xi,\eta) + \Tr_g\Bigl( \sum_{p,q} \partial H_{pq}(\xi,\eta)_\sigma q\otimes p \Bigr). \]
We define $\overline{L}_{p_iq_i}(\xi)=\bcalD_{p_i}(\xi)\times_{E^\infty}\bcalA_{q_i}(\xi)$ with respect to the natural maps $\bcalD_{p_i}(\xi)\to E^\infty$ and $\bcalA_{q_i}(\xi)\to E^\infty$.
By $\partial M(\xi,\eta)=-\overline{s}_\xi+X_0\times_B Y_0^{T}$ (Proposition~\ref{prop:dw=ww}) and the argument in the proof of Lemma~\ref{lem:dtheta-0}, $\partial \theta_\mathrm{Z}(\xi,\eta)$ can be rewritten as follows.
\[ \begin{split}
  &-\overline{s}_\xi+X_0\times_B Y_0^{T} -X_0\times_B Y_0^{T}
  +\sum_i\Tr_{\psi_i g \psi_i^{-1}}\Bigl(\sum_{p_i,q_i} \partial H_{p_iq_i}(\xi,\eta)_{\sigma_i}q_i\otimes p_i\Bigr)\\
  &=-\overline{s}_\xi-\sum_i\Tr_{\psi_i g \psi_i^{-1}}\Bigl(\sum_{p_i,q_i}ST^v(\overline{L}_{p_iq_i})_{\sigma_i}q_i\otimes p_i\Bigr).
\end{split} \]
The last $\Tr$ term in the second row corresponds to the collision of two endpoints of a separated segment that occurs on a flow-line between $p_i$ and $q_i$. This term and the corresponding term in $\partial \theta_\mathrm{Z}^*(\xi,\eta)$ cancel each other since $o^*_E(\calA_{q_i}(\xi))_z\wedge o^*_E(\calD_{p_i}(\xi))_z=(-1)^{|q_i|(4-|p_i|)}o^*_E(\calD_{p_i}(\xi))_z\wedge o^*_E(\calA_{q_i}(\xi))_z=o^*_E(\calD_{p_i}(\xi))_z\wedge o^*_E(\calA_{q_i}(\xi))_z$ implies the equivalence of the coorientations of $ST^v(\overline{L}_{p_iq_i})$ and $ST^v(\overline{L}_{q_ip_i})$ in $ST^v(E)$, and since the coefficients of them are given by $(\psi_i g \psi_i^{-1})_{q_ip_i}$ and $-(\psi_i g \psi_i^{-1})^T_{p_iq_i}=-(\psi_i g \psi_i^{-1})_{q_ip_i}$, respectively. Hence we have $\partial \theta_\mathrm{Z}(\xi,\eta)+\partial \theta_\mathrm{Z}^*(\xi,\eta)=-\overline{s}_\xi-\overline{s}_{-\xi}$.\hfill\qed

%%%%%%%%%%%%%%%%%%%%%%%%%%%%%%%
\subsection{Transversality: Proof of Lemma~\ref{lem:dxooy}}\label{ss:transversality}

By the definition of iterated integrals, each term of $\displaystyle\int_\sigma X\underbrace{\Omega\cdots\Omega}_n\,\t{Y}$ consists of spaces of the form of the pullback in the following diagram:
\begin{equation}\label{eq:pullback}
 \xymatrix{
   & \bcalA_{q_1}\times\pbcalM_{q_1q_2}\times\cdots\times\pbcalM_{q_nq_{n+1}}\times\bcalD_{q_{n+2}} \ar[d]^-{\phi_1\times\cdots\times \phi_{n+2}}\\
  \bcalD_c(\eta)\times \Delta^{n+1} \ar[r]_-{\hat{\sigma}} & B^{n+2}
} 
\end{equation}
where $\hat{\sigma}(\gamma;s_1,\ldots,s_{n+1})=(\bar{\gamma}(s_1),\ldots,\bar{\gamma}(s_{n+1}),\bar{\gamma}(1))$. We shall prove that $\hat{\sigma}$ and $\phi_1\times\cdots\times \phi_{n+2}$ can be made strata transversal (\cite[Appendix]{BT} for the definition of strata transversality). Note that since $\hat{\sigma}$ cannot be perturbed in arbitrary direction in the space of smooth maps, the transversality is not obvious from Thom's transversality theorem. 

For each point $\bvec{y}=(y_1,\ldots,y_{n+2})=(\bar\gamma(s_1),\cdots,\bar\gamma(s_{n+1}),\bar\gamma(1))\in B^{n+2}$ of the image of $\hat{\sigma}$, we prove that 
$T_{\tbvec{y}}B^{n+2}=T_{y_1}B\oplus\cdots\oplus T_{y_{n+2}}B$ is spanned by the images of $d\hat{\sigma}$ and $d\phi_1\oplus\cdots\oplus d\phi_{n+2}$. It suffices to prove that for each $j$, the subspace $T_{y_j}B$ of $T_{\tbvec{y}}B^{n+2}$ is spanned by vectors of these images. In the following, we assume that $\sigma:\bcalD_c(\eta)\to B$ and the stratum $B^{(1)}$ of $i/i$-intersections are strata transversal, without loss of generality. 
\par\medskip

\noindent{\bf Case 1}: Suppose that $\pbcalM_{q_jq_{j+1}}$ is the moduli space of $i+\ell/i$-intersection ($\ell\geq 1$) for all $j$ and that $y$ is in the image from the codimension 0 stratum of $\bcalA_{q_1}\times\pbcalM_{q_1q_2}\times\cdots\times\pbcalM_{q_nq_{n+1}}\times\bcalD_{q_{n+2}}$. In this case, $\phi_1\times\cdots\times \phi_{n+2}$ is locally a submersion, and hence transversal to $\hat{\sigma}$. 
\par\medskip

\noindent{\bf Case 2}: Suppose that either $\pbcalM_{q_{j-1}q_j}$ is the space of $i/i$-intersections for some $j$, or $\bvec{y}$ comes from a point of $\bcalA_{q_1}\times\pbcalM_{q_1q_2}\times\cdots\times\pbcalM_{q_nq_{n+1}}\times\bcalD_{q_{n+2}}$ that includes a point of codimension 1 stratum of $\pbcalM_{q_{j-1}q_j}$ of $i+1/i$-intersections. Since the latter is equivalent to the former, we need only to consider the former case. Note that by Lemma~\ref{lem:i/i-str} the strata of $B^{(1)}$ of codimension $\geq 2$ are transversal intersections of several codimension 1 strata. So it suffices to check the transversality at a point of each codimension 1 stratum gathering around a codimension $\geq 2$ stratum. Since $B^{(1)}$ consists of immersed codimension 1 submanifolds in $B$, the image of $d\phi_j:T\pbcalM_{q_{j-1}q_j}\to TB$ is a codimension 1 subbundle over the image of $\phi_j$. Since we assume that $\sigma$ and $B^{(1)}$ are strata transversal, $B^{(1)}_\sigma=\sigma^{-1}(B^{(1)})$ consists of immersed codimension 1 submanifolds of $\bcalD_c(\eta)$ and intersects $\partial \bcalD_c(\eta)$ transversally.

The image of $d\hat{\sigma}$ in $T_{\tbvec{y}}B^{n+2}$ is spanned by $(\frac{d\bar{\gamma}}{ds})_{y_1},\ldots,(\frac{d\bar{\gamma}}{ds})_{y_{n+2}}$, and the tuple of tangent vectors that 
a perturbation of $\bar{\gamma}$ along $\bcalD_c(\eta)$ induces on $T_{y_1}B,\ldots,T_{y_{n+2}}B$. Now we assume that $\pbcalM_{q_{j-1}q_j}$ is a space of $i/i$-intersections. If $\phi_j:\pbcalM_{q_{j-1}q_j}\to B$ and $\bar{\gamma}$ is transversal at $y_j$, then $T_{y_j}B$ is spanned by $\mathrm{Im}\,d\phi_j$ and $(\frac{d\bar{\gamma}}{ds_j})_{y_j}$. If such a transversality condition is satisfied for each $i/i$-intersection over $\bvec{y}$, then it follows that $\hat{\sigma}$ and $\phi_1\times\cdots\times\phi_{n+2}$ are transversal at $\bvec{y}$. If $\phi_j$ and $\bar{\gamma}$ are not transversal at $y_j$, then we need extra work. In this case, $(\frac{d\bar{\gamma}}{ds_j})_{y_j}$ is included in the image of $d\phi_j$. Since $B^{(1)}$ consists of codimension 1 strata of $B$, $T_{y_j}B^{(1)}$ is the image of $d\phi_j$ in $T_{y_j}B$. The remaining 1-dimension needs to be taken from a direction in $T_{y_j}\calD_c(\eta)-T_{y_j}B^{(1)}_\sigma$. We shall see below that such a direction can be obtained by a perturbation of $\bar{\gamma}$.

If, for each $\ell\neq j$, $\phi_\ell$ and $\bar{\gamma}$ are transversal at $y_\ell$ (e.g., Figure~\ref{fig:D_c} (1)), then the proof is easy. For a vector $\bar{\nu}_j$ in $T_{y_j}\calD_c(\eta)-T_{y_j}B^{(1)}_\sigma$, choose a tangent vector $\nu\in T_\gamma \calD_c(\eta)$ that induces $\bar{\nu}_j$ in $T_{y_j}\calD_c(\eta)$. Namely, if we put $\bar{\nu}=d\hat{\sigma}(\nu)=(\bar{\nu}_1,\ldots,\bar{\nu}_{n+2})$, then $-\bar{\nu}_\ell$ can be written as a linear combination of a vector in $\mathrm{Im}\,d\phi_\ell$ and $(\frac{d\bar{\gamma}}{ds_\ell})_{y_\ell}$ since 
$d\phi_\ell$ and $\bar{\gamma}$ are transversal at $y_\ell$. Then we may obtain $\bar{\nu}-\sum_{\ell\neq j} \bar{\nu}_\ell=(0,\ldots,0,\bar{\nu}_j,0,\ldots,0)$. This implies the transversality at $\bvec{y}$. 

\begin{figure}
\begin{center}
\begin{tabular}{ccc}
\includegraphics[height=30mm]{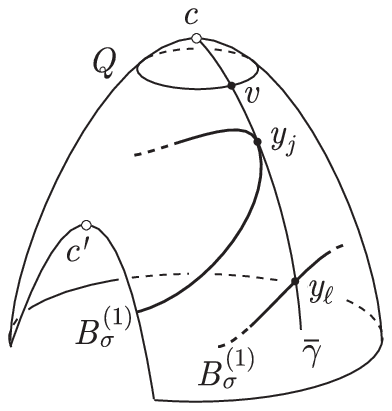} & & \includegraphics[height=40mm]{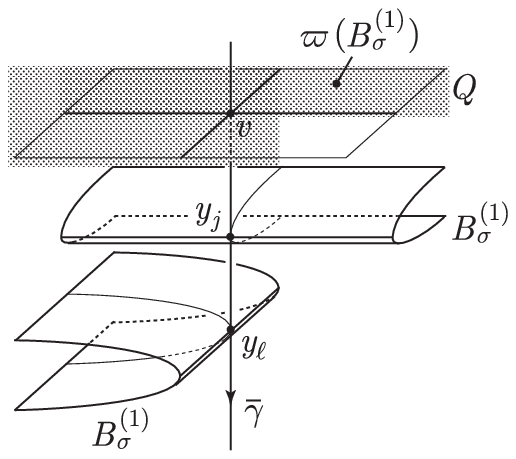}\\
(1) & \phantom{aaa} & (2)
\end{tabular}
\end{center}
\caption{The case where $\bar{\gamma}$ is not transversal to $B_\sigma^{(1)}$}\label{fig:D_c}
\end{figure}

When there is $\ell\neq j$ such that $\phi_\ell$ and $\bar{\gamma}$ is not transversal at $y_\ell$ too (e.g., Figure~\ref{fig:D_c} (2)), one cannot construct the correction term $-\bar{\nu}_\ell$ in arbitrary direction in $T_{y_\ell}B$ as the span of $\mathrm{Im}\,d\phi_\ell$ and $(\frac{d\bar{\gamma}}{ds_\ell})_{y_\ell}$, but only in $T_{y_\ell}B^{(1)}$. If $T_{y_\ell}B^{(1)}_\sigma$ is carried by the flow of $\eta$ to $T_{y_j}B^{(1)}_\sigma$, then the argument above does not work. Namely, if 
$\bar{\gamma}$ is perturbed by moving in a direction independent from $T_{y_j}B^{(1)}_\sigma$, then one cannot construct a vector in the image of $d\phi_\ell$ that cancels $\bar{\nu}_\ell\in T_{y_\ell}B$. We shall see that such a situation can be avoided by perturbing $\xi$. For the critical point $c$ of $h$, put $a_0=h(c)$. For a small $\ve>0$, put $Q=\sigma^{-1}(h^{-1}(a_0-\ve))$, which is a codimension 1 submanifold of $\bcalD_c(\eta)$. Without loss of generality, we may assume that $h(B^{(1)}_\sigma)\subset (-\infty,a_0-\ve)$, namely, $B^{(1)}_\sigma$ lies under $Q$ with respect to the height of $h$. If $B^{(1)}_\sigma$ is carried upward by the flow of $\eta$, then $B_\sigma^{(1)}-\partial\bcalD_c(\eta)$ arrives at $Q$. This gives a piecewise smooth map $\varpi:B^{(1)}_\sigma-\partial\bcalD_c(\eta)\to Q$.

Now we assume that $\bar{\gamma}(1)\in \calD_c(\eta)$ for $\bar{\gamma}$ as above. In this case, $\bar{\gamma}$ is nonsingular and intersects $Q$ at a point $v$ in the image of $\varpi$. By assumption, $v$ is a critical value of $\varpi$ and there are at least two singular points, including $y_j,y_\ell$, in the preimage $\varpi^{-1}(v)\subset B^{(1)}_\sigma$. It is known that when $\dim{X}=\dim{Y}$, finite maps $X\to Y$ between manifolds is residual in $C^\infty(X,Y)$ (\cite[Theorem~VII.2.6 (p.169)]{GG}). If a smooth map $X\to Y$ is finite, it is ``finite-to-one'' on compact subsets of $X$ (\cite[Proposition~VII.2.2 (p.167)]{GG}). 
Hence after a small perturbation of $\xi$ which may perturb $B^{(1)}_\sigma$, we may assume that the number of singular points in the compact set $\varpi^{-1}(w)\subset B^{(1)}_\sigma$ is finite for every point $w$ in the image of $\varpi$. Under this assumption, the preimage $\varpi^{-1}(O)$ of a small neighborhood $O$ of $v\in Q-\varpi(\partial\bcalD_c(\eta))$ consists of finitely many sheets. Let $U_1,\ldots,U_r$ be the sheets of $\varpi^{-1}(O)$ that have singularities mapped by $\varpi$ to $v$. By perturbing $B^{(1)}_\sigma$ further inside $U_1,\ldots,U_r$, we may assume that the restrictions of $\varpi$ on $U_1,\ldots,U_r$ is transversal at $v$. By an inductive argument with a locally finite open cover $\{O_\alpha\}$ of the paracompact set of multiple singular values in $Q-\varpi(\partial\bcalD_c(\eta))$, i.e., the values $v$ of $\varpi$ with at least two singular points in $\varpi^{-1}(v)$, the transversality of $\varpi$ between singularities at every multiple singular value can be proved. 

For the resulting $B^{(1)}_\sigma$, if the singular points of $\varpi$ in $\{y_1,\ldots,y_{n+2}\}$ are $y_j,y_{\ell_1},\ldots,y_{\ell_r}$, then $d\varpi_{y_{\ell_1}}(T_{y_{\ell_1}}B^{(1)}_\sigma)\cap \cdots\cap d\varpi_{y_{\ell_r}}(T_{y_{\ell_r}}B^{(1)}_\sigma)$ and $d\varpi_{y_j}$ are transversal in $T_vQ$. We take a nonzero vector $\delta$ in $d\varpi_{y_{\ell_1}}(T_{y_{\ell_1}}B^{(1)}_\sigma)\cap \cdots\cap d\varpi_{y_{\ell_r}}(T_{y_{\ell_r}}B^{(1)}_\sigma)$. Then a perturbation of $\bar{\gamma}$ that induces $\delta$ on $T_vQ$ induces a tangent vector $\bar{\mu}=(\bar{\mu}_1,\ldots,\bar{\mu}_{n+2})\in T_{\tbvec{y}}B^{n+2}$. By the assumption on $\delta$, we see that $\bar{\mu}_{\ell_1},\ldots,\bar{\mu}_{\ell_r}$ is tangent to $B^{(1)}_\sigma$, and $\bar{\mu}_j$ is independent from $T_{y_j}B_\sigma^{(1)}$. Hence for every $\ell\neq j$, we may obtain $\bar{\mu}'=(0,\ldots,0,\bar{\mu}_j,0,\ldots,0)$ by adding vectors in $T_{y_\ell}B^{(1)}_\sigma=\mathrm{Im}\,d\phi_{\ell}$. Since $\bar{\mu}_j$ and $T_{y_j}B^{(1)}_\sigma$ spans $T_{y_j}B$, this proves the desired transversality at $\bvec{y}$ under every multiple singular value in $Q-\varpi(\partial\bcalD_c(\eta))$. 

Next we assume that $\bar{\gamma}(1)$ lies in the image of $\partial \sigma:\partial\bcalD_c(\eta)\to B$, in which case $\bar{\gamma}$ may pass through critical points of $h$ on the way. At a critical point $c_j$ of $h$ that $\bar{\gamma}$ passes, take a level surface $Q_j$ located just below $c_j$. Then applying the argument given above for $c=c_j$ and $Q=Q_j$, we may assume that the intersections between singularities of $\varpi$ are transversal in $Q_j$ by perturbing $B^{(1)}_\sigma$. Also, by applying similar argument for $\calM'_{cc_j}(\eta)$ with the $Q$ just below $c$, we may assume a similar transversality condition for $\calM'_{cc_j}(\eta)$. Hence the transversality at $\bvec{y}$ on a singular flow-line from $\calM'_{cc_j}(\eta)\times \calD_{c_j}(\eta)$ is proved. Since transversality is a generic condition, we may then assume that the transversality of (\ref{eq:pullback}) is satisfied on a small neighborhood of $\bar{\gamma}$ in $\bcalD_c(\eta)$. This proves the transversality of (\ref{eq:pullback}) on the codimension $\leq 1$ strata of $\bcalD_c(\eta)$. The result for higher-codimension strata can be proved inductively in a similar way. 

Now, we have proved that $\hat{\sigma}$ and $\phi_1\times\cdots\times\phi_{n+2}$ are strata transversal. Since each entry of the iterated integral $\displaystyle\int_\sigma X\underbrace{\Omega\cdots\Omega}_n\,\t{Y}$ is the fiber product of $\hat{\sigma}$ and $\phi_1\times\cdots\times\phi_{n+2}$, its codimension 1 stratum is given by the fiber product of codimension 1 stratum and a codimension 0 stratum or vice versa (\cite[Proposition~A.5]{BT}). $S_n$ of (\ref{eq:dxooy}) comes from $\partial(X\times\underbrace{\Omega\times\cdots\times\Omega}_n\times\t{Y})$, $T_n+U_n$ comes from $C\times\partial \Delta^{n+1}$, and $V_n$ comes from $\partial C\times\Delta^{n+1}$.
\qed

%%%%%%%%%%%%%%%%%%%%%%%%%%%%%%%
%%%%%%%%%%%%%%%%%%%%%%%%%%%%%%%
\mysection{Cycles in $B\Diff(D^4,\partial)$ associated to graphs}{s:cycles}

We shall construct $(D^4,\partial)$-bundles by an analogue of Goussarov--Habiro's graph-clasper surgery that would be detected by $\hat{Z}_k^\adm$, and shall review some fundamental properties of the surgery. 

%%%%%%%%%%%%%%%%%%%%%%%%%%%%%%%
\subsection{Borromean rings (e.g., \cite{Ma})}\label{ss:borromean}

If $d$ is a positive integer and if $p,q,r$ are integers such that $0<p,q,r<d, p+q+r=2d-3$, then the Borromean rings is defined as the three-component link
$B(p,q,r)_d:S^p\cup S^q\cup S^r\to \R^d$, 
whose components are given by the submanifolds of $\R^d=\R^{d-p-1}\times \R^{d-q-1}\times \R^{d-r-1}$ of points $(x,y,z)$ such that
\[ \begin{split}
&\textstyle\frac{|y|^2}{4}+|z|^2=1,\quad x=0\quad \mbox{or}\\
&\textstyle\frac{|z|^2}{4}+|x|^2=1,\quad y=0\quad \mbox{or}\\
&\textstyle\frac{|x|^2}{4}+|y|^2=1,\quad z=0.
\end{split} \]

Standard (normal) framings for the Borromean rings is given as follows. Let $n_1,n_2,n_3$ be the outward unit normal vector field on $S^p\subset \R^{p+1}$, $S^q\subset\R^{q+1}$, $S^r\subset \R^{r+1}$, respectively. Then the normal framings on the three components are given by $(\frac{\partial}{\partial x_1},\ldots, \frac{\partial}{\partial x_{d-p-1}},n_1)$, $(\frac{\partial}{\partial y_1},\ldots, \frac{\partial}{\partial y_{d-q-1}},n_2)$, $(\frac{\partial}{\partial z_1},\ldots, \frac{\partial}{\partial z_{d-r-1}},n_3)$, respectively.

A long version of the (framed) Borromean rings is obtained as follows. We call an affine embedding $\R^p\to \R^d$ a standard inclusion. Given a link $L:\R^p\cup \R^q\cup \R^r \to \R^d$ consisting of disjoint standard inclusions, and a Borromean rings $B(p,q,r)_d$ that is disjoint from $L$, we join the images of $\R^p$ and $S^p$, $\R^q$ and $S^q$, $\R^r$ and $S^r$, by three mutually disjoint arcs that are also disjoint from components of the links $L$ and $B(p,q,r)_d$ except their endpoints. Then replace the arcs with thin tubes $S^{p-1}\times I$, $S^{q-1}\times I$, $S^{r-1}\times I$ to construct connected sums. The result is a (naturally framed) long link 
$B(\underline{p},\underline{q},\underline{r})_d:\R^p\cup \R^q\cup \R^r\to \R^d$. One may also consider partial connected sum, which joins $B(p,q,r)_d$ to a link of standard inclusions with less components and denote the resulting embedding by $B(\underline{p},\underline{q},r)_d$ etc. Long Borromean embeddings $D^p\cup D^q\cup D^r\to D^d$ such that the preimage of $\partial D^d$ is $\partial D^p\cup \partial D^q\cup \partial D^r$ can also be defined similarly and we denote them by the same symbols as above. 

Let $\fEmb(D^p\cup D^q\cup D^r,D^d)$ denote the space of (normally) framed long embeddings such that the boundaries are all mapped to $\partial D^d$. 
Its subspace consisting of embeddings that are properly isotopic to the standard inclusion is denoted by $\fEmb_0(D^p\cup D^q\cup D^r,D^d)$. The subspace of $\fEmb(D^p\cup D^q\cup D^r,D^d)$ of embeddings such that some components are standard near the boundaries is denoted like $\fEmb(\underline{D}^p\cup D^q\cup D^r,D^d)$, where the underlined component(s) is standard near the boundary.

%%%%%%%%%%%%%%%%%%%%%%%%%%%%%%%
\subsection{Vertex-oriented arrow graph}\label{ss:arrow-graph}

We orient each edge of a trivalent graph such that each vertex has both input and output incident edges. That any trivalent graph has such an orientation can be proved by induction on the number of edges. We call a trivalent graph equipped with such an orientation an {\it arrow graph}. The edge orientation for arrow graph is independent of the edge orientation for $\vec{C}$-graph. Possible status of input/output of the three incident edges at a vertex of an arrow graph are as shown in the following figure.
\[  \includegraphics[height=20mm]{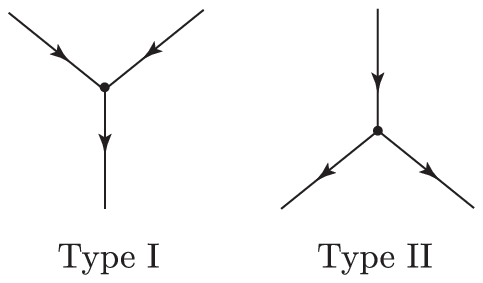} \]

The quotient by the label change relation in the definition of $\calA_k$ is equivalent to considering unlabeled oriented graph $(\Gamma,o)$ ($o$ is an orientation of the real vector space $\R^{\mathrm{Edges}(\Gamma)}$) modulo the relation $(\Gamma,-o)\sim -(\Gamma,o)$. This orientation determines a vertex-orientation for an arrow graph as follows. We decompose each edge $e$ of an arrow graph $\Gamma$ into half-edges $H(e)=\{e_+,e_-\}$ ordered according to the arrow orientation of $e$, namely, so that $e_-$ is input and $e_+$ is output, and let $\deg\,{e_+}=1$, $\deg\,{e_-}=2$. Then an orientation of $\Gamma$ is given by $o=(e_{1+}\wedge e_{1-})\wedge \cdots\wedge (e_{3k+}\wedge e_{3k-})$, and this can be rewritten as $o=\tau_1\wedge\tau_2\wedge\cdots\wedge \tau_{2k}$, $\tau_i=e_{p\pm}\wedge e_{q\pm}\wedge e_{r\pm}$ by rearrangement, where $e_{p\pm}, e_{q\pm}, e_{r\pm}$ are half-edges meeting at the $i$-th vertex. Each term $\tau_i$ gives a vertex-orientation, namely, an ordering of the half-edges meeting at each vertex. Conversely, given a vertex-labelling and a vertex-orientation to an arrow graph, an orientation of $\R^{\mathrm{Edges}(\Gamma)}$ is obtained. In this section, we mainly represent oriented graphs by vertex-labelled vertex-oriented arrow graphs. 

%%%%%%%%%%%%%%%%%%%%%%%%%%%%%%%
\subsection{Y-link associated to a trivalent vertex}\label{ss:Y-link}

To a vertex-oriented arrow graph $\Gamma$, we associate a Y-link $G=G_1\cup\cdots\cup G_{2k}$ in $D^4$ as follows (Figure~\ref{fig:G-to-Y-graph}).
\begin{enumerate}
\item We take an embedding $\iota:\Gamma\to \mathrm{Int}\,D^4$.
\item For each edge $e$, let $P(e)\subset \mathrm{Int}\,D^4$ be a small closed 4-ball centered at the middle point of $\iota(e)$ such that $P(e)$ is disjoint from vertices and other edges of $\iota(\Gamma)$. Further, we assume that $P(e)\cap P(e')=\emptyset$ if $e\neq e'$.
\item We decompose the closed interval $P(e)\cap \iota(e)$ into three subintervals: $P(e)\cap \iota(e)=[a,b]\cup [b,c]\cup [c,d]$. Then we remove the middle one $[b,c]$ and insert a standard Hopf link $S^1\cup S^2\to \mathrm{Int}\,P(e)$ instead, so that the image of $S^2$ is attached to $b\in [a,b]$ and the image of $S^1$ is attached to $c\in [c,d]$. 
\end{enumerate}
\begin{figure}
\begin{center}
\includegraphics[height=35mm]{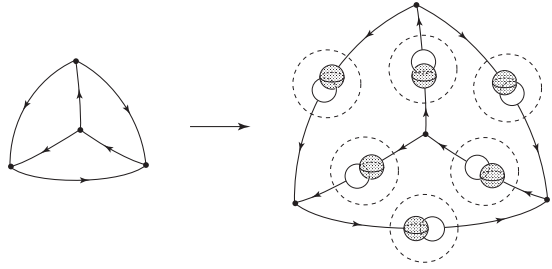}
\end{center}
\caption{An embedded arrow graph to a Y-link}\label{fig:G-to-Y-graph}
\end{figure}

\if0
\begin{wrapfigure}[9]{r}[0pt]{5cm}
\begin{center}
\end{center}
\end{wrapfigure}
\fi
The above procedure gives a disjoint union $G_1\cup G_2\cup \cdots \cup G_{2k}$ of $2k=|V(\Gamma)|$ components. We call each component $G_i$ a {\it Y-graph}, and $G=G_1\cup G_2\cup \cdots \cup G_{2k}$ a {\it Y-link} (or a graph clasper). There are two types for a Y-graph, according to whether the corresponding vertex is of type I or II in the following figure.
\[  \includegraphics[height=25mm]{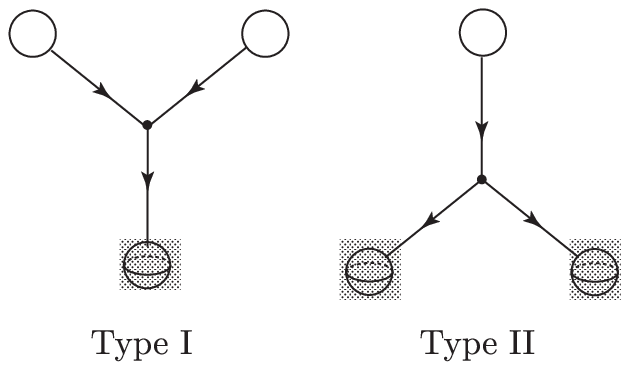} \]

By taking a small smooth closed tubular neighborhood $V_i\subset\mathrm{Int}\,D^4$ for each component $G_i$, we obtain a tuple $\vec{V}_G=(V_1,\ldots,V_{2k})$ of mutually disjoint handlebodies in $\mathrm{Int}\,D^4$. 

%%%%%%%%%%%%%%%%%%%%%%%%%%%%%%%
\subsection{Surgery along Y-links}

We shall construct a $(D^4,\partial)$-bundle by a family of surgery along $\vec{V}_G=(V_1,\ldots,V_{2k})$.
We take a family $\alpha_i:K\to \Diff(\partial V_i)$ of diffeomorphisms parametrized by a compact manifold $K$. This defines a bundle automorphism $\bar{\alpha}_i:\partial V_i\times K\to \partial V_i\times K$ of the trivial $V_i$-bundle over $K$ in a natural way. We put
\[ (D^4\times K)^{V_i,\alpha_i}=((D^4-\mathrm{Int}\,V_i)\times K)\cup_{\bar{\alpha}_i}(V_i\times K), \]
where the boundaries are glued together by $\bar{\alpha}_i$. Then the product structures on the two parts induce a bundle projection $\pi(\alpha_i):(D^4\times K)^{V_i,\alpha_i}\to K$.

Since the handlebodies $V_i$ are mutually disjoint, the surgery can be done at every $V_i$ simultaneously. Taking $\vec{\alpha}=(\alpha_1,\ldots,\alpha_{2k})$, $\alpha_i:K_i\to \Diff(\partial V_i)$, we do surgery at each $V_i$ by using $\alpha_i$, and then we obtain a family of surgeries parametrized by $K_1\times \cdots \times K_{2k}$ and a bundle projection
\[ \pi(\vec{\alpha}):(D^4\times\prod_{i=1}^{2k}K_i)^{\vec{V}_G,\vec{\alpha}}\to K_1\times\cdots\times K_{2k}.\]
More precisely, we may write the sub $(V_i,\partial)$-bundle of $(D^4\times K_i)^{V_i,\alpha_i}$ corresponding to $V_i\times K_i$ as $\bigcup_{t_i\in K_i}V_i(t_i)$, and we define $(D^4\times\prod_{i=1}^{2k}K_i)^{\vec{V}_G,\vec{\alpha}}$ by
\[ ((D^4-\mathrm{Int}\,(V_1\cup\cdots\cup V_{2k}))\times\prod_{i=1}^{2k}K_i)\cup_\partial\bigcup_{(t_1,\ldots,t_{2k})}(V_1(t_1)\cup\cdots\cup V_{2k}(t_{2k})). \]

In the following, we take a special one as $\alpha_i$. Put $V=V_i$ for simplicity. 
\begin{itemize}
\item Let $\alpha_\mathrm{I}:S^0\to \Diff(\partial V)$, $S^0=\{-1,1\}$, $\alpha_\mathrm{I}(-1)=\mathrm{id}$, $\alpha_\mathrm{I}(1)$ be the ``Borromean twist'' corresponding to $B(\underline{2},\underline{2},\underline{1})_4$. Detailed definition of $\alpha_\mathrm{I}$ will be given in \S\ref{ss:p-borr-I}.

\item Let $\alpha_{\mathrm{II}}:S^1\to \Diff(\partial V)$ be the ``parametrized Borromean twist'' given by parametrizing $B(\underline{2},\underline{2},\underline{1})_4$ over $S^1$. Detailed definition of $\alpha_\mathrm{II}$ will be given in \S\ref{ss:p-borr-II}.
\end{itemize}
Vertex-orientation is used to associate the components in the Borromean string link $B(\underline{2},\underline{2},\underline{1})_4$ to handles of a handlebody.

\begin{Def} Let $\Gamma$ be a vertex-oriented labelled arrow graph with $2k$ vertices. According to the type of the $i$-th vertex of $\Gamma$, we put $\alpha_i=\alpha_{\mathrm{I}}$ or $\alpha_{\mathrm{II}}$, and put $\vec{\alpha}=(\alpha_1,\ldots,\alpha_{2k})$. Then we put
\[ \pi^\Gamma=\pi(\vec{\alpha}),\quad E^\Gamma=(D^4\times\prod_{i=1}^{2k}K_i)^{\vec{V}_G,\vec{\alpha}},\quad B_\Gamma=\prod_{i=1}^{2k}K_i. \]
We also consider the straightforward analogue of this surgery for $(\R^4,U_\infty')$-bundles which is given by replacing $D^4$ with $\R^4$ in the definition above. 
\end{Def}

\begin{Thm}[Proof in \S\ref{ss:primitiveness} and \S\ref{s:computation}]\label{thm:Z(G)}
\begin{enumerate}
\item $\pi^\Gamma:E^\Gamma\to B_\Gamma$ is a $(D^4,\partial)$-bundle.
\item The $(D^4,\partial)$-bundle bordism class of $\pi^\Gamma:E^\Gamma\to B_\Gamma$ corresponds to an element of $\mathrm{Im}\,H$, and the following holds.
\[ \hat{Z}_k^\adm(\pi^\Gamma)=[\Gamma]. \]
\end{enumerate}
\end{Thm}

Theorem~\ref{thm:commute} follows immediately from Theorem~\ref{thm:Z(G)}. Namely, letting 
\[ \Psi_k:\calG_k'\to \mathrm{Im}\,H\otimes \Q\]
be defined by $\Psi_k(\Gamma)=[\pi^\Gamma:E^\Gamma\to B_\Gamma]$ by choosing arrows on $\Gamma$ arbitrarily,
then by Theorem~\ref{thm:Z(G)}, this satisfies the condition of Theorem~\ref{thm:commute}. We do not know whether the bordism class of $\Psi_k(\Gamma)$ depends on the choice of the arrows.

Let $M$ be a compact 4-manifold. For an embedding $\iota:\Gamma\to M$, one may also consider the straightforward analogue of the surgery defined above, giving an $(M, \partial)$-bundle $\pi^\iota:E^\iota\to B_\Gamma$. The following theorem can be proved just by replacing $D^4$ with $M$ in the proof of Theorem~\ref{thm:Z(G)}.
\begin{Thm}
The class of $\pi^\iota$ represents an element of $\Omega_k(B\Diff(M,\partial))$, and it is contained in the image of the natural map $H:\pi_kB\Diff(M,\partial)\to \Omega_k(B\Diff(M,\partial))$.
\end{Thm}
The class of $\pi^\iota$ depends only on the homotopy class of $\iota$, which can be described by $\Gamma$ as above with edges decorated by elements of $\pi_1(M)$, considered modulo certain relations as in \cite{GL}.

\subsection{Parametrized Borromean surgery of type I}\label{ss:p-borr-I}

In the following, we shall define parametrized Borromean twists $\alpha_\mathrm{I}$. As a preliminary, we fix a coordinate on $V$. Let $T$ be a handlebody obtained from a 3-disk by removing several 1-handles and 0-handles, and we put $V=T\times I$. We fix a coordinate on $T$ as follows. Let $T_0=[0,4]\times [-1,1]\times [0,1]$, and for $n=1,2,3$ and $\ve>0$, we define $T$ as follows (Figure~\ref{fig:T-type-I-II}). 
\[ \begin{split} 
  &h_n^1=\{(x,y)\in\R^2\mid (x-n)^2+y^2< \ve^2\}\times [0,1],\\
  &h_n^0=\{(x,y,z)\in \R^3\,|\, (x-n)^2+y^2+(z-\textstyle\frac{1}{2})^2< \ve^2\},\\
  &T=T_0-(h_1^{e_1}\cup h_2^{e_2}\cup h_3^{e_3}),\qquad (e_1,e_2,e_3)=\left\{\begin{array}{ll}
  (1,1,0) & \mbox{($V$: type I)}\\
  (1,0,0) & \mbox{($V$: type II)}
  \end{array}\right.
\end{split} \]

\begin{figure}
\begin{center}
\begin{tabular}{ccc}
\includegraphics[height=30mm]{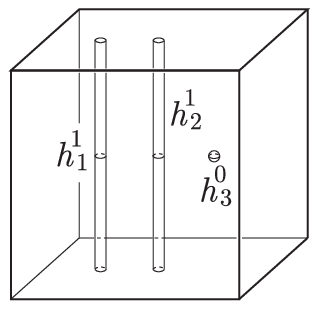} & & \includegraphics[height=30mm]{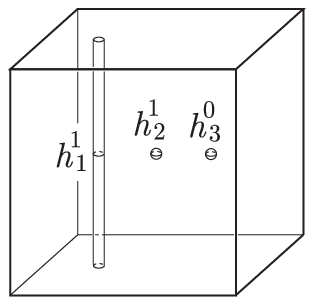}\\
(1) & & (2)
\end{tabular}
\end{center}
\caption{(1) $T$ in $V$ of type I. (2) $T$ in $V$ of type II.}\label{fig:T-type-I-II}
\end{figure}

The handlebody $V$ of type I is diffeomorphic to a handlebody obtained from $D^4$ by removing two 2-handles and one 1-handle, which are thin. We represent the thin handles in $D^4$ by a framed string link. If this framed string link is changed, the complement of it in $D^4$ changes accordingly. Especially, we consider the handlebody $V'$ that is the complement to the framed Borromean string link $B(\underline{2},\underline{2},\underline{1})_4$ of $\pi_0 \fEmb(\underline{D}^2\cup \underline{D}^2\cup \underline{D}^1,D^4)$. Since string links are standard near boundary and we are considering framed embeddings, a framed string link induces a trivialization of the sides of handles as sphere bundles over the cores, and $\partial V'$ is naturally identified with $\partial V$. 

For the type I handlebody $V$, we shall see that the handlebody $V'$ thus obtained can be realized as the mapping cylinder of a relative diffeomorphism $\varphi_0:(T,\partial T)\to (T,\partial T)$, which is defined by $C(\varphi_0)=(T\times I)\cup_{\varphi_0}(T\times \{0\})$. Note that the boundary of $C(\varphi_0)$ is $(T\times \{0,1\})\cup (\partial T\times I)=\partial V$. 
\begin{Lem}\label{lem:T-bundle-I}
For a handlebody $V$ of type I, there exists a relative diffeomorphism $\varphi_0:(T,\partial T)\to (T,\partial T)$ and a relative diffeomorphism $(V',\partial V)\to (C(\varphi_0),\partial V)$ that restricts to $\mathrm{id}$ on $\partial V$.
\end{Lem}
\begin{proof} By considering the third component of the framed tangle $B(\underline{1},\underline{1},1)_3$ (Figure~\ref{fig:beta-0} (1)) as a 1-parameter family of points, we obtain an element of $\pi_1\fEmb_0(\underline{D}^1\cup \underline{D}^1\cup D^0,D^3)$. By realizing this as a graph in the trivial $D^3$-bundle $T_0\times I\to I$, we obtain a framed string link $B(\underline{2},\underline{2},\underline{1})_4$ in $D^4$. Since the complement of a tangle of $\fEmb_0(\underline{D}^1\cup \underline{D}^1\cup D^0,D^3)$ is a handlebody relatively diffeomorphic to $T$, $V'$ can be considered as the total space of a $(T,\partial T)$-bundle that is trivialized on $\partial I$. Hence $V'$ is the mapping cylinder of a relative diffeomorphism of $T$.
\end{proof}
The relative diffeomorphism $\varphi_0:(T,\partial T)\to (T,\partial T)$ of Lemma~\ref{lem:T-bundle-I} extends to a diffeomorphism $\varphi_\mathrm{I}$ of $\partial V=(T\times\{0,1\})\cup (\partial T\times I)$ by setting $\varphi_0$ on $T\times\{0\}$ and $\mathrm{id}$ otherwise. 
\begin{Def}\label{def:alpha-I}
We define the map $\alpha_\mathrm{I}:S^0\to \Diff(\partial V)$ by $\alpha_\mathrm{I}(-1)=\mathrm{id}$, $\alpha_\mathrm{I}(1)=\varphi_\mathrm{I}$. 
Let $\widetilde{V}$ be the total space of the bundle $V'\cup (-V)\to S^0$ that is the union of $V'\to \{1\}$ and $-V\to \{-1\}$. 
\end{Def}

\subsection{Parametrized Borromean surgery of type II}\label{ss:p-borr-II}

The handlebody $V$ of type II is diffeomorphic to a handlebody obtained from $D^4$ by removing one 2-handle and two 1-handles, which are thin. Let us construct a $(V,\partial)$-bundle $\widetilde{V}\to S^1$ by using an element $\beta\in\pi_1\fEmb_0(\underline{D}^2\cup \underline{D}^1\cup \underline{D}^1,D^4)$ corresponding to a framed Borromean rings $B(\underline{2},\underline{2},\underline{1})_4$. If the second component of $B(\underline{2},\underline{2},\underline{1})_4$ is considered as the locus of a 1-parameter family of string knots $I\to  T_0\times I$ whose endpoints are mapped to $T_0\times\{0,1\}$, then a map $\beta'':I\to \fEmb_0(\underline{D}^2\cup D^1\cup \underline{D}^1,D^4)$ is obtained. Although this is not a loop, one may obtain a loop $\beta':S^1\to \fEmb_0(\underline{D}^2\cup D^1\cup \underline{D}^1,D^4)$ by taking fiberwise closures of the second component by a trivial arc, as in Figure~\ref{fig:beta-0} (2). 
Moreover, we deform the family of 1-disks for the second component by pressing a neighborhood of the boundary onto that of the standard inclusion $I\to T_0\times I; w\mapsto (2,0,0,w)$ simultaneously for all parameters, and get a family
\[ \beta\in \pi_1\fEmb_0(\underline{D}^2\cup \underline{D}^1\cup \underline{D}^1,D^4). \]
For a family of framed long embeddings $D^2\cup D^1\cup D^1\to D^4$ giving $\beta$, the complement $\widetilde{W}$ of its open $\ve$-tubular neighborhood in $D^4$ is a family of manifolds each diffeomorphic to $V$. The framing gives a bundle isomorphism $\partial \widetilde{W}\cong \partial V\times S^1$ over $S^1$, which gives $\widetilde{V}$ a structure of a $(V,\partial)$-bundle $\widetilde{V}\to S^1$. 

Next, let us show that thus obtained $\widetilde{V}$ is a 1-parameter family of mapping cylinders for an element of $\pi_1\Diff(T,\partial T)$. For a family of relative diffeomorphisms $\varphi_{0,t}:(T,\partial T)\to (T,\partial T)$ ($t\in S^1$), we put $\widetilde{C}(\{\varphi_{0,t}\})=\bigcup_{t\in S^1}C(\varphi_{0,t})$. This has a natural structure of a $(V,\partial)$-bundle over $S^1$. 
\begin{Lem}\label{lem:T-bundle-II}
For a handlebody $V$ of type II, there exist a family of relative diffeomorphisms $\varphi_{0,t}:(T,\partial T)\to (T,\partial T)$ ($t\in S^1$) and a relative bundle isomorphism
\[ (\widetilde{V},\partial V\times S^1)\to (\widetilde{C}(\{\varphi_{0,t}\}),\partial V\times S^1)\]
that restricts to $\mathrm{id}$ on the boundary. 
\end{Lem}
\begin{proof}
We shall see that $\beta\in \pi_1\fEmb_0(\underline{D}^2\cup \underline{D}^1\cup \underline{D}^1,D^4)$ is obtained by rewriting an element $\beta_0\in \pi_2\fEmb_0(\underline{D}^1\cup D^0\cup D^0,D^3)$ into a family over $I$ by suspension. 

\begin{figure}
\begin{center}
\begin{tabular}{ccccc}
\raisebox{5mm}{\includegraphics[height=25mm]{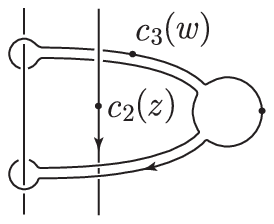}} &  & \raisebox{5mm}{\includegraphics[height=25mm]{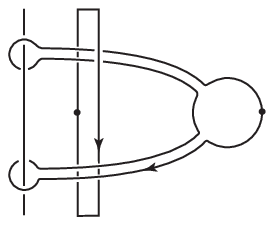}} & & \includegraphics[height=40mm]{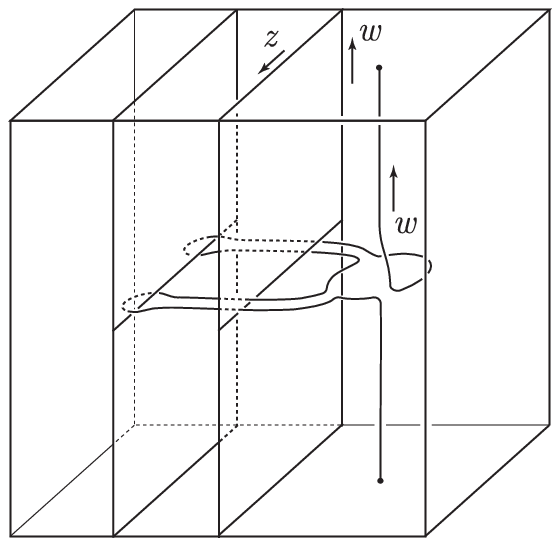}\\
(1) & & (2) & & (3)
\end{tabular}
\end{center}
\caption{(1) $B(\underline{1},\underline{1},1)_3$ parametrized by $(z,w)\in I\times I$. (2) $B(\underline{1},1,1)_3$ parametrized by $S^1\times I$. (3) $\beta'':I\to \fEmb_0(\underline{D}^2\cup D^1\cup \underline{D}^1,D^4)$. Horizontal section is parallel to the 3-disk $T_0$ on the top.}\label{fig:beta-0}
\end{figure}
We shall construct $\beta_0$ explicitly. The 1-handles and 0-handles in $T_0$ become 2-handles and 1-handles in $T_0\times I$, whose complement is $V$. We saw that $\beta$ is obtained by replacing the trivial $S^1$-family of the handles in $(T_0\times I)\times S^1$ by a family corresponding to the Borromean string link $B(\underline{2},\underline{2},\underline{1})_4$. 
Let $w$ be the parameter of $I$ in $V=T\times I$ and let $s$ be the parameter of $J=[0,1]\subset S^1=[0,2\pi]/{0\sim 2\pi}$. We may assume that the first and second components $C_1,C_2\subset T_0\times I$ of $B(\underline{2},\underline{2},\underline{1})_4$ are standard 2-disks. Hence in particular, the second component can be written by using $z\in I$ as
\[ C_2(z,w)=(2,0,z,w)\in T_0\times I. \]
This is a map to a single fiber of $(T_0\times I)\times S^1\to S^1$. By replacing the parameter $z$ with $s\in J$, the second component $C_2$ can be considered as a 1-parameter family of 1-disks in $T_0\times I$ varying with respect to $s\in J$. 
Also, the third component may be assumed to be monotonic with respect to the height $w$, and may be written as a curve $C_3(w)$ ($w\in I$) in $T_0\times I$. Moreover, $C_3(w)$ can be taken as the lift of a simple curve $c_3(t)$ in $T_0$. If $C_1',C_2'(s,w),C_3'(w)$ are the projections of $C_1,C_2(s,w),C_3(w)$ on $T_0$, then 
\[ C_1' \cup C_2'(s,w) \cup C_3'(w) \subset T_0\]
gives a 2-parameter family of $D^1\cup D^0\cup D^0\to T_0$ with respect to $(w,s)\in I\times J$.
The locus of the projection of the family of embeddings on $T_0$ is $D^1\cup c_2\cup c_3$ of Figure~\ref{fig:beta-0} (1), which gives $B(\underline{1},\underline{1},1)_3$. 

Let us see that this 2-parameter family gives $\beta$ and $\beta_0$. If we realize the $I\times J$-family of embeddings $D^1\cup D^0\cup D^0\to T_0$ as a graph in $T_0\times I\times J$, an element
\[ \beta'':J\to \fEmb_0(\underline{D}^2\cup D^1\cup \underline{D}^1,D^4) \]
is obtained. Then by closing this family by replacing $J$ with $S^1$ and by shrinking the boundary of the second component as before, we obtain an element $\beta\in \pi_1\fEmb_0(\underline{D}^2\cup \underline{D}^1\cup \underline{D}^1,D^4)$, and moreover, by taking a 2-parameter family of the complemental 3-dimensional handlebodies for $\beta$ as explained above, we obtain an element $\beta_0\in\pi_2\fEmb_0(\underline{D}^1\cup D^0\cup D^0,D^3)$. 

Finally, $\beta_0$ corresponds to a $(T,\partial T)$-bundle over $I\times J$, and it is a 1-parameter family of the mapping cylinder on $T$. 
\end{proof}

\begin{Def}
We define the map $\alpha_\mathrm{II}:S^1\to \Diff(\partial V)$ by extending $\{\varphi_{0,t}\}$ to a 1-parameter family of diffeomorphisms of $\partial V$ by $\mathrm{id}$. 
\end{Def}

There is a natural ``graphing'' map $G:\pi_1\fEmb_0(\underline{D}^2\cup \underline{D}^1\cup \underline{D}^1,D^4)\to \pi_0\fEmb(\underline{D}^3\cup \underline{D}^2\cup \underline{D}^2,D^5)$, which is obtained by representing a 1-parameter family of framed long embeddings $D^2\cup D^1\cup D^1\to D^4$ by a single map $(D^2\cup D^1\cup D^1)\times I\to D^4\times I$. 
\begin{Lem}\label{lem:B(3,2,2)}
The image of $\beta\in \pi_1\fEmb_0(\underline{D}^2\cup \underline{D}^1\cup \underline{D}^1,D^4)$ under $G$ is the class of $B(\underline{3},\underline{2},\underline{2})_5$.
\end{Lem}
\begin{proof}
The element $G(\beta)$ can be represented by the graph of an $I$-family of embeddings $D^1\to D^4-(D^2\cup D^1)$, whose image covers the third component in $B(\underline{2},\underline{1},\underline{2})_4$. Thus in the graph of the family, the first two components are standard, and most of the third component can be collapsed into a single fiber $L$, which coincides with the third one in $B(\underline{2},\underline{1},\underline{2})_4$. The result is an embedding $D^3\cup D^2\cup D^2\to D^5$, which is obtained by taking suspensions of the first two components in $B(\underline{2},\underline{1},\underline{2})_4$ in $L$. It is $B(\underline{3},\underline{2},\underline{2})_5$.
\end{proof}

%%%%%%%%%%%%%%%%%%%%%%%%%%%%%%%
\subsection{Homotopical properties of $\alpha_\mathrm{I}$ and $\alpha_\mathrm{II}$}

We take standard cycles $a_1,a_2,a_3,b_1,b_2,b_3$ of $\partial V$ as follows. Here we again use the standard coordinate of $V$ fixed in \S\ref{ss:p-borr-I}.
When $V$ is of type I, we let $b_1,b_2,b_3\subset T=T\times\{1\}$ be defined by
\[ \begin{split}
  b_1=S^1_{2\ve}(1,0)\times \{\textstyle\frac{1}{2}\},\quad  
  b_2=S^1_{2\ve}(2,0)\times \{\textstyle\frac{1}{2}\},\quad
  b_3=S^2_{2\ve}(n,0,\frac{1}{2}).
\end{split}\]
Here, we denote by $S^1_\delta(a,b)\subset\R^2$, $S^2_\delta(a,b,c)\subset \R^3$, the codimension 1 spheres centered at $(a,b)$, $(a,b,c)$ respectively, with radius $\delta$. We consider $b_1,b_2$ as 1-cycles by counter-clockwise orientations in circles of $[0,4]\times [-1,1]$. We consider $b_3$ as a 2-cycle by inducing an orientation from $\R^3$ by outward-normal-first convention. We define disks $a_1^T,a_2^T,a_3^T\subset T$ by 
$a_1^T=\{1\}\times[-1,-\ve]\times[0,1]$, $a_2^T=\{2\}\times[-1,-\ve]\times[0,1]$, $a_3^T=\{3\}\times[-1,-\ve]\times\{\textstyle\frac{1}{2}\}$, and put
\[ a_\ell=(a_\ell^T\times\{1\})\cup (\partial a_\ell^T\times I)\cup (-a_\ell^T\times\{0\})\subset \partial V.\]
We orient $a_\ell$ so that $a_\ell\cdot b_\ell=1$. When $V$ is of type II, we change $b_2$ and $a_2^T$ as
\[ b_2=S^2_{2\ve}(n,0,\textstyle\frac{1}{2}), \quad a_2^T=\{2\}\times[-1,-\ve]\times\{\textstyle\frac{1}{2}\} \]
in the definitions above for type I. We define the cycles $\widetilde{a}_\ell, \widetilde{b}_\ell$ of $\partial V\times S^1$ by 
\[ \widetilde{a}_\ell=a_\ell\times S^1,\quad  \widetilde{b}_\ell=b_\ell\times S^1,
\]
and orient them so that $\widetilde{a}_\ell\cdot b_\ell=1$, $a_\ell\cdot\widetilde{b}_\ell=1$.
\begin{figure}
\begin{center}
\begin{tabular}{ccc}
\includegraphics[height=30mm]{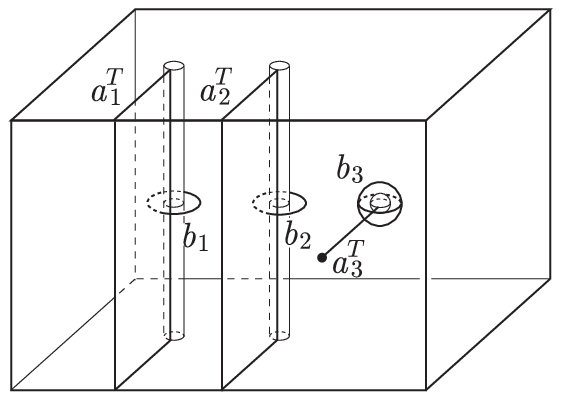} & \phantom{aaaaa} & \includegraphics[height=30mm]{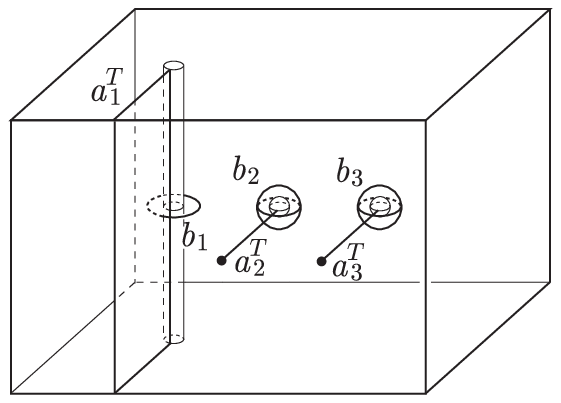}\\
(1) & & (2)
\end{tabular}
\end{center}
\caption{$a_1^T,a_2^T,a_3^T,b_1,b_2,b_3\subset T$, (1) in $V$ of type I, (2) in $V$ of type II.}\label{fig:ab-type-I-II}
\end{figure}

\begin{Prop}\label{prop:wh-prod}
\begin{enumerate}
\item \emph{(Type I)} By the diffeomorphism $\alpha_{\mathrm{I}}(1):\partial V\to \partial V$, the cycles $a_\ell, b_\ell$ change as follows.
\begin{equation}\label{eq:whitehead_1}
  \alpha_{\mathrm{I}}(1)_*\,a_\ell \simeq a_\ell \pm [b_m, b_n], \qquad \alpha_{\mathrm{I}}(1)_*\,b_\ell  \simeq b_\ell.
 \end{equation}
Here, $[\,,\,]$ is the Whitehead product, and $m,n$ are numbers such that $\{\ell,m,n\}=\{1,2,3\}$. The $\pm$ on the right hand side is the connected sum between base points. In particular, $a_\ell$ and $\alpha_{\mathrm{I}}(1)_*\,a_\ell$, $b_\ell$ and $\alpha_{\mathrm{I}}(1)_*\,b_\ell$ belong to the same bordism classes of $\Omega_*(\partial V)$.
\item \emph{(Type II)} By the bundle automorphism $\bar{\alpha}_{\mathrm{II}}:\partial V\times S^1\to \partial V\times S^1$, the cycles $\widetilde{a}_\ell, \widetilde{b}_\ell$ change as follows.
\begin{equation}\label{eq:whitehead_2}
  \bar{\alpha}_{\mathrm{II}*}\,\widetilde{a}_\ell  \simeq \widetilde{a}_\ell \pm [{b}_m, {b}_n],\qquad
  \bar{\alpha}_{\mathrm{II}*}\,\widetilde{b}_\ell  \simeq \widetilde{b}_\ell.
\end{equation}
Here, $m,n$ are numbers such that $\{\ell,m,n\}=\{1,2,3\}$. 
In particular, $\widetilde{a}_\ell$ and $\bar{\alpha}_{\mathrm{II}*}\,\widetilde{a}_\ell$, $\widetilde{b}_\ell$ and $\bar{\alpha}_{\mathrm{II}*}\,\widetilde{b}_\ell$ belong to the same bordism classes of $\Omega_*(\partial V\times S^1)$.
\end{enumerate}
\end{Prop}

We shall prove Proposition~\ref{prop:wh-prod} in the rest of this subsection.
\par\medskip

\noindent{\bf Type I:} Since $b_\ell$ is parallel to $\partial T$, the relation $\alpha_{\mathrm{I}}(1)_*\,b_\ell  \simeq b_\ell$ is immediate. We shall consider the change of $a_\ell$. We have seen in Lemma~\ref{lem:T-bundle-I} that $V'$ is the mapping cylinder of a relative diffeomorphism $\varphi_0$ of $T$ (Lemma~\ref{lem:T-bundle-I}). The differential of the restriction of the natural map $T\times I\to V'=C(\varphi_0)$ on $\{x\}\times I$ gives a gradient-like vector field $\nu$ on $V'$. By taking $T=T\times\{1\}$ along the flow $\Phi_{-\nu}:\R\times V'\to V'$ until it gets to the bottom face, $\varphi_0$ is obtained.

\begin{wrapfigure}[9]{r}[0pt]{3cm}
\begin{center}
 \includegraphics[height=30mm]{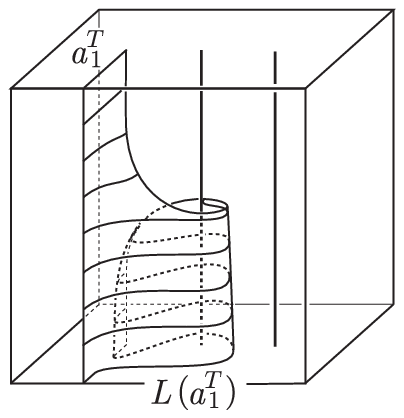}
\end{center}
\end{wrapfigure}
Let $L(a_\ell^T)$ be the locus in $V'$ of $a_\ell^T$ under the flow. 
Put $S=T\times\{0\}$ and $a_\ell^S=L(a_\ell^T)\cap S$. In order to prove (\ref{eq:whitehead_1}), we consider the difference of $a_\ell^T$ and $a_\ell^S$, both considered as disks in $T$ under the identification $S=T$. Since $\varphi_0$ is a relative diffeomorphism of $T$, the boundaries of $a_\ell^T$ and $a_\ell^S$ agree.

Let $p$ be the dimension of the disk $a_\ell^T$, which is 1 or 2. By definition, $L(a_\ell^T)$ is the image of a level preserving embedding $\lambda:D^p\times I\to V'=C(\varphi_0)$. We choose a base point $c\in \partial D^p$ so that $\lambda(c,0)$ is mapped to a point of $\partial T_0\times\{0\}$ that is disjoint from the attaching map of the handle $h_\ell^{e_\ell}$, and put 
$U=(D^p\times \{0\})\cup (\{c\}\times I)\cup (\partial D^p\times \{1\})$, $O=\partial D^p\times \{1\}$, $P=(D^p\times \{0\})\cup (\partial D^p\times I)$, $Q=D^p\times \{1\}$. The following is observed. 
\begin{enumerate}
\item $\lambda$ is standard on $U$.
\item The relative homotopy class of $\lambda|_P:(P, U)\to (V',\lambda(U))$ represents a unique element of $\pi_p(V',\lambda(O))$ since $\lambda(U)$ deformation retracts to $\lambda(O)$. 
\item The relative homotopy class of $\lambda|_Q:(Q,O)\to (V',\lambda(O))$ represents an element of $\pi_p(S,\lambda(O))$.
\item The inclusion map $S\to V'$ induces a bijection $\pi_p(S,\lambda(O))\cong \pi_p(V',\lambda(O))$. In the following, we identify the two by this correspondence.
\end{enumerate}

\begin{Lem}\label{lem:lambda_P_Q}
The identity $[\lambda|_P]=[\lambda|_Q]$ in $\pi_p(V',\lambda(O))$ holds.
\end{Lem}
\begin{proof}
We denote the boundary $P\cup (-Q)$ of $D^p\times I$ by $P-Q$. The map $\lambda_{P-Q}:(P-Q,O)\to (V',\lambda(O))$ induced by $\lambda$ extends to a relative map $\lambda:(D^p\times I,O)\to (V',\lambda(O))$. Hence $\lambda_{P-Q}$ is relatively nullhomotopic, fixing $\lambda(O)$.
\end{proof}

\begin{proof}[Proof of Proposition~\ref{prop:wh-prod} (1)]
The Borromean string link $B(\underline{2},\underline{2},\underline{1})_4$ can be deformed by isotopy in $D^4$ into one with two trivial components around which the remaining component links. If it is the $\ell$-th component, the homotopy class of the difference between $\alpha_I(1)_*a_\ell$ and $a_\ell$ is given by the third component. The homotopy type of the complement $D^4-(D^p\cup D^q)$ of the two trivial components is the same as $S^{3-p}\vee S^{3-q}$, and as is well-known, the third component in the Borromean string link $D^p\cup D^q\cup D^r\to D^4$, considered as an $r$-disk in $D^4-(D^p\cup D^q)$, represents the Whitehead product in $[\pi_{3-p},\pi_{3-q}]\subset \pi_r\,S^{3-p}\vee S^{3-q}$ (\cite[\S{5}]{Ma}). This together with Lemma~\ref{lem:lambda_P_Q} completes the proof.
\end{proof}

\noindent{\bf Type II:} Recall that $\widetilde{V}$ is a 1-parameter family of mapping cylinders of relative diffeomorphisms of $T$. 
\begin{proof}[Proof of Proposition~\ref{prop:wh-prod} (2)]
Put $\widetilde{T}=T\times S^1$, and let $\widetilde{a}_\ell^T$ be the disk $\widetilde{a}_\ell\cap \widetilde{T}$. After the replacements $V'\to \widetilde{V}$, $a_\ell^T\to \widetilde{a}_\ell^T$, $T\to \widetilde{T}$, the arguments for type I can be applied here and Proposition~\ref{prop:wh-prod} (2) is proved similarly. 
\end{proof}

%%%%%%%%%%%%%%%%%%%%%%%%%%%%%%%
\subsection{Primitiveness of $\pi^\Gamma$}\label{ss:primitiveness}

We shall prove Theorem~\ref{thm:Z(G)} (1) and the first half of Theorem~\ref{thm:Z(G)} (2).

\begin{Prop}[The first half of Theorem~\ref{thm:Z(G)} (2)]\label{prop:primitive}
The $(D^4,\partial)$-bundle $\pi^\Gamma:E^\Gamma\to B_\Gamma$ is bundle bordant to a $(D^4,\partial)$-bundle $\varpi^\Gamma:F^\Gamma\to S^k$. Namely, there exist a compact oriented $(k+1)$-cobordism $\widetilde{B}$ with $\partial\widetilde{B}=B_\Gamma\tcoprod (-S^k)$ and a $(D^4,\partial)$-bundle $\widetilde{\pi}:\widetilde{E}\to \widetilde{B}$ such that the restriction of $\widetilde{\pi}$ on $\partial \widetilde{B}$ agrees with $\pi^\Gamma$ and $\varpi^\Gamma$.
\end{Prop}

Let $k_\ell$ be the index of the $\ell$-th handle of $V$ with nonzero index and let $V_{[\ell]}\subset D^4$ be the handlebody obtained from $V$ by attaching a $(k_\ell+1)$-handle along a $k_\ell$-sphere on $\partial V$ that is parallel to the core of the $\ell$-th handle in $V$, more precisely along $b_\ell$. By extending the bundle projection $\widetilde{V}\to S^a$ by  the trivial family of $(k_\ell+1)$-handles, one obtains a $(V_{[\ell]},\partial)$-bundle $\widetilde{V}_{[\ell]}\to S^a$. 

Let us describe a standard model for $\widetilde{V}_{[\ell]}$. Recall that $\widetilde{V}$ is given by the complement of the family $\beta$ of framed string links of $\pi_0\fEmb_0(\underline{D}^2\cup \underline{D}^2\cup\underline{D}^1,D^4)$ or $\pi_1\fEmb_0(\underline{D}^2\cup \underline{D}^1\cup\underline{D}^1,D^4)$ corresponding to the Borromean rings. Put $U=T_0\times I$, $L_n^2=\{(n,0,z,w)\mid z,w\in I\}$, $L_n^1=\{(n,0,0,w)\mid w\in I\}$ ($n=1,2,3$), and let $H_n^e$ ($e=1,2$) be the open $\ve$-tubular neighborhood of $L_n^e\subset U$ for a sufficiently small $\ve>0$. Let $L_1^{e_1}(t)\cup L_2^{e_2}(t)\cup L_3^{e_3}(t)\subset U$ ($t\in S^a$) be a family of framed string links corresponding to $\beta$ such that $L_\ell^{e_\ell}(t_0)=L_\ell^{e_\ell}$ at the base point $t_0\in S^a$. Let $H_n^e(t)$ be the open $\ve$-tubular neighborhood of $L_n^e(t)\subset U$. Put
\[ \begin{split}
  &W(t)=U-(H_1^{e_1}(t)\cup H_2^{e_2}(t)\cup H_3^{e_3}(t)),\quad W_{[\ell]}(t)=W(t)\cup H_\ell^{e_\ell}(t), \\
  &\widetilde{W}=\bigcup_{t\in S^a} W(t)\times\{t\},\quad \widetilde{W}_{[\ell]}=\bigcup_{t\in S^a} W_{[\ell]}(t)\times\{t\}.
\end{split} \]
Framing on string link gives trivializations $\partial V\times S^a\cong \partial \widetilde{W}$, $\partial V_{[\ell]}\times S^a\cong \partial \widetilde{W}_{[\ell]}$, which are extended to bundle isomorphism $\widetilde{V}\cong \widetilde{W}$, $\widetilde{V}_{[\ell]}\cong \widetilde{W}_{[\ell]}$. Then we may consider $\widetilde{W}$, $\widetilde{W}_{[\ell]}$ as standard models for $\widetilde{V}$, $\widetilde{V}_{[\ell]}$.

Although the following lemma is independent from Proposition~\ref{prop:primitive}, it will be used in Lemma~\ref{lem:occupied}, which plays a key role in computing the value of the invariant.

\begin{Lem}\label{lem:framing}
Let $\tau_0$ be the $SO_4$-framing on a single fiber $V$ induced from the standard framing of $D^4$. Let $\widetilde{\tau}_0$ be the vertical framing on $T^v\widetilde{V}|_{\partial \widetilde{V}=\partial V\times S^a}$ induced from $\tau_0$ by the product structure. Then there exists a vertical framing $\widetilde{\tau}$ on $\widetilde{V}$ that extends $\widetilde{\tau}_0$.
\end{Lem}
\begin{proof}
It suffices to prove the lemma for $V=W(t_0)$. Since $\widetilde{\tau}_0$ on $\partial \widetilde{V}$ induces a vertical framing $\widetilde{\upsilon}_0$ on $\partial\widetilde{W}$ by the differential of the trivialization $\partial\widetilde{V}=\partial V\times S^a\cong \partial\widetilde{W}$ slightly extended to its neighborhood, we need only to find a vertical framing on $\widetilde{W}$ that extends $\widetilde{\upsilon}_0$ induced from outside. Since $\widetilde{\tau}_0$ for $V=W(t_0)$ is induced from the direct sum of normal and tangent framings of $L_n^{e_n}$ on the side of each handle, $\widetilde{\upsilon}_0$ is induced from the direct sum of normal and tangent framings of $L_n^{e_n}(t)$ on the side of each handle, too (Figure~\ref{fig:framing} (2)). 

\begin{figure}
\begin{center}
\includegraphics[height=30mm]{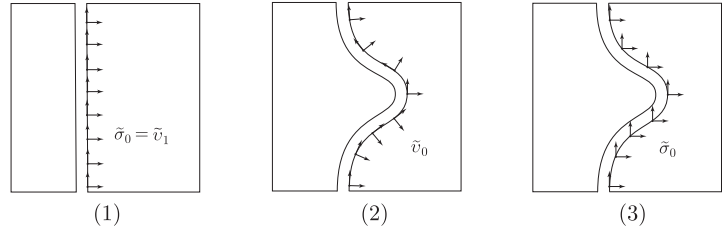}
\end{center}
\caption{Two framings on $\partial \widetilde{W}$}\label{fig:framing}
\end{figure}

Since $\widetilde{W}$ is a codimension 0 submanifold of $U\times S^a$, the standard $SO_4$-framing of $U$ induces a standard vertical framing $\widetilde{\sigma}_0$ on $\widetilde{W}$ (Figure~\ref{fig:framing} (1), (3)). We would like to glue $\widetilde{v}_0$ from outside and $\widetilde{\tau}_0$ from inside of $\widetilde{V}$ after a suitable homotopy. Since $\widetilde{\upsilon}_0$ and $\widetilde{\sigma}_0$ may not agree on $\partial \widetilde{W}$, we shall replace $\widetilde{\sigma}_0$ with the deformed one in a neighborhood of the boundary. It follows from a property of the Borromean rings that each component $L_i^{e_i}(t)$ is fiberwise regularly homotopic to $L_i^{e_i}$ as a framed string knot in $U$. Hence there is a regular homotopy $F(t,s)$ ($s\in I$) from $\overline{H_i^{e_i}(t)}$ to $\overline{H_i^{e_i}}$ in $U$. A $GL_4(\R)$-framing $\widetilde{\upsilon}_s$ on the immersion $F(t,s):\overline{H_i^{e_i}}\to U$ is induced from the standard framing of $\overline{H_i^{e_i}}\subset U$ by $d(F(t,s)):T\overline{H_i^{e_i}}\to TU$. At $s=0$, $\widetilde{\upsilon}_s$ agrees with $\widetilde{\upsilon}_0$ defined above, and at $s=1$, we have $\widetilde{\upsilon}_1=\widetilde{\tau}_0=\widetilde{\sigma}_0$ (Figure~\ref{fig:framing} (1)) since $\overline{{H}_i^{e_i}}$ is a standard embedding. The difference between $\widetilde{\upsilon}_s$ and $\widetilde{\sigma}_0$ gives a map $\overline{H_i^{e_i}}\to GL_4(\R)$, and the restriction of this map on $\partial\overline{H_i^{e_i}}$ gives a homotopy from $\widetilde{\upsilon}_0$ to $\widetilde{\sigma}_0$ in $\partial\widetilde{W}$. Therefore, we may deform $\widetilde{\sigma}_0$ by a homotopy on a neighborhood of $\partial \widetilde{W}$ so that it agrees with $\widetilde{\upsilon}_0$ on $\partial \widetilde{W}$. Let $\widetilde{\tau}$ be the vertical framing of $\widetilde{W}$ thus obtained. Then obviously this extends $\widetilde{\upsilon}_0$ inside $\widetilde{V}$, as desired.
\end{proof}

\begin{Lem}\label{lem:V[l]-trivial}
The bundle $\widetilde{V}_{[\ell]}\to S^a$ is a trivial $(V_{[\ell]},\partial)$-bundle, and hence the corresponding classifying map $S^a\to B\Diff(V_{[\ell]},\partial)$ is nullhomotopic.
\end{Lem}
\begin{proof}
We only prove the lemma for $\ell=1$ since the proof for the other cases are similar. Replacing $W(t)$ with $W_{[1]}(t)$ corresponds to removing the first component in the string link $L_1^{e_1}(t)\cup L_2^{e_2}(t)\cup L_3^{e_3}(t)$ in $U$. It follows from a property of Borromean rings that the remaining components $L_2^{e_2}(t)\cup L_3^{e_3}(t)$ is fiberwise isotopic to the trivial ones $L_2^{e_2}\cup L_3^{e_e}$ (Figure~\ref{fig:reg_ho}). The fiberwise isotopy can be given by shrinking $L_2^{e_2}(t)$ along its standard spanning disk, and shrinking $L_3^{e_3}(t)$ along a curved spanning disk, which avoids $L_2^{e_2}(t)$. This implies that the bundle $\widetilde{V}_{[1]}\to S^a$ can be deformed into the trivial one. 
\begin{figure}
\includegraphics[height=25mm]{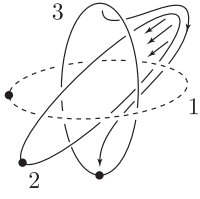}
\caption{Collapsing $L_2^{e_2}(t)\cup L_3^{e_3}(t)$, when the first component is removed.}\label{fig:reg_ho}
\end{figure}
\end{proof}

\begin{Cor}[Theorem~\ref{thm:Z(G)} (1)]
If $\alpha=\alpha_\mathrm{I}$ or $\alpha_\mathrm{II}$, then the bundle $\pi(\alpha):(D^4\times S^a)^{V,\alpha}\to S^a$ obtained from the trivial $D^4$-bundle $D^4\times S^a$ by surgery along $V$ is a trivial $(D^4,\partial)$-bundle. Hence the bundle $\pi^\Gamma:E^\Gamma\to B_\Gamma$ obtained by doing surgery $2k$ times is a $(D^4,\partial)$-bundle, too.
\end{Cor}
\begin{proof}
Since the trivial subbundle $V\times S^a$ of $D^4\times S^a$ extends to a trivial subbundle $V_{[\ell]}\times S^a$, surgery of $D^4\times S^a$ along $V$ and surgery along $V_{[\ell]}$ using $\widetilde{V}_{[\ell]}$ produce equivalent results. By Lemma~\ref{lem:V[l]-trivial}, the result is a trivial $D^4$-bundle.
\end{proof}

Let $V_{[\ell,\ell']}\subset\R^4$ $(\ell\neq \ell'$) be the handlebody obtained from $V$ by attaching a $(k_\ell+1)$- and a $(k_{\ell'}+1)$-handle along the cycles $b_{\ell}$ and $b_{\ell'}$, which are parallel to the $\ell$-th and $\ell'$-th handle of positive indices $k_\ell,k_{\ell'}$ respectively. The handlebody $V_{[\ell,\ell',\ell'']}\subset\R^4$ is defined similarly. Also, $\widetilde{V}_{[\ell,\ell']}$, $\widetilde{V}_{[\ell,\ell',\ell'']}$ etc. are defined similarly as $\widetilde{V}_{[\ell]}$.
\begin{Lem}\label{lem:B}
\begin{enumerate}
\item The nullhomotopies of the classifying maps for $\widetilde{V}_{[\ell]}$ and for $\widetilde{V}_{[\ell']}$ of Lemma~\ref{lem:V[l]-trivial} are homotopic through nullhomotopies for $\widetilde{V}_{[\ell,\ell']}$.
\item There is a $\Delta^2$-family of nullhomotopies of the classifying maps for $\widetilde{V}_{[\ell,\ell',\ell'']}$ that extends the homotopies (between nullhomotopies) for $\widetilde{V}_{[\ell,\ell']}$, $\widetilde{V}_{[\ell',\ell'']}$, $\widetilde{V}_{[\ell'',\ell]}$ on the boundary.
\end{enumerate}
\end{Lem}
\begin{proof}
  (1) We consider nullhomotopies of the classifying maps for $\widetilde{V}_{[1,2]}$. Let $A_i$ be a fiberwise isotopy deformation of $\widetilde{V}_{[1,2]}$ induced by the standard compressing regular homotopy of $L_i^{e_i}(t)$ given in the proof of Lemma~\ref{lem:V[l]-trivial}. This is a path in the space of bundle isomorphisms of $\widetilde{V}_{[1,2]}$ starting from $\mathrm{id}$. Let $A_i'$ be a fiberwise isotopy deformation of $\widetilde{V}_{[1,2]}$ induced by the curved regular homotopy of $L_i^{e_i}(t)$ that avoids $L_{i-1}^{e_{i-1}}(t)$ (considering $L_{-1}=L_3$ and $e_{-1}=e_3$), as in Figure~\ref{fig:reg_ho}. Note that $A_i$ and $A_i'$ are defined only if all the strings that intersects the regular homotopy of the $i$-th component are deleted. Then there are fiberwise isotopies from $A_i'$ to $A_i$. Now the nullhomotopy for $\widetilde{V}_{[1]}$ is induced by the composite path $A_3'\circ A_2$, and that for $\widetilde{V}_{[2]}$ is induced by $A_3\circ A_1'$. A homotopy from $A_3'\circ A_2$ to $A_3\circ A_1'$ in $\widetilde{V}_{[1,2]}$ is constructed in 2 steps: $A_3'\circ A_2\to A_3'\to A_3\circ A_1'$. In the first step, shrink the $A_2$ deformation of $L_2^{e_2}(t)$ to that of the trivial one $L_2^{e_2}$. Note that the $A_2$ deformation on $L_2^{e_2}$ is trivial. In the second step, homotope $A_3'$ to $A_3$, which is possible. The deformation $\mathrm{id}\to A_1'$ is similar to $A_2\to \mathrm{id}$.

(2) Consider a hexagon $P_1P_2P_3P_4P_5P_6$ around the origin $O$ (Figure~\ref{fig:hexagon}). We put the nullhomotopies $A_2\circ A_3'$, $A_3'$, $A_3\circ A_1'$, $A_1'$, $A_1\circ A_2'$, $A_2'$ on the vertices $P_1,P_2,P_3,P_4,P_5,P_6$ respectively. As in the proof of (1), one can find homotopies for edges and extend them to a family of nullhomotopies parametrized by the hexagon. The homotopy on the path $P_1P_2P_3$ is for $\widetilde{V}_{[1,2]}$, $P_3P_4P_5$ is for $\widetilde{V}_{[2,3]}$, $P_5P_6P_1$ is for $\widetilde{V}_{[3,1]}$. This completes the proof.
\end{proof}
\begin{figure}
 \includegraphics[height=50mm]{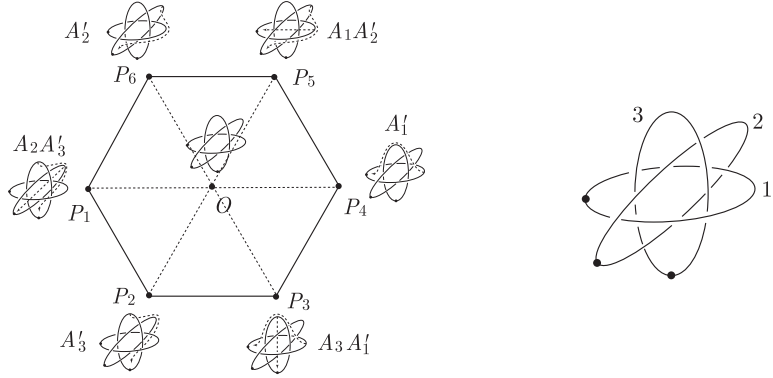} 
\caption{A 2-parameter family of nullhomotopies of the classifying maps for $\widetilde{V}_{[1,2,3]}$.}\label{fig:hexagon}
\end{figure}

Let $B_\Gamma[i]$ be the subspace of $B_\Gamma=K_1\times \cdots\times K_{2k}$ obtained from $B_\Gamma$ by replacing $K_i$ with base point $\{*\}\subset K_i$. The following lemma is an analogue of the fact that surgery along a graph clasper with unlinked leaf is trivial up to isotopy (\cite{Hab}). 

\begin{Lem}\label{lem:null-leaf}
Let $\tilde{f}(\pi^\Gamma):B_\Gamma\to B\Diff(D^4,\partial)$ be the classifying map for $\pi^\Gamma:E^\Gamma\to B_\Gamma$. 
For $0\leq i\leq 2k$, the restriction of $\tilde{f}(\pi^\Gamma)$ to $B_\Gamma[i]$ is nullhomotopic. 
\end{Lem}
\begin{proof}
The restriction of $\pi^\Gamma$ on $B_\Gamma[i]$ is obtained by surgery along the tuple $\vec{V}_G[i]$ of $2k-1$ handlebodies that is obtained from $\vec{V}_G=(V_1,\ldots,V_{2k})$ by removing $V_i$. Since $V_j$ ($j\neq i$) with a handle $h_\ell^{(j)}$ that links with a handle of $V_i$ is included in $\vec{V}_G[i]$, and the handle $h_\ell^{(j)}$ of $V_j$ does not link with any handles of other handlebodies in $\vec{V}_G[i]$, $V_j$ can be extended to $V_{j[\ell]}$ by attaching a disk that is disjoint from other handlebodies in $\vec{V}_G[i]$. Thus surgery along $V_{j[\ell]}$ and that along $V_j$ produces equivalent $(D^4,\partial)$-bundles. Moreover, since by Lemma~\ref{lem:V[l]-trivial}, surgery along $V_{j[\ell]}$ can be deformed in $V_{j[\ell]}$ to the trivial surgery, it is equivalent to surgery along the tuple $\vec{V}_G[i,j]$, which is obtained from $\vec{V}_G[i]$ by removing $V_j$. By repeating similar arguments, the bundle $(\pi^\Gamma)^*E^\Gamma\to B_\Gamma[i]$ can be deformed to the trivial $(D^4,\partial)$-bundle. Namely, the restriction of $\tilde{f}(\pi^\Gamma)$ to $B_\Gamma[i]$ is nullhomotopic. The nullhomotopy on $B_\Gamma[i]$ is constructed along a spanning tree $L$ of $\Gamma$, by gradually extending a neighborhood of the $i$-th vertex in $L$.
\end{proof}

\noindent{\it Proof of Proposition~\ref{prop:primitive}}.\ \ 
Put $B_\Gamma^*=\bigcup_{i=1}^{2k}B_\Gamma[i]$. By Lemma~\ref{lem:null-leaf}, the classifying map $\psi_\Gamma:(B_\Gamma,b_0)\to (B\Diff(D^4,\partial),*)$ for $\pi^\Gamma:E^\Gamma\to B_\Gamma$ is homotopic to a map $(B_\Gamma,B_\Gamma[i])\to (B\Diff(D^4,\partial),*)$. To prove that $\tilde{f}(\pi^\Gamma)$ is bordant to a map from $S^{k}=K_1\wedge\cdots\wedge K_{2k}= B_\Gamma/B_\Gamma^*$, we need to check that the nullhomotopies on $B_\Gamma[i]$ and $B_\Gamma[i']$ for $i\neq i'$ induce homotopically consistent nullhomotopies on $B_\Gamma[i]\cap B_\Gamma[i']=B_\Gamma[i][i']$ (i.e., the paths associated to the nullhomotopies are homotopic rel. ends). Fix a spanning tree $L$ of $\Gamma$ and assume that the nullhomotopies on $B_\Gamma[i]$ for every $i$ are constructed inductively along $L$. The two vertices $i$ and $i'$ of $\Gamma$ are connected by a unique path $\gamma_{i,i'}$ in $L$. The essential difference between the induced nullhomotopies on $B_\Gamma[i][i']$ from $B_\Gamma[i]$ and $B_\Gamma[i']$ arises from different orders of collapses at vertices on $\gamma_{i,i'}$: $i\to\cdots\to i'$ and $i'\to \cdots\to i$. By using Lemma~\ref{lem:B} (1), the order of collapses along $\gamma_{i,i'}$ can be reversed through homotopies. This proves the homotopical consistency at $B_\Gamma[i][i']$. 
\begin{wrapfigure}[7]{r}[0pt]{2.5cm}
\includegraphics[height=20mm]{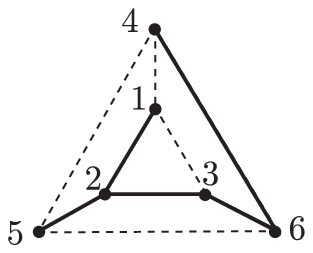}
\end{wrapfigure}
For example, we consider a nullhomotopy on $B_\Gamma[1][6]$ for the spanning tree $L$ in the figure on the right.
Suppose that a nullhomotopy induced from that of $B_\Gamma[1]$ is given by a sequence $(1)\to 2\to 3\to (6)\to 4\to 5$ of collapses, where $(1)$ and $(6)$ correspond to constant homotopies at $V_1$ and $V_6$ respectively. By a continuous parameter change, this is homotopic to $(6)\to (1)\to 2\to 3\to 4\to 5$. The collapse at the vertex 3 is applied after 2, which can be homotoped to the one applied after $(6)$, by Lemma~\ref{lem:B} (1) for the vertex 3. After the homotopy, we obtain $(6)\to 3\to (1)\to 2\to 4\to 5$, followed by $(6)\to 3\to 2\to (1)\to 4\to 5$ by Lemma~\ref{lem:B} (1) again. Now the path $(1)\to 2\to 3\to (6)$ has been reversed and the resulting path gives a nullhomotopy induced from that of $B_\Gamma[6]$. 

The homotopical consistency at deeper intersections, such as $B_\Gamma[i][i'][i'']$ etc., can be checked by using Lemma~\ref{lem:B} (3) at the trivalent vertex of $L$ connecting three input vertices. 

Therefore, one may extend the classifying map $\psi_\Gamma$ to one from $B_\Gamma\cup_{B_\Gamma^*} C$ for a CW complex $C$ that is homotopy equivalent to the cone over $B_\Gamma^*$ with a base point $c_0$ corresponding to the cone point that is mapped to the base point of $B\Diff(D^4,\partial)$. The complex $B_\Gamma\cup_{B_\Gamma^*} C$ can be considered as a deformation retract of its small neighborhood $U$ in $\R^N$ for some $N$. Since $B_\Gamma\cup_{B_\Gamma^*} C$ is homotopic to a cone over $B_\Gamma^*$ embedded in $U$, and $U$ also deformation retracts onto the cone, there is a homotopy of $\psi_\Gamma$ that takes $B_\Gamma^*$ to the base point. The result is bordant to a map from $S^k=B_\Gamma/B_\Gamma^*$.
\qed

%%%%%%%%%%%%%%%%%%%%%%%%%%%%%%%
%%%%%%%%%%%%%%%%%%%%%%%%%%%%%%%
\mysection{Computation of the invariant}{s:computation}

In this section, we complete the proof of Theorem~\ref{thm:Z(G)} (2). The main idea is to prove that after modifying the $v$-gradient suitably we need only to count Z-graphs whose trivalent vertices are caught by the handlebodies $V_1\cup\cdots\cup V_{2k}$ for the surgery (Key Lemma~\ref{lem:occupied}) and that the sum of counts of such Z-graphs are given by homological intersections among relative cycles in the handlebodies (Lemma~\ref{lem:homological-formula}). 

%%%%%%%%%%%%%%%%%%%%%%%%%%%%%%%
\subsection{Fiberwise Morse functions compatible with Y-links}\label{ss:fiberwise-MF}

Let $G=G_1\cup G_2\cup\cdots\cup G_{2k}\subset D^4$ be the Y-link given in \S\ref{ss:Y-link} and $V_1,V_2,\ldots,V_{2k}$ be the handlebodies associated to $G$. We take a Morse function $\mu_i:V_i\to (-\infty,0]$ for each $i$ satisfying the following conditions.
\begin{itemize}
\item $\mu_i^{-1}(0)=\partial V_i$.
\item $\mu_i$ has four critical points. If $G_i$ is a Y-graph of type I, then the indices of critical points of $\mu_i$ are $0,1,1,2$. If $G_i$ is a Y-graph of type II, then the indices of critical points of $\mu_i$ are $0,1,2,2$.
\item $\mu_i$ gives a standard handle decomposition of $V_i$ via some gradient-like vector field.
\end{itemize}
We extend the Morse function $\mu_1\cup\cdots\cup\mu_{2k}$ on $V_1\cup \cdots \cup V_{2k}$ to a Morse function $m:\R^4\to \R$ so that it satisfies Assumption~\ref{hyp:eta} (4) (Figure~\ref{fig:h-adapted}). Moreover, we assume that $m$ satisfies $m^{-1}(-\infty,0]=V_1\cup\cdots\cup V_{2k}\cup \R^4_-$ ($\R^4_-$ is a half-plane). Let $m^{(1)},\ldots,m^{(3k)}:\R^4\to \R$ be a sequence of Morse functions that satisfy Assumption~\ref{hyp:eta} (4) that are obtained from $m$ by compact support small perturbations and small $SO_4$-rotations outside compact sets. We assume $(m^{(j)})^{-1}(0)=\partial V_1\cup\cdots\cup \partial V_{2k}\cup g_j\cdot\partial \R^4_-$ and $(m^{(j)})^{-1}(-\infty,0]=V_1\cup\cdots\cup V_{2k}\cup g_j\cdot\R^4_-$ ($g_j\in SO_4$) for each $j$. 

\begin{figure}
\includegraphics[height=30mm]{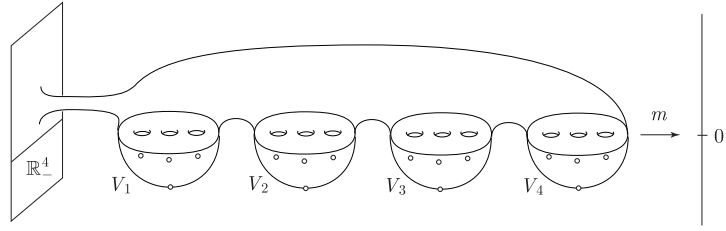}
\caption{Morse function on $\R^4$ adapted to a $Y$-link}\label{fig:h-adapted}
\end{figure}

Next, we extend $m$ to a fiberwise Morse function on a $(\R^4,U_\infty')$-bundle. Let $f$ be the function on the product $\R^4\times B_\Gamma$ defined by $f=m\circ\mathrm{pr}_1$, where $\mathrm{pr}_1:\R^4\times B_\Gamma\to \R^4$ is the projection. Suppose that $\R^4$ is equipped with a Riemannian metric that is standard outside a compact set and whose gradient vector field for $m$ is gradient-like for $m$, and that $\R^4\times B_\Gamma$ is equipped with a product Riemannian metric. Let $\xi$ be the $v$-gradient of $f$ that is obtained from the gradient of $m$. Since the surgery along $G$ can be performed along the 0-level surface locus of $f$, it naturally induces a fiberwise Morse function $f^{G}:E^\Gamma\to \R$. Similarly, for $f^{(j)}=m^{(j)}\circ\mathrm{pr}_1:\R^4\times B_\Gamma\to \R$, fiberwise Morse functions $f^{(j)G}:E^\Gamma\to \R$ are induced. Let $\xi^{(1)},\ldots,\xi^{(3k)}$ be the fiberwise gradient vector fields for $f^{(1)},\ldots,f^{(3k)}$, and let $\xi^{(1)G},\ldots,\xi^{(3k)G}$ be the fiberwise gradient vector fields for $f^{(1)G}\,\ldots,f^{(3k)G}$. Here, $G$ may be replaced by its subset $G'=\bigcup_{i\in I}G_i$, and $f^{(j)G'}$, $\xi^{(j)G'}$ etc. are defined similarly.

Let $B^1_\Gamma,B^2_\Gamma,\ldots,B^{2^k}_\Gamma$ be the path-components of $B_\Gamma$. Each component is a copy of the $k$-torus $S^{1}\times\cdots\times S^{1}$. We take the Morse functions $\lambda:S^{1}\to \R$, $\lambda(e^{i\theta})=\mathrm{Re}\,{e^{i\theta}}=\cos{\theta}$ and $\lambda':S^{1}\times\cdots\times S^{1}\to \R$, $\lambda'(x_1,\ldots,x_k)=\lambda(x_1)+\cdots+\lambda(x_k)$. Let $\eta$ be a Morse--Smale gradient-like vector field for $\lambda'$. This gives an $h$-gradient for $\pi^\Gamma$. We take the maximal point of $\lambda'$ as a base point $b_0$ of $S^{1}\times\cdots\times S^{1}$, and let $b^n_0$ be the corresponding base point of $B^n_\Gamma$.

%%%%%%%%%%%%%%%%%%%%%%%%%%%%%%%
\subsection{Coherent $v$-gradients}\footnote{The content of this subsection was inspired by Pajitnov's $\mathfrak{C}$-approximation theorem \cite[Ch.~8]{Pa}.}

We would like to count Z-graphs for the $v$-gradients $\xi^{(1)G},\ldots,\xi^{(3k)G}$ and the $h$-gradient $\eta$. 
To simplify the computation, we take convenient $v$-gradients. Let $p_1^{(i)},p_2^{(i)},p_3^{(i)}$ be the three critical points of the Morse function $\mu_i:V_i\to (-\infty,0]$ corresponding to $p_1^{(i)},p_2^{(i)},p_3^{(i)}$ of $\mu_i$. Let $\widetilde{V}_i$ be the subset of $(D^4\times K_i)^{V_i,\alpha_i}$ given by the component of $(f^{G_i})^{-1}(-\infty,0]$ corresponding to $V_i$. We also denote by $p_1^{(i)},p_2^{(i)},p_3^{(i)}$ the critical loci of the induced fiberwise Morse function $\widetilde{\mu}_i:\widetilde{V}_i\to (-\infty,0]$ of index 1 or $2$. Let $B_\ell^{(i)}$ ($\ell=1,2,3$) be the intersection of $\calA_{p_\ell^{(i)}}(\xi)$ with $V_i\times K_i$ and let $A_\ell^{(i)}$ ($\ell=1,2,3$) be the intersection of $\calA_{p_\ell^{(i)}}(\xi^{G_i})$ with $\widetilde{V_i}$. 

The restrictions $A_\ell^{(i)}\cap (\partial V_i\times K_i)$ $(\ell=1,2,3)$ and $B_\ell^{(i)}\cap (\partial V_i\times K_i)$ $(\ell=1,2,3)$ both forms a disjoint triple of simple cycles in $\partial V_i\times K_i$. Note that $A_\ell^{(i)}\cap (\partial V_i\times K_i)$ and $B_\ell^{(i)}\cap (\partial V_i\times K_i)$ may intersect and by Proposition~\ref{prop:wh-prod}, they are bordant in $\partial V_i\times K_i$. Let $C_\ell^{(i)}$ be an oriented bordism between $A_\ell^{(i)}\cap (\partial V_i\times K_i)$ and $B_\ell^{(i)}\cap (\partial V_i\times K_i)$ in $\partial V_i\times K_i$ whose image lies in $(\partial V_i-\{*\})\times K_i$ for some point $*$.

\begin{Lem}\label{lem:ribbon-graph-I}
Suppose that $G_i$ is a Y-graph of type I, in which case $K_i=S^0$. The gradient-like vector fields for $\widetilde{\mu}_i$ and the bordism $C_\ell^{(i)}$ can be chosen so that they satisfy the following.
\begin{enumerate}
\item There exists an embedded handlebody $T_i\subset \partial V_i\times K_i$ with a handle filtration $T_i^{(0)}\subset T_i^{(1)}\subset T_i^{(2)}=T_i$ with at most 2-handles and $C_\ell^{(i)}$ is a map into $T_i$ for $\ell=1,2,3$.
\item The bordism $C_\ell^{(i)}$ is ``parallel'' to $T_i$, where we say that a map $\beta$ from a manifold $S$ into a handlebody $T$ is {\it parallel} to $T$ if the restriction of $\beta$ to the preimage of each $r$-handle $\sigma$ of $T$ with $r> 0$ is a (not necessarily disjoint) union of bundle maps $D^r\times F\to \sigma=D^r\times D^{3-r}$ over $D^r$ for some compact manifolds $F$.
\item The 0-, 1- and $2$-handles of $T_i$ can be taken as arbitrarily small tubular neighborhoods of their cores. Namely, for every $\ve>0$, we may arrange that the diameters of 0-handles are less than $\ve$ and that the thickness of 1- or $2$-handles are less than $\ve/2$. 
\end{enumerate}
We say that such a $v$-gradient $\xi^{G_i}$ is \emph{coherent}. 
\end{Lem}
\begin{proof}
Since the surgery on $G_i$ is trivial on $-1\in K_i$ (Definition~\ref{def:alpha-I}), the cycles $A_\ell^{(i)}$ and $B_\ell^{(i)}$ agree on $-1$. Thus $A_\ell^{(i)}\cap (\partial V_i\times\{-1\})$ and $B_\ell^{(i)}\cap (\partial V_i\times\{-1\})$ are connected by a trivial bordism, that is, the one of the form $x+(-x)$, and it satisfies the conditions of the lemma trivially.

On $1\in K_i$, they may differ by surgery. Let $\delta_\ell^{(i)}$ be an oriented bordism in $\partial V_i\times\{1\}$ such that $\partial \delta_\ell^{(i)}=B_\ell^{(i)}\cap (\partial V_i\times\{1\})-A_\ell^{(i)}\cap (\partial V_i\times\{1\})$. Then $\delta_\ell^{(i)}$ is given by a smooth map from a manifold of dimension $4-|p_\ell^{(i)}|=2$ or $3$. We take a minimal Morse function $\nu:\partial V_i\to \R$ whose critical points of $\nu$ are disjoint from $\partial\delta_\ell^{(i)}$ for all $\ell$. Then the critical points of $\nu$ form a basis of the Morse homology of $\nu$. 
Now we define a filtration $T_i^{(0)}\subset T_i^{(1)}\subset T_i^{(2)}=T_i$ by the handle filtration for the Morse--Smale gradient-like vector field of $\nu$ up to index $2$. 
Then we modify $\delta_\ell^{(i)}$ by applying the negative gradient flow $\Phi_{-\nu}:\R\times \partial V_i\to \partial V_i$ for $\nu$. For a sufficiently large $T$, the bordism $\Phi_{-\nu}(T,\cdot)\circ \delta_\ell^{(i)}$ concentrates on a small neighborhood of the $2$-skeleton of the cellular decomposition of $\partial V_i$ with respect to the gradient-like vector field for $\nu$ (Figure~\ref{fig:converge-ribbon}) and in particular, the image of the bordism is included in $T_i$. The conditions (1) and (3) are fulfilled. Moreover, we may assume that the intersections of the bordisms $\delta_\ell^{(i)}$ and the horizontal ascending manifolds of critical points of $\nu$ of positive indices are all transversal. Hence we may assume that the restriction of $\Phi_{-\nu}(T,\cdot)\circ \delta_\ell^{(i)}$ for large $T$ to the preimages of 1- or 2-handles of $T_i$ is a sum of morphisms $D^1\times F\to  D^1\times D^{2}$ or $D^{2}\times F\to D^{2}\times D^1$ of bundles over the cores of handles, where $F$ is some compact manifold with boundary. The condition (2) is fulfilled. All the modifications above can be realized by isotopies of $\partial V_i$ and the isotopies can be realized by modifying the gradient-like vector field for $\mu_i$. 
\begin{figure}
\includegraphics[height=25mm]{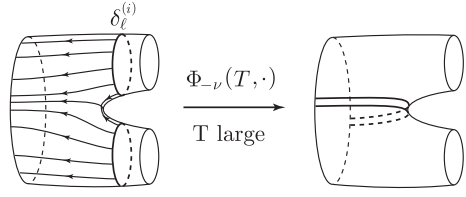}
\caption{Immersed handlebody is obtained after a long time gradient flow}\label{fig:converge-ribbon}
\end{figure}
\end{proof}

Next, we consider the case $K_i=S^{1}$ and $V_i$ is of type II. Let $T_i^{(0)}\subset T_i^{(1)}\subset T_i^{(2)}=T_i$ be a handle filtration of $\partial V_i$ obtained from a minimal Morse function $\nu:\partial V_i\to \R$, as in the proof of Lemma~\ref{lem:ribbon-graph-I}. We also take a handle filtration $U^{(0)}\subset U^{(1)}=S^{1}$ of the base space $S^1$ such that $U^{(0)}$ is a small disk around a point, and $U^{(0)}$ is disjoint from the two critical points of $\lambda:S^1\to \R$. Then 
\[ \widetilde{T}_i^{(m)}=\bigcup_{p+q\leq m}T_i^{(p)}\times U^{(q)}, \]
gives a handle filtration $\widetilde{T}_i^{(0)}\subset \widetilde{T}_i^{(1)}\subset \widetilde{T}_i^{(2)}\subset \widetilde{T}_i^{(3)}=\widetilde{T}_i\subset \partial V_i\times S^{1}$ of $\partial V_i\times S^1$ with at most 3-handles. The direct product of a $p$-handle of $\{T_i^{(p)}\}$ and a $q$-handle of $\{U^{(q)}\}$ is a $(p+q)$-handle. We define its index by $(p,q)$ and call it a $(p,q)$-handle. 

\begin{Lem}\label{lem:ribbon-graph-II}
Suppose that $G_i$ is a Y-graph of type II, in which case $K_i=S^{1}$. The gradient-like vector fields for $\widetilde{\mu}_i$, $\widetilde{T}_i$, and the bordism $C_\ell^{(i)}$ can be chosen so that they satisfy the following. Put $c_\ell^{(i)}=3-|p_\ell^{(i)}|$.
\begin{enumerate}
\item $C_\ell^{(i)}$ is a map into $\widetilde{T}_i$ for $\ell=1,2,3$.
\item If $|p_\ell^{(i)}|=1$, then the handles of $\widetilde{T}_i$ that may intersect $C_\ell^{(i)}$ are of indices $(0,1), (2,1), (0,0), (1,0), (2,0)$.
\item If $|p_\ell^{(i)}|=2$, then the handles of $\widetilde{T}_i$ that may intersect $C_\ell^{(i)}$ are of indices $(0,1), (1,1), (0,0), (1,0), (2,0)$.
\item In a $(c_\ell^{(i)},1)$-handle, $C_\ell^{(i)}$ is a trivial bordism. Namely, a trivial family of immersed subbundle that is parallel to a $c_\ell^{(i)}$-handle of $T_i$. 
\item In a $(1,0)$- or $(2,0)$-handle, $C_\ell^{(i)}$ is the sum of trivial bordism as in (4) and an immersed subbundle that is parallel to the handle. 
\item The handles of $\widetilde{T}_i$ can be taken as arbitrarily small tubular neighborhoods of their cores. Namely, for every $\ve>0$, we may arrange that the diameters of 0-handles are less than $\ve$ and that the thickness of handles with indices $(1,0), (2,0), (0,1)$ are less than $\ve/2$, and that the thickness of handles with indices $(1,1), (2,1)$ are less than $\ve/4$.
\end{enumerate}
We say that such a $v$-gradient $\xi^{G_i}$ is \emph{coherent}. 
\end{Lem}
\begin{proof}
Let $\delta_\ell^{(i)}$ be an oriented bordism in $\partial V_i\times S^1$ such that $\partial \delta_\ell^{(i)}=B_\ell^{(i)}\cap (\partial V_i\times S^1)-A_\ell^{(i)}\cap (\partial V_i\times S^1)$. We take a minimal Morse function $\nu:\partial V_i\to \R$. Let $\widetilde{\nu}$ be the fiberwise Morse function on the trivial bundle $\partial V_i\times S^{1}\to S^{1}$ given by $\widetilde{\nu}=\nu\circ \mathrm{pr}_1$. We modify $\alpha_{\mathrm{II}}$ by a homotopy so that the support of the diffeomorphism of $\partial V_i$ for the surgery is disjoint from all the critical points of $\nu$, and the projection of the support on $S^1$ is included in the interior of the small disk $U^{(0)}\subset S^{1}$. Such a modification is possible by the definition of $\alpha_\mathrm{II}$. Here, for sufficiently large $T>0$, the bordism $\Phi_{-\widetilde{\nu}}(T,\cdot)\circ \delta_\ell^{(i)}$ satisfies the conditions (1), (2), (3). Moreover, since the twists of $\partial V_i$ over the 1-handle of $\{U^{(q)}\}$ is done by a trivial family, we may take a trivial bordism there, and (4) is satisfied. Over $U^{(0)}$, where the family $A_\ell^{(i)}\cap (\partial V_i\times K_i)$ of cycles may be nontrivial, the family of cycles is the connected sum of $B_\ell^{(i)}\cap (\partial V_i\times K_i)$ and a loop that represents a Whitehead product, by Proposition~\ref{prop:wh-prod}. Since the connecting part for the connected sum may be assumed to be disjoint from the horizontal ascending manifolds of $\widetilde{\nu}$, we may assume that the connecting part would be included in the $(0,0)$-handle after taking it by the flow $\Phi_{-\widetilde{\nu}}(T,\cdot)$ for large $T$. Hence, the difference of the two families of cycles over $U^{(0)}$ is represented by a nullbordant $(c_\ell^{(i)}+1)$-loop in $\partial V_i\times U^{(0)}$ whose base point is in the $(0,0)$-handle. It does not intersect the locus of the maximal point of $\nu$, and thus intersects only $(0,0)$, $(1,0)$ or $(2,0)$-handle. After taking the bordism by a long time gradient flow of $-\widetilde{\nu}$, it becomes parallel on $(1,0)$ and $(2,0)$-handle, by a similar reason as in the proof of Lemma~\ref{lem:ribbon-graph-I}. Now (5) has been proved. 
\end{proof}

We remark that the bordisms $C_\ell^{(i)}$ for different $\ell$ may intersect each other.
\begin{Def}\label{def:coherent}
We say that the $v$-gradient $(f^G,\xi^G)$ is \emph{coherent} if its restriction on $V_i$ is the direct product of the restriction of the coherent one $(f^{G_i},\xi^{G_i})$ over $K_i$ with $K_1\times\cdots\times \widehat{K_i}\times\cdots\times K_{2k}$ (the $i$-th factor excluded). We say that the $v$-gradient $(f^G,\xi^G)$ is \emph{generic with respect to $\vec{V}_G=(V_1,\ldots,V_{2k})$} if a property in question is satisfied if for each $i$ the $v$-gradient $\xi^{(i)G}$ is perturbed slightly, preserving the following conditions.
\begin{enumerate}
\item The restriction of $\xi^{(i)G}$ to $(\R^4-\mathrm{Int}\,(V_1\cup\cdots\cup V_{2k}))\times B_\Gamma$ is the constant family with respect to the parameter in $B_\Gamma$.
\item The restriction of $\xi^{(i)G}$ to $\bigcup_{(t_1,\ldots,t_{2k})\in B_\Gamma}(V_1(t_1)\cup\cdots\cup V_{2k}(t_{2k}))$ is of the form $\bigcup_{(t_1,\ldots,t_{2k})}(\xi^{(i)G_1}(t_1)\cup\cdots\cup \xi^{(i)G_{2k}}(t_{2k}))$.
\end{enumerate}
\end{Def}

Lemma~\ref{lem:0-mfd} still holds for a $v$-gradient $(f^G,\xi^G)$ that is generic with respect to $\vec{V}_G$. See \S\ref{ss:transversality-surgery} for a proof.

Let $F_\ell^{(i)}$ be the cycle in the closed manifold $S_i=\widetilde{V}_i\cup_{\partial} (-V_i\times K_i)$
defined by
\[ F_\ell^{(i)}=A_\ell^{(i)}-B_\ell^{(i)}+C_\ell^{(i)}. \]
We denote the class of $F_\ell^{(i)}$ in $H_*(S_i)$ by $\alpha_\ell^{(i)}$.

\begin{Lem}\label{lem:aaa=1}
\begin{enumerate}
\item $\langle \alpha_\ell^{(i)},\alpha_{\ell'}^{(i)}\rangle_{S_i} =0$ for $\ell,\ell'\in\{1,2,3\}$,
\item $\langle \alpha_1^{(i)},\alpha_2^{(i)},\alpha_3^{(i)}\rangle_{S_i} =\pm 1$.
\end{enumerate}
\end{Lem}
\begin{proof}
(1) follows from the homological property of the Borromean surgery in Proposition~\ref{prop:wh-prod}. Since the left hand side of (2) is the triple intersection among cycles in closed manifold $S_i$, the result follows from \cite[Lemma~5.3]{Wa2}. Roughly speaking, $\alpha_\ell^{(i)}$ can be represented by the cycle obtained by gluing ``immersed Seifert surface'' of a component of the Borromean string link in its complement $\widetilde{V}_i$ with that of a trivial string link in $-V_i\times K_i$. The triple intersection can be seen in the coordinate description of the Borromean rings in \S\ref{ss:borromean}.
\end{proof}

Since critical points of the Morse function $m$ on fiber do not change by the surgery along $G$, the correspondence between critical points gives a canonical $\Z$-linear isomorphism $\phi_n:C_*(f_{b_0^n})\to C_*(f_{b_0^n}^G)$. 

\begin{Lem}\label{lem:morse-complex-invariant}
If $(f^{G},\xi^{G})$ is coherent and generic with respect to $\vec{V}_G$, then the following hold.
\begin{enumerate}
\item[(a)] For each $n$, $\phi_n$ is a chain isomorphism.
\item[(b)] $(f^G,\xi^G,\eta)$ satisfies Assumption~\ref{hyp:cv-const}.
\item[(c)] $(f^G,\xi^G,\eta)$ satisfies Assumption~\ref{hyp:eta}.
\end{enumerate}
\end{Lem}

\begin{proof}
(a) It suffices to check that the boundary operator is invariant under the surgery on $G_i$ of type I. The $i+1/i$-flow-lines that result in the difference between the counts of the flow-lines are those from critical points outside $V_i$ to those inside $V_i$. Hence the change of the count under surgery is given by the intersection of the chains $\calD_p(\xi_{b_0^n})\cap \partial V_i$ and $(A_k^{(i)}\cup (-B_k^{(i)}))\cap \partial V_i$ for some $k$. Generically, $\calD_p(\xi_{b_0^n})\cap \partial V_i$ intersects transversally with cores of the handles in the handlebody $T_i$ of Lemma~\ref{lem:ribbon-graph-I} in a single fiber over $b_0^n$. Since it intersects boundaries of each parallel bordism in $C_k^{(i)}$ for a coherent $v$-gradient, the intersection number cancels in pair, which shows that the intersection of $\calD_p(\xi_{b_0^n})\cap \partial V_i$ and $(A_k^{(i)}\cup (-B_k^{(i)}))\cap \partial V_i$ is zero.

(b) That the critical values of $f^G$ are constant over $B_\Gamma$ is clear by the definition of $f^G$. For each pair $a,b\in\Sigma(\eta)$ with $|a|=|b|+1$ and for a flow-line $\alpha:I\to B_\Gamma$ of $-\eta$ between them we shall prove that $n_\alpha(x,y)=0$ for $x\neq y$, $x,y\in\Sigma(\xi^G)$, where $n_\alpha(x,y)$ was defined in (\ref{eq:n(x,y)}). For this we only need to prove that the count of the transversal intersections between $\calD_x(\xi^G)$ and $\calA_y(\xi^G)$ over $\alpha$, i.e., the $j/j$-intersections between $x$ and $y$, is 0. By the definition of $\xi^G$, the flow-lines in question that may be counted are those between critical points outside and inside $V_i$. It suffices to consider the case where $\alpha$ intersects the support of the diffeomorphisms for a surgery of type II since otherwise $n_\alpha(x,y)=0$ is obvious. In that case, the count of the flow-lines in question are given by the intersection of the chains $\calD_x(\xi^G)\cap \partial \widetilde{V}_i$ and $(A_k^{(i)}\cup (-B_k^{(i)}))\cap \partial \widetilde{V}_i$ for some $k$. By coherence and Lemma~\ref{lem:ribbon-graph-II}, the latter fits into a thin handlebody in $\partial \widetilde{V}_i$ and generically $\calD_x(\xi^G)\cap \partial \widetilde{V}_i$ intersects transversally with the cores of the thin handles with positive indices. Also, since the intersection with the parallel pieces of $(A_k^{(i)}\cup (-B_k^{(i)}))\cap \partial \widetilde{V}_i$ is that with the boundary of the parallel bordisms in thin handles, it is generically 0. Thus $n_\alpha(x,y)=0$ for $x\neq y$ follows.

(c) Assumption~\ref{hyp:eta} (1) is satisfied by choosing $\xi^G$ generically with respect to $\vec{V}_G$. Assumption~\ref{hyp:eta} (2), (4) are clear from the definition of $(f^G,\xi^G)$. For Assumption~\ref{hyp:eta} (3), we prove that there are no $i/i+\ell$-intersections ($\ell\geq 1$) for $\xi^G$. Put $E^\Gamma_+=(f^G)^{-1}[0,\infty)$, $E^\Gamma_-=(f^G)^{-1}(-\infty,0]$. Inside $E^\Gamma_+$, the $v$-gradient $\xi^G$ is a trivial family, and thus there are no $i/i+\ell$-intersections ($\ell\geq 1$) there. Also, inside $E^\Gamma_-$, there are no $i/i+\ell$-intersection ($\ell\geq 1$) by the definition of the surgery of $v$-gradient. Hence there may be only $i/i+\ell$-intersections between critical points of $E^\Gamma_+$ and $E^\Gamma_-$ for $\ell\geq 1$. Since the indices of critical points in $E^\Gamma_-$ are 0,1,2 and $0/1$-intersection is impossible, there may be only $1/2$-intersections. For a critical locus $p$ of index 1 in $E^\Gamma_+$, $\calD_p(\xi^{G_i})\cap \partial \widetilde{V}_i$ is a trivial $\dim{K_i}$-parameter family of 0-manifolds. By coherence of $\xi^{G_i}$, this may be assumed to be disjoint from $T_i$ or $\widetilde{T}_i$ of Lemmas~\ref{lem:ribbon-graph-I} and \ref{lem:ribbon-graph-II}, and thus may also be assumed to be disjoint from the ascending manifolds of critical points of index 2 in $E^\Gamma_-$. Namely, there are no $1/2$-intersections. This completes the proof. 
\end{proof}

%%%%%%%%%%%%%%%%%%%%%%%%%%%%%%%
\subsection{Key Lemma}

The following is a key lemma to simplify the computation. Put 
\[ \widetilde{V}_i'=K_1\times\cdots\times K_{i-1}\times \widetilde{V}_i\times K_{i+1}\times\cdots\times K_{2k}.\]
This is a bundle over $B_\Gamma=K_1\times\cdots\times K_{2k}$, which is the pullback of the bundle $\widetilde{V}_i\to K_i$ by the projection $B_\Gamma\to K_i$. 
\begin{Lem}[Key Lemma]\label{lem:occupied}
Let $k\geq 1$. If $(f^{(j)G},\xi^{(j)G})$ is coherent for each $j\in\{1,2,\ldots,3k\}$, generic with respect to $\vec{V}_G$, and if the width $\ve$ in Lemmas~\ref{lem:ribbon-graph-I}, \ref{lem:ribbon-graph-II} is sufficiently small, then we have the following.
\begin{enumerate}
\item The value of $\hat{Z}_k^\Morse$ for $\pi^\Gamma$ (Theorem~\ref{thm:GCF}) agrees with the sum 
\[ \frac{(-1)^{3k}}{2^{3k}}\sum_{\ve_i=\pm 1}Z_k^\Morse((\ve_1\xi^{(1)G},\cdots, \ve_{3k}\xi^{(3k)G}),\eta) \]
for the coherent $v$-gradients $(\xi^{(1)G},\cdots, \xi^{(3k)G})$. We denote the value of $Z_k^\Morse((\xi^{(1)G},\ldots,\xi^{(3k)G}),\eta)$ by $Z_k^\Morse(\pi^\Gamma)$.
\item The Z-graphs that  may contribute to $Z_k^\Morse(\pi^\Gamma)$ are such that for each $i$ there are exactly one black vertex in $\widetilde{V}'_i$. (We say that such a Z-graph \emph{occupies} $G$.)
\end{enumerate}
\end{Lem}
\begin{proof}
(1) Since the manifold $\widetilde{V}_i$ can be obtained from $V_i\times K_i$ by spherical modifications, namely, the Borromean rings ($B(2,2,1)_4$ or $B(3,2,2)_5=B(2+1,1+1,1+1)_{4+1}$, see Lemma~\ref{lem:B(3,2,2)}) can be obtained by surgery on three Hopf links, similar to \cite[Figure~1]{GGP}, one can find a compact oriented manifold $T_i$ with $\partial T_i=S_i=\widetilde{V}_i\cup_\partial (-V_i\times K_i)$. Put $W^0=D^4\times B_\Gamma\times I$, which gives a $(k+5)$-cobordism between the trivial $D^4$-bundles $D^4\times B_\Gamma\times\{1\}$ and $-D^4\times B_\Gamma\times\{0\}$. Let $\widetilde{B}_\Gamma$ be an oriented cobordism between $B_\Gamma$ and $-S^k$ and put $W^1=D^4\times \widetilde{B}_\Gamma$. We construct a $(k+5)$-cobordism $W^\Gamma$ with corners between $E^\Gamma$ and $-D^4\times S^k$ by modifying $W^0\cup_{D^4\times B_\Gamma\times\{0\}} W^1$. Let $\alpha_i':B_\Gamma\to \Diff(\partial V_i)$ be the composition of the projection $B_\Gamma\to K_i$ with $\alpha_i:K_i\to \Diff(\partial V_i)$. By attaching the $(k+5)$-manifold $T_i'=K_1\times\cdots\times K_{i-1}\times T_i\times K_{i+1}\times\cdots\times K_{2k}$ to $W^0\cup W^1$ along $V_i\times\prod_{j=1}^{2k} K_j=K_1\times\cdots\times K_{i-1}\times(V_i\times K_i)\times K_{i+1}\times\cdots\times K_{2k}$ in $D^4\times B_\Gamma\times\{1\}$, a $(k+5)$-cobordism $T_i'\cup_{V_i\times\prod_j K_j}(W^0\cup W^1)$ between $(D^4\times B_\Gamma)^{V_i,\alpha_i'}$ and the trivial $D^4$-bundle over $S^k$ is obtained. Moreover, by attaching $T_1',\ldots,T_{2k}'$ to $W^0\cup W^1$ along $(V_1\cup\cdots \cup V_{2k})\times\prod_j K_j$ simultaneously in a similar manner, a $(k+5)$-dimensional cobordism $W^\Gamma=(T_1'\cup\cdots\cup T_{2k}')\cup(W^0\cup W^1)$ between $E^\Gamma$ and the trivial bundle over $S^k$ is obtained. By Proposition~\ref{prop:primitive} and Lemma~\ref{lem:framing}, this gives a pair of a bundle and cobordism that satisfies Assumption~\ref{hyp:W}.

Let us prove that the correction term of $\hat{Z}_k^\Morse$ for $W^\Gamma$ vanishes. There is an $SO_4$-framing in $TW^0$ induced by the structure of trivial $D^4$-bundle, which spans a trivial linear $\R^4$-bundle $T^vW^0$. On the $D^4\times B_\Gamma\times\{1\}$ side, there is an admissible section of $ST^vW^0|_{D^4\times B_\Gamma\times\{1\}}$ given by the trivial family of the $v$-gradient $\xi^{(j)}$, and on the $-D^4\times B_\Gamma\times\{0\}$ side, there is a tuple $\vec{v}$ of admissible sections of $ST^vW^0|_{-D^4\times B_\Gamma\times\{0\}}$ given by constant sections. We take a tuple $\vec{\eta}_{W^0}$ of admissible sections of $ST^vW^0$ that extends the two, which may be assumed to be independent of the parameter of $B_\Gamma$. We take a tuple $\vec{\eta}_{W^1}$ of admissible sections of $ST^vW^1$ by the constant admissible sections $\vec{v}$. Also, we can find a tuple $\vec{\eta}_{T_i'}$ of admissible sections on $T_i'$ as follows. By Lemma~\ref{lem:framing}, there is a vertical framing on $\widetilde{V}_i$ that extends the standard framing on the boundary. Thus one can find a trivialization $\tau_{S_i}$ of the rank 4 vector bundle $T^v S_i$, which is obtained by gluing $T^v\widetilde{V}_i$ and $T^v(V_i\times K_i)$ along the boundary, where the fibers of $T^vS_i$ over $\partial \widetilde{V}_i$ are not tangent to $S_i$. By gluing the trivial vector bundle $\R^4\times T_i$ to $T^vS_i$ by using the bundle isomorphism $\tau_{S_i}$, we obtain a trivial vector bundle $T^v T_i\to T_i$ that extends $T^vS_i$. There is a tuple of admissible sections of $T^vT_i$ that extends those of $T^vS_i$ given by the $v$-gradients on $\widetilde{V}_i$ and $V_i\times K_i$. Extending this trivially to $T_i'$ by the product structure, we obtain a tuple $\vec{\eta}_{T_i'}$ of admissible sections of $T^vT_i'$. We define the tuple $\vec{\eta}_{W^\Gamma}$ of admissible sections of $T^vW^\Gamma$ by setting $\vec{\eta}_{W^0}$ on $W^0$, $\vec{\eta}_{W^1}$ on $W^1$, and $\vec{\eta}_{T_i'}$ on $T_i'$. Then clearly we have
\[ \alpha_k^\adm(\vec{\eta}_{W^\Gamma})=\alpha_k^\adm(\vec{\eta}_{W^0})+\alpha_k^\adm(\vec{\eta}_{W^1})+\sum_{i=1}^{2k}\alpha_k^\adm(\vec{\eta}_{T_i'}). \]

Since $W^0$ is a trivial $D^4$-bundle and $\vec{\eta}_{W^0}$ is a trivial family with respect to $B_\Gamma$, we have $\alpha_k^\adm(\vec{\eta}_{W^0})=0$ by a dimensional reason. Similarly, $\alpha_k^\adm(\vec{\eta}_{W^1})=0$ by a dimensional reason. Also, for $k\geq 2$ and for each $i$, there is $j\neq i$ such that $V_j$ is of type I. Then $T_i'$ is decomposed into two parts by the coordinate of $K_j=\{-1,1\}$, and accordingly $\alpha_k^\adm(\vec{\eta}_{T_i'})$ is decomposed into two parts. Since the result is of the form $A-A$, its sum $\alpha_k^\adm(\vec{\eta}_{T_i'})$ is 0 and thus $\alpha_k^\adm(\vec{\eta}_{W^\Gamma})=0$ for $k\geq 2$. For $k=1$, if $V_i$ is of type II, then the proof is the same as for $k\geq 2$. If $k=1$ and $V_i$ is of type I, then $T_1'=T_1\times S^1$. Since $\vec{\eta}_{T_1'}$ is defined as a trivial 1-parameter family of admissible sections on $T_1$, we have $\alpha_k^\adm(\vec{\eta}_{T_1'})=0$ by a dimensional reason. When $k+5\equiv 0$ (mod 4), we need to prove $P_k(W^\Gamma;\tau_{E^\Gamma})=0$. The proof of this fact is the same as the vanishing of $\alpha_k^\adm(\vec{\eta}_{W^\Gamma})$ above by using the additivity of the integral and by dimensional reasons. Thus (1) has been proved. 

Next we prove (2). If a Z-graph $I$ with $2k$ black vertices does not meet $\widetilde{V}'_i$, then the count of such $I$ vanishes by a dimensional reason since we assume that the $v$-gradients outside $\widetilde{V}_i'$ are constant families over $K_i$. If a Z-graph $I$ with $2k$ black vertices meets $\widetilde{V}'_i$ but no black vertices of $I$ lie there, then there must be white vertices of $I$ in $\widetilde{V}'_i$. Let $p$ be the critical point attached to one of such white vertices. If $|p|=0$, then the Z-graph $I$ can be interpreted as a 0-chain in the bundle of configuration space of points in fiber of $\overline{E^\Gamma\setminus \widetilde{V}'_i}$.  This shows that the count of $I$ vanishes by a dimensional reason. If $|p|=1$ or $2$, then $p=p_k^{(i)}$ for some $k$ and the Z-graph $I$, viewed as a 0-chain in the configuration space of points in the fibers of $E^\Gamma$, can be interpreted as the transversal intersection of the ascending manifold locus of $p$ and a piecewise smooth chain of $\overline{E^\Gamma\setminus \widetilde{V}'_i}$. \emph{Now we use the coherence of $(f^{(j)G},\xi^{(j)G})$.} The count of $I$ is the transversal intersection of the locus of the (upward) gradient flow from the boundary of a thin parallel piece in $C_k^{(i)}$ of Lemma~\ref{lem:ribbon-graph-I} or \ref{lem:ribbon-graph-II} with the piecewise smooth chain of $\overline{E^\Gamma\setminus \widetilde{V}'_i}$. Since the locus from the parallel piece can be made arbitrarily thin, the number of the transversal intersections is 0. Namely, the Z-graphs that does not have any black vertices in $\widetilde{V}'_i$ for each $i$ does not contribute to $Z_k^\Morse$.
\end{proof}

\begin{Cor}\label{cor:Y-disjoint}
Let $k\geq 1$. Under the assumption of Lemma~\ref{lem:occupied}, only the Z-graphs that are the disjoint union of $2k$ $Y$-shaped components may contribute nontrivially to $Z_k^\Morse(\pi^\Gamma)$.
\end{Cor}
\begin{proof}
By Lemma~\ref{lem:occupied}, a Z-graph $I$ that may contribute to $Z_k^\Morse$ occupies $G$. Thus there is no compact edge in $I$ connecting two black vertices. Namely, every edge of $I$ must be separated one.
\end{proof}

As we will see in \S\ref{ss:eval-Z}, the counts of Y-shaped Z-graphs can be computed explicitly in terms of the homology of the closed manifold $S_i$. According to Lemma~\ref{lem:occupied} (1), the computation of the correction terms is no longer necessary and the value of the invariant can be computed exactly.

%%%%%%%%%%%%%%%%%%%%%%%%%%%%%%%
\subsection{Graphs of surviving type}\label{ss:surviving}

Let $I$ be a Z-graph in $E^\Gamma$ consisting only of $Y$-shaped components and suppose that $I$ occupies $G$. This implies that every output white vertices of a $Y$-shaped component are mapped to critical points in some $\widetilde{V}_\ell'$. Thus we may assume the following.
\begin{enumerate}
\item The index of the critical point attached on any output white vertex of $I$ is either of 0,1 or 2.
\item The index of the critical point attached on any input white vertex of $I$ is either of 1,2 or 3.
\end{enumerate}

We define the index of a $Y$-shaped component $J$ in $I$ as $(a_1,\ldots,a_r\,|\,a_{r+1},\ldots,a_3)$ if the indices of the critical points attached on the input white vertices of $J$ are $a_1,\ldots,a_r$ and if the indices of the critical points attached on the output white vertices of $J$ are $a_{r+1},\ldots,a_3$. Let $\ell\in \{1,\ldots,2k\}$ be such that the black vertex of $J$ lies in $\widetilde{V}_\ell'$. In fact, the moduli space of $J$ in $E^\Gamma$ is the pullback of the moduli space of $J$ in the restriction $E^\Gamma|_{K_\ell}$ of $E^\Gamma$ on $K_\ell$ under the projection $B_\Gamma\to K_\ell$. If there is a component $J$ whose moduli space in $E^\Gamma|_{K_\ell}$ has negative dimension, then the moduli space of the whole Z-graph including $J$ is empty. Hence 
{\it we need only to consider $J$ whose moduli space in $E^\Gamma|_{K_\ell}$ is exactly 0 dimensional. } Let $\acalM_J(\vec{\xi}^{G_\ell})=\acalM_J(\vec{\xi}^{G_\ell},\eta_\ell)$ denote the moduli space of Z-graphs from $J$ into $E^\Gamma|_{K_\ell}$, where $\eta_\ell$ is a Morse--Smale $h$-gradient for the Morse function $\lambda'|_{K_\ell}$ (\S\ref{ss:fiberwise-MF}). The condition that the dimension of the moduli space of $J$ in $E^\Gamma|_{K_\ell}$ is 0 is given by the equation
\[ \sum_{j=1}^r (4-a_j) + \sum_{j=r+1}^3 a_j 
= \left\{\begin{array}{ll}
4 & \mbox{if $V_\ell$ is of type I}\\
5 & \mbox{if $V_\ell$ is of type II}
\end{array}\right.
\]
All the indices of $J$ that satisfy this equation are as follows.
\begin{equation}\label{eq:list}
  \begin{split}
    \mbox{(Type I)}\quad & (2,3,3\,|\,),(\,|\,0,2,2),(\,|\,1,1,2), (1,3\,|\,0),(2,2\,|\,0),(2,3\,|\,1),(3,3\,|\,2),\\
    & (1\,|\,0,1),(2\,|\,0,2),(2\,|\,1,1),(3\,|\,1,2)\\
    \mbox{(Type II)}\quad & (1,3,3\,|\,),(2,2,3\,|\,),(\,|\,1,2,2), (1,2\,|\,0),(1,3\,|\,1),(2,2\,|\,1),(2,3\,|\,2),\\
    &(1\,|\,0,2),(1\,|\,1,1),(2\,|\,1,2),(3\,|\,2,2) 
  \end{split}
\end{equation}
Note that the edge-orientation of a $\vec{C}$-graph $J$ is independent of the type (I or II) of $J$. We say that $Y$-shaped component in a $\vec{C}$-graph is of {\it surviving type} if its index is on the above list (\ref{eq:list}).

%%%%%%%%%%%%%%%%%%%%%%%%%%%%%%%
\subsection{Algebraic Seifert surface of a leaf and linking number}\label{ss:dg}

Here, we shall give a formula for the chain $\bD_{g(p)}\cap \widetilde{V}_\ell$ for a combinatorial propagator in terms of the linking number. 

Let $\sigma$ be the fundamental cycle of $B_\Gamma^n$. In the following, we denote $\bD_*(\xi^G,\eta)_\sigma, \bA_*(\xi^G,\eta)_\sigma$ simply by $\bD_*, \bA_*$, respectively. Suppose that an $\R^4$-bundle $E_1\to S^1$ has sub $S^a$-bundle $\widetilde{c}$ and sub $S^{2-a}$-bundle $\widetilde{c}'$ with $\widetilde{c}\cap \widetilde{c}'=\emptyset$, and that $\tilde{c}$ bounds an $(a+1)$-chain $d$ in $E_1$. Then the linking number of $\widetilde{c}$ and $\widetilde{c}'$ is defined by $\widetilde{\ell k}(\tilde{c},\tilde{c}')=\langle d, \tilde{c}'\rangle$. 

\begin{Lem}\label{lem:dg(p)=p} Let $p,p'$ be critical points of Morse functions $\mu_i:V_i\to (-\infty,0]$ and $\mu_\ell:V_\ell\to (-\infty,0]$ ($i\neq \ell$) respectively with $|p|, |p'|>0$. We consider $p$ and $p'$ are points in the base fiber $F_0$. The following holds.
\begin{enumerate}
\item $\partial\bD_{g(p)}=\bD_p$, $\partial\bcalD_{g(p)}(\xi_{b_0^n})=\bcalD_p(\xi_{b_0^n})$.
\item $\langle \bD_{g(p)}, \bcalD_{p'}(\xi_{b_0^n})\rangle_{(\R^4\times K_\ell)^{V_\ell,\alpha_\ell}}=\ell k(c(p),c(p'))$ for $p'\neq p$, $|p|+|p'|=3$, where $c(p)$ etc. is the cycle in $\R^4$ given by $\bcalD_p(\xi_{b_0^n})$ etc.
\item If $V_\ell$ is of type II, then $\langle \bD_{g(p)}, \bD_{p'}\rangle_{(\R^4\times K_\ell)^{V_\ell,\alpha_\ell}}=\widetilde{\ell k}(\tilde{c}(p),\tilde{c}(p'))$ for $p'\neq p$, $|p|+|p'|=2$, where $\widetilde{c}(p)$ etc. is the cycle in $(\R^4\times K_\ell)^{V_\ell,\alpha_\ell}$ given by $\bcalD_p(\xi^G)$ etc. 
\end{enumerate}
\end{Lem}
\begin{proof} By $\overline{\partial} g(p)=\partial g(p)=(\partial g+g\partial)(p)=p$ and Proposition~\ref{prop:dD}, we have
\[ \partial\bD_{g(p)}=\sum_r \langle \overline{\partial} g(p),r\rangle\bD_r=\sum_r\langle p,r\rangle \bD_r=\bD_p. \]
Also, $\langle \bD_{g(p)},\bcalD_{p'}(\xi_{b_0^n})\rangle_{(\R^4\times K_\ell)^{V_\ell,\alpha_\ell}} = \langle \bcalD_{g(p)}(\xi_{b_0^n}),\bcalD_{p'}(\xi_{b_0^n})\rangle_{F_0}=\ell k(c(p),c(p'))$ and the identity (2) is proved. The identity of (3) follows from (1) and the definition of $\widetilde{\ell k}$.
\end{proof}

\begin{Lem}\label{lem:rel_cycle}
Let $p$ be one of critical points of $\mu_i:V_i\to (-\infty,0]$ with $0<|p|\leq 2$. Let $p_1,p_2,p_3$ be the critical loci of $\widetilde{\mu}_\ell:\widetilde{V}_\ell\to (-\infty,0]$ ($\ell\neq i$) of positive indices. Then $\bD_{g(p)}$ induces a (piecewise smooth singular) relative cycle of $(\widetilde{V}_\ell,\partial \widetilde{V}_\ell)$ on $K_\ell$, and its $\Z$-homology class is 
\[ \sum_j \ell k(c(p),c(p_j))\,[\bA_{p_j}]. \]
Here the sum is taken for $j$ such that $|p|+|p_j|=3$.
\end{Lem}
\begin{proof}
By Proposition~\ref{prop:wh-prod}, the families of twists of $\partial \widetilde{V}_\ell$ are homologically trivial and we have $H_*(\widetilde{V}_\ell,\partial \widetilde{V}_\ell)\cong H_*(V_\ell,\partial V_\ell)\otimes H_*(K_\ell)$, whose basis is given explicitly by the ascending manifolds of $p_j$. Thus we may put
\[ [\bD_{g(p)}\cap \widetilde{V}_\ell]=\left\{
\begin{array}{ll}
  \displaystyle\sum_j c_j[\bA_{p_j}] & (\mbox{$V_\ell$ is of type I})\\
  \displaystyle\sum_j c_j[\bA_{p_j}]+\sum_{j'} c_{j'}[\bcalA_{p_{j'}}(\xi_{b_0^n})] & (\mbox{$V_\ell$ is of type II})
\end{array}\right. \]
($c_j,c_{j'}\in\Z$). Here, the first sum in both lines is taken for $j$ such that $|p|+|p_j|=3$, and the second sum in the second row is taken for $j'$ such that $|p|+|p_{j'}|=2$. By Lemma~\ref{lem:dg(p)=p}, we have
$\langle \bD_{g(p)}\cap \widetilde{V}_\ell, \bcalD_{p_j}(\xi^G_{b_0^n})\rangle=\langle \bD_{g(p)}, \bcalD_{p_j}(\xi^G_{b_0^n})\rangle=\ell k(c(p),c(p_j)) $, $\langle \bA_{p_j},\bcalD_{p_{j'}}(\xi^G_{b_0^n})\rangle=\delta_{jj'}$, and thus we have $c_j=\ell k(c(p),c(p_j))$. If $V_\ell$ is of type II and $|p|+|p_j|=2$, then by Lemma~\ref{lem:dg(p)=p} again, we have $\langle \bD_{g(p)}\cap \widetilde{V}_\ell, \bD_{p_{j'}}\rangle=\widetilde{\ell k}(\tilde{c}(p),\tilde{c}(p_j)) $, $\langle \bcalA_{p_j},\bD_{p_{j'}}\rangle=\delta_{jj'}$, and thus we have $c_{j'}=\widetilde{\ell k}(\tilde{c}(p),\tilde{c}(p_{j'}))$. Moreover, by looking at the way of linking of the Y-link $G_1\cup\cdots\cup G_{2k}$, we see that all $\widetilde{\ell k}$ are 0 and hence the sum for $j'$ is 0.
\end{proof}

%%%%%%%%%%%%%%%%%%%%%%%%%%%%%%%
\subsection{Orientations of completely separated graphs}\label{ss:o-separated}

Let $H$ be an edge-oriented, vertex-oriented labelled trivalent graph with $2k$ vertices. Let $H^\circ$ be a $\vec{C}$-graph obtained from $H$ by replacing every edge with separated one. Then $H^\circ$ is a disjoint union of $Y$-shaped components. Let $H^\circ_i$ ($i=1,2,\ldots,2k$) be the $Y$-component of $H^\circ$ that includes the $i$-th black vertex. 

Recall that the orientation of a trivalent graph was defined by $o=(e_{1+}\wedge e_{1-})\wedge\cdots\wedge(e_{3k+}\wedge e_{2k-})$, where $\{e_{j+},e_{j-}\}$ is a half-edge decomposition of the $j$-th edge $e_j$, and this can be rewritten as $o=\tau_1\wedge\cdots\wedge \tau_{2k}$, $\tau_i=e_{p\pm}\wedge e_{q\pm}\wedge e_{r\pm}$, where $e_{p\pm}, e_{q\pm}, e_{r\pm}$ are the half-edges adjacent to the $i$-th vertex. In \S\ref{ss:arrow-graph}, we used an arrow on each edge $e$ to determine a decomposition of $e$ into half-edges $\{e_+,e_-\}$ such that $\deg\,{e_+}=1$ and $\deg\,{e_-}=2$. Here, we do not use arrows to avoid confusion with the edge-orienation of a $\vec{C}$-graph. Instead, we use the indices of the pair of critical points attached to the white vertices of a separated edge. Namely, we identify each segment in a separated edge with a half-edge and we define its degree (mod 2) by the index of the critical point attached to the segment. For example, if a separated edge $e$ consists of two segments $e',e''$ with critical points of indices 1 and 2 attached to the white vertices of $e'$ and $e''$ respectively, then we define $\deg\,e'=1$ and $\deg\,e''=2$, independent of the edge-orientation of $e$ for a $\vec{C}$-graph.

%%%%%%%%%%%%%%%%%%%%%%%%%%%%%%%
\subsection{Homological description of graph counting}\label{ss:eval-Z}

If the half-edges $e,e',e''$ around the $i$-th black vertex of $H^\circ$ gives $\tau_i$ in this order, and if the critical points that are attached to the white vertices of $e,e',e''$ are $p,p',p''$ respectively, then we write $H^\circ_i=H^\circ_i(\bvec{p})$ ($\bvec{p}=(p,p',p'')$). Also, if $H^\circ=H^\circ_1(\bvec{p}_1)\cup\cdots\cup H^\circ_{2k}(\bvec{p}_{2k})$, then we write $H^\circ=H^\circ(\bvec{p}_1,\ldots,\bvec{p}_{2k})$ and allow substitution $H^\circ(\bvec{p}_1',\ldots,\bvec{p}_{2k}')$ by different tuple $(\bvec{p}_1',\ldots,\bvec{p}_{2k}')$ with possibly different indices from $(\bvec{p}_1,\ldots,\bvec{p}_{2k})$. When $\bvec{p}_1,\ldots,\bvec{p}_{2k}$ are such that every separated edge is of degree 1 (in the sense of \S\ref{ss:gcf}) and that $H_i^\circ(\bvec{p}_i)$ is of surviving type for each $i$, we denote by $\mathfrak{S}_{2k}^{H^\circ}(\bvec{p}_1,\ldots,\bvec{p}_{2k})$ the subset of $\mathfrak{S}_{2k}$ consisting of bijections $\beta:\{1,2,\ldots,2k\}\to \{1,2,\ldots,2k\}$ such that $H^\circ_j(\bvec{p}_j)$ is of type I (resp. type II) in the list (\ref{eq:list}) of surviving type if and only if $V_{\beta(j)}$ is of type I (resp. type II). Let $\mathrm{sgn}'(\sigma)$ be defined by the identity: $\tau_{\sigma(1)}\wedge\cdots\wedge \tau_{\sigma(2k)}=\mathrm{sgn}'(\sigma)\,\tau_1\wedge\cdots\wedge\tau_{2k}$. Note that $\mathrm{sgn}'(\sigma)$ depends on the choice of $\bvec{p}_1,\ldots,\bvec{p}_{2k}$.
The following lemma is a consequence of Corollary~\ref{cor:Y-disjoint} and the observation in \S\ref{ss:surviving}.

\begin{Lem}\label{lem:I(H)}
Under the assumption of Lemma~\ref{lem:occupied}, the contribution in $2^{3k}(2k)!(3k)!\,Z_k^\Morse(\pi^\Gamma)$ of the Z-graphs from graphs of the form $H^\circ$ is 
\begin{equation}\label{eq:I(H)}
 \footnotesize\sum_{\tbvec{p}_1,\ldots,\tbvec{p}_{2k}} 
\sum_{{\sigma\in}\atop{\mathfrak{S}_{2k}^{H^\circ}(\tbvec{p}_1,\ldots,\tbvec{p}_{2k})}}
\mathrm{sgn}'(\sigma)\,\mathrm{Tr}_{\vec{g}}\Bigl(
\#\acalM_{H^\circ_1(\tbvec{p}_1)}(\vec{\xi}^{G_{\sigma(1)}})\cdots\#\acalM_{H^\circ_{2k}(\tbvec{p}_{2k})}(\vec{\xi}^{G_{\sigma(2k)}})H^\circ(\bvec{p}_1,\ldots,\bvec{p}_{2k})\Bigr),\end{equation}
where $\bvec{p}_1,\ldots,\bvec{p}_{2k}$ are such that every separated edge is of degree 1 and that $H_i^\circ(\bvec{p}_i)$ is of surviving type for each $i$. The orientation of $\acalM_{H^\circ_i(\tbvec{p}_i)}(\vec{\xi}^{G_{\sigma(i)}})$ is given by the coorientation in $\widetilde{V}_{\sigma(i)}$ that is the wedge of coorientations of descending or ascending manifolds associated with the vertex-orientation $\tau_i$. We denote (\ref{eq:I(H)}) by $I(\vec{\xi}^G,H)$. 
\end{Lem}

Each term $\#\acalM_{H^\circ_i(\tbvec{p}_i)}(\vec{\xi}^{G_\ell})$ is the triple intersection of some three chains in $\widetilde{V}_\ell$ formed along the vertex-orientation. For example, if among the three white vertices in $H^\circ_i(p_i,p_i',p_i'')$ the one with $p_i$ attached is incoming and the other two with $p_i',p_i''$ attached are outgoing, then $\#\acalM_{H^\circ_i(\tbvec{p}_i)}(\vec{\xi}^{G_\ell})=\langle \bD_{p_i},\bA_{p_i'},\bA_{p_i''}\rangle_{\widetilde{V}_\ell}$. The reason of $\mathrm{sgn}'(\sigma)$ is that the pullback of the projection of the moduli space $\acalM_{H_i^\circ(\tbvec{p}_i)}(\vec{\xi}^{G_{\sigma(i)}})$ on $K_{\sigma(i)}$ to $B_\Gamma$ has coorientation $\pm o(K_{\sigma(i)})$, where $\pm$ is determined by the orientation of the triple intersection. Then the moduli space of the whole graph $H^\circ$ projected on $B_\Gamma$ has coorientation $\pm o(K_{\sigma(1)})\wedge\cdots\wedge o(K_{\sigma(2k)})=\pm \mathrm{sgn}'(\sigma)\,o(K_1)\wedge\cdots\wedge o(K_{2k})$.

The term $\#\acalM_{H^\circ_1(\tbvec{p}_1)}(\vec{\xi}^{G_{\sigma(1)}})\cdots\#\acalM_{H^\circ_{2k}(\tbvec{p}_{2k})}(\vec{\xi}^{G_{\sigma(2k)}})H^\circ(\bvec{p}_1,\ldots,\bvec{p}_{2k})$ is an integer multiple of a graph with $3k$ inputs and $3k$ outputs. Let $y_1,y_2,\ldots,y_{3k}$ be the critical points that are attached on the input white vertices of $H^\circ(\bvec{p}_1,\ldots,\bvec{p}_{2k})$, and let $x_1,x_2,\ldots,x_{3k}$ be those on the corresponding output white vertices, so that $y_i$ and $x_i$ are attached on the $i$-th separated edge of $H^\circ$. The combinatorial structure of the graph $H$ gives the bijective correspondence $(\bvec{p}_1,\ldots,\bvec{p}_{2k})\leftrightarrow (x_1,\ldots,x_{3k};y_1,\ldots,y_{3k})$. The following lemma is evident.

\begin{Lem}\label{lem:R-X}
For a permutation $\sigma\in\mathfrak{S}_{2k}^{H^\circ}(\bvec{p}_1,\ldots,\bvec{p}_{2k})$, put
\[ \begin{split}
R_\sigma(x_1,\ldots,x_{3k};y_1,\ldots,y_{3k})&=\mathrm{sgn}'(\sigma)\,\#\acalM_{H^\circ_1(\tbvec{p}_1)}(\vec{\xi}^{G_{\sigma(1)}})\cdots\#\acalM_{H^\circ_{2k}(\tbvec{p}_{2k})}(\vec{\xi}^{G_{\sigma(2k)}}),\\
X_\sigma(x_1,\ldots,x_{3k};y_1,\ldots,y_{3k})&=R_\sigma(x_1,\ldots,x_{3k};y_1,\ldots,y_{3k})H^\circ(\bvec{p}_1,\ldots,\bvec{p}_{2k}).
\end{split}\]
Then $I(\vec{\xi}^G,H)$ can be rewritten as
\begin{equation}\label{eq:XR}
 \begin{split}
   &\sum_{{{x_1,\ldots,x_{3k}}\atop{y_1,\ldots,y_{3k}}}}\sum_{{\sigma\in}\atop{\mathfrak{S}_{2k}^{H^\circ}(\tbvec{p}_1,\ldots,\tbvec{p}_{2k})}}
\mathrm{Tr}_{\vec{g}}\Bigl(
X_\sigma(x_1,\ldots,x_{3k};y_1,\ldots,y_{3k})
\Bigr)\\
&=(-1)^{3k}\sum_{x_1,\ldots,x_{3k}}\sum_{{\sigma\in}\atop{\mathfrak{S}_{2k}^{H^\circ}(\tbvec{p}_1,\ldots,\tbvec{p}_{2k})}}\Bigl[R_\sigma(x_1,\ldots,x_{3k};g^{(1)}(x_1),\ldots,g^{(3k)}(x_{3k}))\,\mathrm{Close}(H^\circ)\Bigr].
\end{split}
 \end{equation}
Here, $x_1,\ldots,x_{3k},y_1,\ldots,y_{3k}$ are such that every separated edge is of degree 1 and that $H_i^\circ(\bvec{p}_i)$ is of surviving type for each $i$, $R_\sigma(x_1,\ldots,x_{3k};y_1,\ldots,y_{3k})$ is extended to sequences of chains by $\Z$-linearity and $\mathrm{Close}(H^\circ)$ is the trivalent graph that is obtained from the $\vec{C}$-graph $H^\circ$ by identifying the pair of white vertices in each separated edge. 
\end{Lem}

We will denote $\mathrm{sgn}'(\sigma)$ for $\bvec{x}=(x_1,\ldots,x_{3k};g^{(1)}(x_1),\ldots,g^{(3k)}(x_{3k}))$ by $\mathrm{sgn}'_{\tbvec{x}}(\sigma)$, which is determined by $x_1,\ldots,x_{3k}$ since we only count $\vec{C}$-graphs that are occupied. Under the assumption of Lemma~\ref{lem:occupied}, we have, for $k\geq 1$ and the coherent $v$-gradients $\xi^{(j)G}$ ($j=1,2,\ldots,3k$),  
\begin{equation}\label{eq:Z_I(x,H)}
 Z_k^\Morse(\pi^\Gamma)=\frac{1}{2^{3k}(2k)!(3k)!}\sum_HI(\vec{\xi}^G,H)
\end{equation}
by Corollary~\ref{cor:Y-disjoint} and Lemma~\ref{lem:I(H)}. Here, the sum is taken for edge-oriented labelled trivalent graphs $H$ with $2k$ trivalent vertices. Let $S_H(\sigma(1),\ldots,\sigma(2k))$, $\sigma\in\mathfrak{S}_{2k}$, denote the part of $I(\vec{\xi}^G,H)$ in (\ref{eq:Z_I(x,H)}) consisting of terms of Z-graphs such that for each $i$ the $i$-th black vertex lies in $\widetilde{V}_{\sigma(i)}$. Then we have
\[ Z_k^\Morse(\pi^\Gamma)=\frac{1}{2^{3k}(2k)!(3k)!}\sum_{\sigma\in\mathfrak{S}_{2k}}\sum_HS_H(\sigma(1),\ldots,\sigma(2k)). \]
The following lemma is evident from Lemma~\ref{lem:R-X}.
\begin{Lem}\label{lem:S_H}
The following identity holds.
\[
S_H(\sigma(1),\ldots,\sigma(2k))
=(-1)^{3k}\Bigl[\sum_{x_1,\dots,x_{3k}}\mathrm{sgn}'_{\tbvec{x}}(\sigma)\,Q_{\tbvec{x}}^{(1)}(\sigma)Q_{\tbvec{x}}^{(2)}(\sigma)\cdots Q_{\tbvec{x}}^{(2k)}(\sigma)\,\mathrm{Close}(H^\circ)\Bigr]
\]
Here, the sum is the same as (\ref{eq:XR}), $\bvec{x}=(x_1,\ldots,x_{3k};g^{(1)}(x_1),\ldots,g^{(3k)}(x_{3k}))$, 
\[ Q_{\tbvec{x}}^{(i)}(\sigma)=\left\{\begin{array}{ll}
\#\acalM_{H^\circ_i(\tbvec{p}_i)}(\vec{\xi}^{G_{\sigma(i)}}) & \mbox{if the types of $H^\circ_i(\bvec{p}_i)$ and $V_{\sigma(i)}$ agree,}\\
0 & \mbox{otherwise}
\end{array}\right.\]
and $\bvec{p}_1,\ldots,\bvec{p}_{2k}$ are sequences of linear combinations of critical points that correspond to $\bvec{x}$.
\end{Lem}

The Morse indices of the critical points $x_1,\ldots,x_{3k}$ in the formula of Lemma~\ref{lem:S_H} may be restricted further as follows.
\begin{Lem}\label{lem:Y-restrict}
After perturbing the coherent $v$-gradients without affecting the previous assumptions, we may restrict the $Y$-shaped Z-graphs in Corollary~\ref{cor:Y-disjoint} to those such that all the critical points attached on input white vertices are of index 2 or 3 and that all the critical points attached on output white vertices are of index 1 or 2.
\end{Lem}
Proof of Lemma~\ref{lem:Y-restrict} is a bit technical and will be given in \S\ref{ss:indices_Y}. We assume Lemma~\ref{lem:Y-restrict} in the rest of this subsection, in which case there are no output white vertex of index 0 and Lemma~\ref{lem:rel_cycle} can be applied. The following lemma describes $S_H(\sigma(1),\ldots,\sigma(2k))$ only by homological data, as done in \cite{KT}.

\begin{Lem}\label{lem:homological-formula}
Recall that $S_i=\widetilde{V}_i\cup_\partial(-V_i\times K_i)$ (defined after Definition~\ref{def:coherent}). Let $\beta_1^{(i)},\beta_2^{(i)},\beta_3^{(i)}$ be the classes in $H_*(S_i)$ given by the cores of the three handles of $V_i$ with positive indices. If a coherent $v$-gradient is chosen as in Lemma~\ref{lem:Y-restrict} and generic with respect to $\vec{V}_G$, then the following identity holds.
\begin{equation}\label{eq:Q-Q}
\begin{split}
&\sum_{x_1,\ldots,x_{3k}}\mathrm{sgn}'_{\tbvec{x}}(\sigma)\,Q_{\tbvec{x}}^{(1)}(\sigma)\cdots Q_{\tbvec{x}}^{(2k)}(\sigma)\\
&=\pm\Bigl\{\prod_{{e=(j,\ell)}\atop{\in\mathrm{Edges}(H^\circ)}}\sum_{1\leq a,b\leq 3}\ell k(\beta_a^{(j)},\beta_b^{(\ell)})\Bigr\}\prod_{i=1}^{2k}\langle \alpha_1^{(i)},\alpha_2^{(i)},\alpha_3^{(i)}\rangle_{S_i},
\end{split}
\end{equation}
where the sum in the left hand side is the same as (\ref{eq:XR}).
\end{Lem}
\begin{proof}
Recall that tach term $Q_{\tbvec{x}}^{(i)}(\sigma)=\#\acalM_{H^\circ_i(\tbvec{p}_i)}(\vec{\xi}^{G_{\sigma(i)}})$ is given by a linear combination of triple intersections like $\langle \bD_{g(p)},\bA_{p'},\bA_{p''}\rangle_{\widetilde{V}_j}$ ($j=\sigma(i)$). We claim that the triple intersections do not change if each $\bA_*$ is replaced with the cycle $F_*^{(j)}$ (defined after Definition~\ref{def:coherent}) in $S_i$ by closing it using the trivial family $B_*^{(j)}$ in $-V_j\times K_j$ and the bordism $C_*^{(j)}$ in $\partial V_j\times K_j$. 

Indeed, by the coherence of $v$-gradients, $C_\ell^{(j)}$ is collapsed into a small neighborhood of a skeleton of $\partial\widetilde{V}_j$ of less dimension. So we may consider $C_\ell^{(j)}$ as an object of less dimension. In particular, for $V_j$ of type I, if $|p_\ell^{(j)}|=2$, then $C_\ell^{(j)}$ can be considered as a map from a ribbon graph, and if $|p_\ell^{(j)}|=1$, then $C_\ell^{(j)}$ is included in a small neighborhood of a 2-dimensional subcomplex of $\partial \widetilde{V}_j$. For $V_j$ of type II, if $|p_\ell^{(j)}|=1$, then $C_\ell^{(j)}$ has parts parallel to $(2,1), (1,0), (2,0)$-handles (with 3,1,2-dimensional cores), and it is included in a small neighborhood of codimension $\geq 1$ subcomplex in $\dim\,\partial\widetilde{V}_j=4$ dimension. If $|p_\ell^{(j)}|=2$, $C_\ell^{(j)}$ has parts parallel to $(1,1)$, $(1,0)$, $(2,0)$-handles (with 2,1,2-dimensional cores), and it is included in a small neighborhood of codimension $\geq 2$ subcomplex in $\dim\,\partial\widetilde{V}_j=4$ dimension. Then the triple intersections of chains including $C_\ell^{(j)}$ may be assumed to be empty by a dimensional reason. Hence the additions of $C_\ell^{(j)}$ do not affect the triple intersections above, and $\bA_*$ can be replaced with the cycle $F_*^{(j)}$. 

Now $Q_{\tbvec{x}}^{(1)}(\sigma)\cdots Q_{\tbvec{x}}^{(2k)}(\sigma)$ can be described as follows. We consider the following formal tensor product:
\begin{equation}\label{eq:L-tensor}
 \bigotimes_{i=1}^{3k}(\bA_{x_i}\otimes \bD_{g(x_i)}). 
\end{equation}
By Lemma~\ref{lem:rel_cycle}, $\bD_{g(x_i)}$ induces a relative cycle at each $(\widetilde{V}_j',\partial \widetilde{V}_j')$ described as follows.
\[ [\bD_{g(x_i)}\cap \widetilde{V}_j]=\sum_{j,\ell} \ell k(c(x_i), c(p_\ell^{(j)})) [\bA_{p_\ell^{(j)}}]. \]
If the terms $\bD_{g(x_i)}$ in the tensor product (\ref{eq:L-tensor}) are replaced with $\sum_{j,\ell} \ell k(c(x_i), c(p_\ell^{(j)})) \bA_{p_\ell^{(j)}}$, then we obtain a linear combination of tensor products of $6k$ $\bA_*$-terms. Each term in the linear combination may be permuted to a tensor product of $2k$ triple products by using the combinatorial structure of the graph $H$, and the value of $Q_{\tbvec{x}}^{(1)}(\sigma)\cdots Q_{\tbvec{x}}^{(2k)}(\sigma)$ will be obtained after replacing the triple products with their triple intersection numbers. Here, one may see that the replacements of $\bD_{g(x_i)}$ with $\sum_{j,\ell} \ell k(c(x_i), c(p_\ell^{(j)})) \bA_{p_\ell^{(j)}}$ do not affect the triple intersection numbers, as follows. Since $\bD_{g(x_i)}\cap (\partial V_\ell\times K_i)=\bD_{g(x_i)}\cap \partial\widetilde{V}_\ell$ does not change by surgery, we may obtain a cycle of $S_\ell$ by gluing together the two $\bD_{g(x_i)}$'s before and after the surgery along the boundary, which is homologous to the one obtained from $\sum_{j,\ell} \ell k(c(x_i), c(p_\ell^{(j)})) \bA_{p_\ell^{(j)}}$ by closing with $C_\ell^{(j)}$. Since $Q_{\tbvec{x}}^{(j)}(\sigma)$ is given by products among homology classes in a closed manifold $S_j$, a change of a cycle within its homology class does not affect the resulting value. 

Based on the above observations, it follows that the left hand side of (\ref{eq:Q-Q}) can be described in the homology level as follows. 
Let $e=(j,\ell)$ be a separated edge of $H^\circ$. Consider the following element of $H_{k+5}(S_j\times S_\ell)\cong H^3(S_j\times S_\ell)$ 
\[ L_e=\sum_{ a,b}\ell k(\beta_a^{(j)},\beta_b^{(\ell)})\,\alpha_a^{(j)}\otimes \alpha_b^{(\ell)}, \]
where the sum is taken for $(a,b)$ such that $|\alpha_a^{(j)}|+|\alpha_b^{(\ell)}|=k+5$. Let $\widehat{L}_e$ be the element of $H_{9k-3}(S_1\times S_2\times \cdots\times S_{2k})$ obtained from the fundamental class $[S_1]\otimes [S_2]\otimes\cdots\otimes [S_{2k}]$ by replacing the $j$-th and the $\ell$-th factors with $L_e$. Now the following identity holds.
\[ \sum_{x_1,\ldots,x_{3k}}\mathrm{sgn}'_{\tbvec{x}}(\sigma)\,Q_{\tbvec{x}}^{(1)}(\sigma)\cdots Q_{\tbvec{x}}^{(2k)}(\sigma)
=\pm\prod_{e\in \mathrm{Edges}(H^\circ)}\widehat{L}_e \]
where the product in the right hand side is the usual intersection form among cycles $H_{9k-3}(\textstyle\prod_{i=1}^{2k} S_i)^{\otimes 3k}\to H_0(\textstyle\prod_{i=1}^{2k} S_i)=\Q$.
The product in the right hand side can be rewritten as the right hand side of the formula of the lemma.
\end{proof}

Although the sign in the right hand side of (\ref{eq:Q-Q}) depends on the graph orientation of $H^\circ$, it is cancelled out by the graph orientation after multiplying (\ref{eq:Q-Q}) to an oriented graph. 

The following lemma is immediate.
\begin{Lem}\label{lem:lk-nontrivial}
$\prod_{e=(j,\ell)\in\mathrm{Edges}(H^\circ)}\sum_{a,b}{\ell k}(\beta_a^{(j)},\beta_b^{(\ell)})\neq 0$ iff there exists an isomorphism $H\cong \pm \Gamma$ that preserves the labels of vertices. 
\end{Lem}

For an abstract trivalent graph $\Gamma$, let $\mathrm{Aut}_e\Gamma$ be the group of automorphisms of $\Gamma$ that fix all the vertices of $\Gamma$. Let $\mathrm{Aut}_v\Gamma$ be the group of permutations of vertices of $\Gamma$ that give automorphism of $\Gamma$. We have $|\mathrm{Aut}\,\Gamma|=|\mathrm{Aut}_e\Gamma||\mathrm{Aut}_v\Gamma|$.

\begin{Lem}\label{lem:eval-hatZ}
If the coherent $v$-gradients are chosen as Lemma~\ref{lem:Y-restrict} and generic with respect to $\vec{V}_G$ and if we choose orientations of the ascending manifolds $A_\ell^{(i)}$ so that the value of Lemma~\ref{lem:aaa=1} is $+1$, then the following identity folds.
\begin{equation}\label{eq:sum_S_H}
 \frac{1}{2^{3k}(2k)!(3k)!}\sum_{\sigma\in\mathfrak{S}_{2k}}\sum_{H}S_H(\sigma(1),\ldots\sigma(2k))
=(-1)^{3k}[\Gamma]. 
\end{equation}
\end{Lem}
\begin{proof}
By Lemma~\ref{lem:aaa=1} (2), the value of the right hand side of (\ref{eq:Q-Q}) agrees with $\pm\prod_{e=(j,\ell)\in\mathrm{Edges}(H^\circ)}\sum_{a,b}\ell k(\beta_a^{(j)},\beta_b^{(\ell)})$. Each nonvanishing coefficient in the sum\\ $\sum_H S_H(\sigma(1),\ldots,\sigma(2k))$ gives the same element in $\calA_k$ after multiplying it to $\mathrm{Close}(H^\circ)$. Thus 
\[ S_H(\sigma(1),\ldots,\sigma(2k))=\left\{\begin{array}{ll}
(-1)^{3k}|\mathrm{Aut}_e\Gamma|[\Gamma] & \mbox{if there exists an isomorphism $H\cong \pm \Gamma$}\\
 & \mbox{\phantom{if }that preserves the labels of vertices}\\
0 & \mbox{otherwise}
\end{array}\right. \]
Among the permutations $\sigma\in \mathfrak{S}_{2k}$ of the set of vertices, there are $|\mathrm{Aut}_v\Gamma|$ elements satisfying the condition of Lemma~\ref{lem:lk-nontrivial}. For a trivalent graph, there are $\frac{2^{3k}(2k)!(3k)!}{|\mathrm{Aut}\,\Gamma|}$ different ways of giving edge-orientations and labellings. Therefore, the LHS of (\ref{eq:sum_S_H}) is
\[ \frac{1}{2^{3k}(2k)!(3k)!}\frac{2^{3k}(2k)!(3k)!}{|\mathrm{Aut}\,\Gamma|}(-1)^{3k}|\mathrm{Aut}_v\Gamma||\mathrm{Aut}_e\Gamma|[\Gamma]=(-1)^{3k}[\Gamma].
\]
This completes the proof.
\end{proof}

\begin{proof}[Proof of Theorem~\ref{thm:Z(G)} (2)]
To complete the proof of Theorem~\ref{thm:Z(G)} (2), we shall prove that $Z_k^\Morse((\ve_1\xi^{(1)},\ldots,\ve_{3k}\xi^{(3k)}),\eta)$ agrees with $Z_k^\Morse((\xi^{(1)},\ldots,\xi^{(3k)}),\eta)$ for all $\ve_i=\pm 1$ for the coherent $v$-gradient $\xi^{(j)G}$ as in Lemma~\ref{lem:occupied}. Z-graphs for non-vanishing terms in $Z_k^\Morse((\ve_1\xi^{(1)},\ldots,\ve_{3k}\xi^{(3k)}),\eta)$ are exactly the same as those for $Z_k^\Morse((\xi^{(1)},\ldots,\xi^{(3k)}),\eta)$. Here, although the orientations of the Z-graphs may change under the replacements $\xi^{(i)}\to \ve_i\xi^{(i)}$, the orientations of graphs change accordingly and the changes do not affect the product $\#\acalM_\Gamma(\vec{\xi}^{G})\cdot \Gamma$. Thus we need only to show that in such a Z-graph a separated edge of $\ve_j\xi^{(j)}$, $\ve_j=-1$, gives the same coefficient as that of $\ve_j=1$. The analogue of the identity $\overline{\partial} g(p)=p$ in the proof of Lemma~\ref{lem:dg(p)=p} for $\ve_j=-1$ is $p=\overline{\partial}^*g^*(p)$, where $\overline{\partial}^*$, $g^*$ are given by matrix transpose for $\overline{\partial}$, $g$, respectively. Then we have
\[ \ell k(c(p),c(p'))
=\langle \bcalD_{g^*(p)}(\xi_{b_0^n}),\bcalD_{p'}(\xi_{b_0^n})\rangle_{F_0}
=\langle \bcalD_{p}(\xi_{b_0^n}),\bcalD_{g(p')}(\xi_{b_0^n})\rangle_{F_0}.\]
There are no changes in Lemma~\ref{lem:rel_cycle} for $\ve_j=-1$ except that $g(p)$ becomes $g^*(p)$, and there are no other changes in the argument of \S\ref{ss:eval-Z} except that $\bD$ and $\bA$ are exchanged. Hence by Lemma~\ref{lem:eval-hatZ} we have $Z_k^\Morse((\ve_1\xi^{(1)},\ldots,\ve_{3k}\xi^{(3k)}),\eta)=Z_k^\Morse((\xi^{(1)},\ldots,\xi^{(3k)}),\eta)=(-1)^{3k}[\Gamma]$. This completes the proof.
\end{proof}

%%%%%%%%%%%%%%%%%%%%%%%%%%%%%%%
\subsection{Proof of Lemma~\ref{lem:Y-restrict}}\label{ss:indices_Y}

Among the $Y$-shaped components $J$ in a Z-graph in $E^\Gamma$ that are of surviving type, listed in (\ref{eq:list}), 
the cases $(2,3,3\,|\,)$, $(\,|\,1,1,2)$, $(2,3\,|\,1)$, $(3,3\,|\,2)$, $(2\,|\,1,1)$, $(3\,|\,1,2)$, $(2,2,3\,|\,)$, $(\,|\,1,2,2)$, $(2,2\,|\,1)$, $(2,3\,|\,2)$, $(2\,|\,1,2)$, $(3\,|\,2,2)$
satisfy the condition of the lemma. We will show below that if $J$ has a white vertex with index 0 critical point, i.e., the index of $J$ is one of $(\,|\,0,2,2)$, $(1,3\,|\,0)$, $(2,2\,|\,0)$, $(1\,|\,0,1)$, $(2\,|\,0,2)$, $(1,2\,|\,0)$, $(1\,|\,0,2)$, then such a graph does not contribute to $Z_k^\Morse(\pi^\Gamma)$. Since an input white vertex with index 1 critical point is paired with an output white vertex with index 0 critical point, the cases $(1,3,3\,|\,)$, $(1,3\,|\,1)$, $(1\,|\,1,1)$ can then be excluded, too. This will complete the proof Lemma~\ref{lem:Y-restrict}.

\begin{Lem}
Let $J$ be  a $Y$-shaped component in $I$ as above. If $J$ has a white vertex with index 0 critical point, i.e., the index of $J$ is one of the following:
\begin{enumerate}
\item (Type I)\quad $(\,|\,0,2,2)$, $(1,3\,|\,0)$, $(2,2\,|\,0)$, $(1\,|\,0,1)$, $(2\,|\,0,2)$, 
\item (Type II)\quad $(1,2\,|\,0)$, $(1\,|\,0,2)$, 
\end{enumerate}
then such a graph $I$ does not contribute to $Z_k^\Morse(\pi^\Gamma)$.
\end{Lem}
\begin{proof}
\underline{$(*,*\,|\,0)$}: Let $K$ be a $Y$-shaped component in $I$ of index $(2,2\,|\,0)$ (type I). Suppose that $K$ has a black vertex in $\widetilde{V}_i=V_i'\tcoprod (-V_i)$ and that the half-edges of $K$ with critical points of indices $2,2,0$ are labelled by $k,\ell,m$ respectively. By (\ref{eq:XR}), we may consider that on the input white vertex of $K$ of index 2 labeled by $k$, the Morse chain $g(p)$ for some critical point $p$ of $f^{(k)}_{b_0^n}$ of index 1 with negative critical value is  attached, extending $I$ to chains. Let $q$ be the critical point of $f^{(\ell)}_{b_0^n}$ of index 2 that is attached on $\ell$ in $K$. According to Lemma~\ref{lem:rel_cycle}, $g(p)$ induces a relative 2-cycle $\sigma$ in $(\widetilde{V}_i,\partial \widetilde{V}_i)$. Let $r$ be the critical point of $f^{(m)}_{b_0^n}$ of index 0 attached on $m$ in $K$. 

If $q$ lies in $\widetilde{V}_i$, then $q$ is a 2-cycle in the Morse complex of the fiber over $b_0^n$. By Lemma~\ref{lem:dg(p)=p} (2), $K$ gives $\langle \sigma, \bA_r,\bD_q\rangle_{\widetilde{V}_i}=\langle \sigma, \bcalD_q\rangle_{\widetilde{V}_i}=\pm({\ell k}(c(p),c(q))-{\ell k}(c(p),c(q)))=0$,
where the two terms of $\ell k$ correspond to the components $V_i'$ and $-V_i$ respectively. Note that the linking number is invariant under the surgery on $G_i$ by the homological triviality result of Proposition~\ref{prop:wh-prod}. Hence the Z-graph $I$ having such $K$ does not contribute to $Z_k^\Morse(\pi^\Gamma)$. The same argument can be applied to $K$ of index $(1,3\,|\,0)$ with $|g(p)|=3$, $|q|=1$.

If $q$ lies outside $\widetilde{V}_i$, it suffices to consider $K$ with $q$ replaced by $g(q')$ for some index 1 critical point $q'$. As for $g(p)$ on the half-edge $k$, $g(q')$ induces a relative 2-cycle $\sigma'$ in $(\widetilde{V}_i,\partial \widetilde{V}_i)$. By gluing $\partial V_i'$ and $-\partial V_i$ together, we may think $\sigma$ and $\sigma'$ are 2-cycles in $\overline{S}_i=V_i'\cup_\partial (-V_i)$. The intersection number $\langle \sigma,\bA_r, \sigma'\rangle_{\widetilde{V}_i}=\sigma\cdot \sigma'$ can be interpreted as that for two 2-cycles in $\overline{S}_i$, which vanishes since the image of the cup product $H^2(\overline{S}_i)\otimes H^2(\overline{S}_i)\to H^4(\overline{S}_i)$ of such 2-cycles is zero (Lemma~\ref{lem:aaa=1} (1)). 

If $K$ is of index $(1,3\,|\,0)$ (type I) and $q$ lies outside $\widetilde{V}_i$, then let $g(p), q, r$ with $|p|=2, |q|=1, |r|=0$ be the Morse chains attached to the white vertex of $K$. Let $\sigma$ be the relative 3-cycle in $(\widetilde{V}_i,\partial\widetilde{V}_i)$ induced from $g(p)$. A leg of $\bD_q$ converges to a critical locus $s$ of index 0 in $\widetilde{V}_i$ by assumption. We remark that since $\bD_q$ induces a relative 1-cycle of $(\widetilde{V}_i,\partial \widetilde{V}_i\cup s)$, the value of $\langle \sigma, \bA_r,\bD_q\rangle_{\widetilde{V}_i}$ will not change by adding to $\sigma$ the boundary of a 4-chain in $\widetilde{V}_i$ that does not meet $\partial \widetilde{V}_i\cup s$. To obtain a good modification of $\sigma$, we now assume that the $v$-gradient for the fiberwise Morse function $\widetilde{\mu}_i$ and its perturbations $\widetilde{\mu}_i^{(1)},\cdots,\widetilde{\mu}_i^{(3k)}$ are as follows. We assume that the support of the relative diffeomorphisms of $\partial V_i$ for the surgery on $G_i$ (Definition~\ref{def:alpha-I}) is included in a thin handlebody $R$ with at most 2-handles in $\partial \widetilde{V}_i$. Then we may assume that for a real number $\kappa<0$ with small absolute value, both the manifolds $\widetilde{V}_i$ and the gradient-like vector field for $\widetilde{\mu}_i$ agree on the complement of $R\times [\kappa,0]$, where the direct product structure $R\times[\kappa,0]$ is the one generated by the gradient-like vector field for $\widetilde{\mu}_i$, and we consider that $\widetilde{V}_i$ is obtained from $V_i\times K_i$ by surgery within $R\times [\kappa,0]$. Furthermore, we may assume that this property is satisfied for every $\xi^{(j)G_i}$. By perturbing the gradient-like vector field for $f^{(\ell)}$ in $\R^4-V_i$, we may assume that $\bD_q\cup s$ does not meet $R\times [\kappa,0]$ in $\widetilde{V}_i$. Note that this perturbation is independent of the gradient-like vector fields in $\widetilde{V}_i$'s and does not affect all the previous vanishing results. 

We may assume that the relative cycle $\sigma$ is disjoint from $s$. By adding the boundary of a 4-chain in $\widetilde{V}_i$ to $\sigma$ that does not meet $\partial \widetilde{V}_i\cup s$, we may assume that the restriction of $\sigma$ on the complement of $R\times [\kappa,0]$ is independent of parameter in $K_i=S^0$. See Figure~\ref{fig:modify-sigma-G}, which summarizes the change of $\sigma$. Under all the above assumptions, $\langle \sigma, \bA_r,\bD_q\rangle_{\widetilde{V}_i}=0$ by cancellation since the intersections avoid $R\times [\kappa,0]$, and hence the Z-graph $I$ does not contribute to $Z_k^\Morse(\pi^\Gamma)$ in this case, too. 
\begin{figure}
\includegraphics[height=30mm]{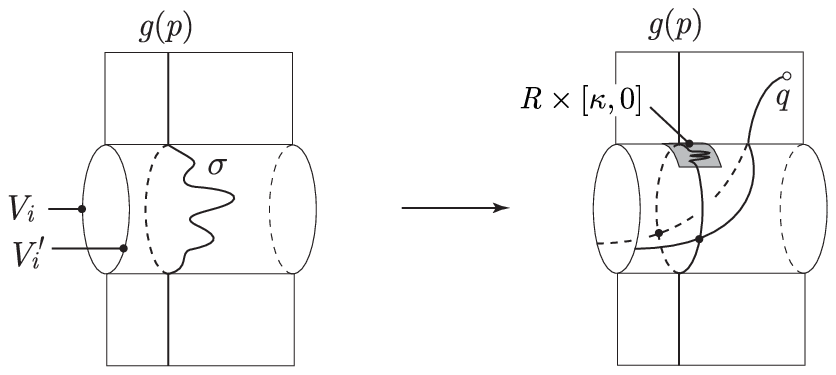}
\caption{}\label{fig:modify-sigma-G}
\end{figure}

If $K$ is of index $(1,2\,|\,0)$ (type II), then the proof of vanishing is almost the same as above for $(1,3\,|\,0)$, except that $|p|=1$, $|q|=1$, $\ell k$ is replaced with $\widetilde{\ell k}$.

\underline{$(\,|\,0,2,2)$}: Let $K$ be a $Y$-shaped component in $I$ of index $(\,|\,0,2,2)$ (type I). Suppose that the black vertex of $K$ lies in $\widetilde{V}_i$ and that the output half-edges of $K$ with critical points $q,r,s$ of index 2,2,0 are labelled by $k,\ell,m$ respectively. Then $K$ gives $\langle \bA_q,\bA_r,\bA_s\rangle_{\widetilde{V}_i}=\langle \bA_q,\bA_r\rangle_{\widetilde{V}_i}$. By coherence of the $v$-gradients, the bordism $C_j^{(i)}$ of Lemma~\ref{lem:ribbon-graph-I} for $p_j^{(i)}=q$ or $r$ is a map from a ribbon graph, and two such do not have intersection in the 3-manifold $\partial \widetilde{V}_i$ generically. Thus we have $\langle \bA_q,\bA_r\rangle_{\widetilde{V}_i}=\langle F_j^{(i)}(\xi^{(k)G_i}),F_{j'}^{(i)}(\xi^{(\ell)G_i})\rangle_{S_i}$, which is 0 by Lemma~\ref{lem:aaa=1} (1). Therefore, it follows that $I$ does not contribute to $Z_k^\Morse(\pi^\Gamma)$, either. 

\underline{$(*\,|\,0,*)$}: Let $K$ be a $Y$-shaped component in $I$ of index $(1\,|\,0,1)$ (type I). Suppose that the black vertex of $K$ lies in $\widetilde{V}_i$ and that the input and output half-edges of $K$ with critical points of index 1 are labelled by $k,\ell$ respectively. Suppose that the output half-edge of $K$ with critical point of index 0 is labelled by $m$. Let $q,r,s$ be the critical points of index $1,1,0$ attached on the half-edges $k,\ell,m$, respectively. Then we may assume that $q$ lies outside $\widetilde{V}_i$, and $K$ gives $\langle \bD_q, \bA_s, \bA_r\rangle_{\widetilde{V}_i}=\langle \bD_q, \bA_r\rangle_{\widetilde{V}_i}$. By the same argument as $(1,3\,|\,0)$ with both input white vertices outside $\widetilde{V}_i$, the difference of $\bA_r$ before and after the surgery on $G_i$ can be squeezed into $R\times [\kappa,0]$ for a thin handlebody $R$ in $\partial \widetilde{V}_i$. Since $\bD_q$ can be made disjoint from $R\times [\kappa,0]$ by perturbing the gradient-like vector field for $f^{(k)}$ in $\R^4-V_i$, the intersection $\langle \bD_q, \bA_r\rangle_{\widetilde{V}_i}$ is zero by cancellation, and the flow-graph $I$ does not contribute to $Z_k^\Morse(\pi^\Gamma)$ in this case. The case $(1\,|\,0,2)$ (type II) is almost the same as this, where the difference of the ascending manifolds of index 2 can be squeezed into the union of $(1,1), (1,0), (2,0)$-handles on $\partial \widetilde{V}_i$ and disjoint from the locus of the descending manifold from an index 1 critical locus. The case $(2\,|\,0,2)$ (type I) is similar to the case where $K$ is of index $(2,2\,|\,0)$ with $q$ not in $\widetilde{V}_i$, in which case the proof is the same as before except that $\bA_q$ for a critical point $q$ of index 2 of an output white vertex is closed by using $C_j^{(i)}$, and the count of $K$ is $\langle\sigma,\bA_q,\bA_r\rangle_{\widetilde{V}_i}=\sigma\cdot \bA_q = \sigma\cdot F_j^{(i)}$. This vanishes by the triviality of the cup product in $\overline{S}_i$ for the corresponding classes (Lemma~\ref{lem:aaa=1} (1)).

Finally, note that the gradient-like vector fields may be perturbed so that every $Y$-shaped component in any $I$ with indices in the given list (1), (2) of the statement simultaneously vanish.
\end{proof}

%%%%%%%%%%%%%%%%%%%%%%%%%%%%%%%
\appendix

%%%%%%%%%%%%%%%%%%%%%%%%%%%%%%%
\mysection{Orientations on manifolds and their intersections}{s:ori}

%%%%%

For a $d$-dimensional orientable manifold $M$, we will represent an orientation on $M$ by a nowhere vanishing $d$-form of $\Omega_{\mathrm{dR}}^d(M)$ and denote by $o(M)$. If $M$ is a submanifold of an oriented Riemannian $e$-dimensional manifold $E$, then we may alternatively define $o(M)$ from an orientation $o^*_E(M)$ of the normal bundle of $M$ by the rule
\begin{equation}\label{eq:coori}
 o(M)\wedge o^*_E(M)\sim o(E). 
\end{equation}
Note that $o^*_E(M)$ is defined canonically by the Hodge star operator: $o^*_E(M)=*o(M)$. $o^*_E(M)$ is called a {\it coorientation} of $M$ in $E$. We assume that (\ref{eq:coori}) is always satisfied so that coorientation is just an alternative way to represent orientation. 

Let $N$ be an oriented smooth manifold and let $\pi:N\to E$ be a smooth map that is transversal to $M$. Then the preimage $\pi^{-1}M$ is naturally an oriented submanifold of $N$. We may define the coorientation of $\pi^{-1}M$ by $\pi^* o^*_E(M)$. We denote simply by $o^*_E(M)$ the coorientation $\pi^* o^*_E(M)$. 

If $M$ has boundary $\partial M$, we provide an induced orientation on $\partial M$ from $o(M)$ as follows: let $n$ be an outward normal vector field on $\partial M$, then we define 
\begin{equation}\label{eq:inward_first}
 o(\partial M)_x=\iota(n_x)o(M)_x. 
\end{equation}

Suppose $M$ and $M'$ are two cooriented submanifolds of $E$ of dimension $i$ and $j$ that intersect transversally. The transversality implies that at an intersection point $x$, the form $o^*_E(M)_x\wedge o^*_E(M')_x$ is a non-trivial $(2e-i-j)$-form. We define 
\begin{equation}\label{eq:coori_int}
 o^*_E(M\pitchfork M')_x=o^*_E(M)_x\wedge o^*_E(M')_x. 
\end{equation}
This depends on the order of the product. 

For an $F$-bundle $\pi:E\to B$, with both base and fiber oriented, we orient the total space by 
\begin{equation}\label{eq:o(E)}
  o(E)_x=o(B)_x\wedge o(F_{\pi(x)})_x.
\end{equation}

%%%%%%%%%%%%%%%%%%%%%%%%%%%%%%%
\mysection{Transversality of flow-graphs}{s:transversality}

%%%%%
\subsection{Proof of Lemma~\ref{lem:0-mfd}}\label{ss:transversality-family}
It suffices to prove that the spaces of vertical flow-graphs in fibers satisfy the property of the assertion. 
Let $r$ be a sufficiently large integer and let $\calX(F_0)$ be the space of $C^r$ gradient-like vector fields on the base fiber $F_0$ that satisfy Assumption~\ref{hyp:eta}. Let $\calU^{(\ell)}\subset \calX(F_0)$ be a small neighborhood of a Morse--Smale element $\xi_0^{(\ell)}$. Let $O$ be an open $k$-disk and let $\widetilde{\calU}^{(\ell)}$ be the space of $C^r$ maps $O\to \calU^{(\ell)}$. An element of $\widetilde{\calU}^{(\ell)}$ gives an $O$-family of gradient-like vector fields on $F_0$, which can be considered as a $v$-gradient on $F_0\times O$.

Now we take a $\vec{C}$-graph $\Gamma$. We define $\calE_\ell(\xi^{(\ell)})$ ($\ell=1,2,\ldots,3k$) as follows. If the $\ell$-th edge of $\Gamma$ is separated and if the critical points $p$ and $q$ of $\xi_0^{(\ell)}$ are attached respectively to the input and output white vertex of the $\ell$-th edge, then we define $\calE_\ell(\xi^{(\ell)})=\calA_q(\xi^{(\ell)})\times_O \calD_p(\xi^{(\ell)})$, where we also denote by $p,q$ the critical loci of the $v$-gradient $\xi^{(\ell)}$ on $F_0\times O$ which correspond to the critical points $p,q$ of $\xi_0^{(\ell)}$, by abuse of notation. If the $\ell$-th edge of $\Gamma$ is compact, we define $\calE_\ell(\xi^{(\ell)})=\calM_2(\xi^{(\ell)})$. We define the map 
\[ \Phi:\bigcup_{\xi^{(\ell)}\in\widetilde{\calU}^{(\ell)}} \calE_1(\xi^{(1)})\times_O\calE_2(\xi^{(2)})\times_O\cdots\times_O\calE_{3k}(\xi^{(3k)}) \to F_0^{6k} \]
between Banach manifolds as follows\footnote{For fiberwise spaces $X_i$ over $O$, we write $X_1\times_O\cdots\times_O X_r=\int_{s\in O}X_1(s)\times\cdots\times X_r(s)$.}. A point of $\calE_\ell(\xi^{(\ell)})$ can be represented by $s\in O$ and a pair $(x_\ell,y_\ell)$ of noncritical endpoints of a possibly separated flow-line. Thus, a point of $\calE_1(\xi^{(1)})\times_O\calE_2(\xi^{(2)})\times_O\cdots\times_O\calE_{3k}(\xi^{(3k)})$ can be represented by $(x_1,y_1)\times(x_2,y_2)\times\cdots\times (x_{3k},y_{3k})\times \{s\}$, and its fiber components $(x_1,y_1)\times(x_2,y_2)\times\cdots\times (x_{3k},y_{3k})$ gives a point of $F_0^{6k}$. We define $\Phi(\{\xi^{(\ell)}\};(x_1,y_1)\times(x_2,y_2)\times\cdots\times (x_{3k},y_{3k})\times \{s\})$ to be this point. For each trivalent vertex $v$ of $\Gamma$, we collect the three components from $(x_1,y_1)\times(x_2,y_2)\times\cdots\times (x_{3k},y_{3k})$ corresponding to the endpoints of the three edges incident to $v$ on the $v$-side, and we identify $F_0^{6k}$ (ordered with respect to the edge numbering) with $F_0^3\times\cdots\times F_0^3$ (ordered with respect to the vertex numbering). Then we denote the main diagonal $\{(x,x,x)\mid x\in F_0\}$ of $F_0^3$ by $\Delta^{(3)}$, and put $\Delta=\Delta^{(3)}\times\cdots\times \Delta^{(3)}$, which is naturally diffeomorphic to $F_0^{2k}$. 
\begin{Lem}\label{lem:transversal}
$\Phi$ is transversal to $\Delta$. Hence $\Phi^{-1}(\Delta)$ is a $C^r$ submanifold of codimension $16k$.
\end{Lem}
\begin{proof}
The idea is almost the same as in \cite{Fu}. Namely, for each point $z\in \mathrm{Im}\,\Phi\cap \Delta$, one may see that any component of $z$ can be shifted in arbitrary tangent direction in $TF_0$ by a small perturbation of a family of gradient-like vector fields near the trivalent vertex. It follows that the tangent space $T_zF_0^{6k}$ is spanned by $\mathrm{Im}\,d\Phi$ and $T_z\Delta$. That $\Phi^{-1}(\Delta)$ is a $C^r$ submanifold follows from the local presentation of transversality (e.g., \cite[Corollary 17.2]{AR}). 
\end{proof}

\begin{Lem}\label{lem:Fredholm}
The projection $\pi:\Phi^{-1}(\Delta)\to \widetilde{\calU}^{(1)}\times \widetilde{\calU}^{(2)}\times\cdots\times \widetilde{\calU}^{(3k)}$ is Fredholm, and its index is 0.
\end{Lem}
\begin{proof}
The dimension of the fiber of the projection $\hat{\pi}:\mathrm{Dom}\,\Phi\to \widetilde{\calU}^{(1)}\times \widetilde{\calU}^{(2)}\times\cdots\times \widetilde{\calU}^{(3k)}$ is $k+5\times 3k=16k$. For $z\in \Phi^{-1}(\Delta)$, the linear map $d\pi_z:T_z\Phi^{-1}(\Delta)\to T_{\pi(z)}(\widetilde{\calU}^{(1)}\times \widetilde{\calU}^{(2)}\times\cdots\times \widetilde{\calU}^{(3k)})$ agrees with the composition of Fredholm operators $T_z\Phi^{-1}(\Delta)\stackrel{\subset}{\to} T_z\mathrm{Dom}\,\Phi\stackrel{d\hat{\pi}_z}{\to} T_{\hat{\pi}(z)}(\widetilde{\calU}^{(1)}\times \widetilde{\calU}^{(2)}\times\cdots\times \widetilde{\calU}^{(3k)})$, and it follows that $\pi$ is Fredholm. Moreover, by the additivity of the index of Fredholm operators under composition, we see that the index of $\pi$ is $(-16k)+16k=0$.
\end{proof}

\begin{proof}[Proof of Lemma~\ref{lem:0-mfd}]
It follows from the Sard--Smale theorem (\cite{Sm3}, \cite[Theorem~16.2]{AR}) that the set of regular values of $\pi$ is a residual subset of $\prod_\ell\widetilde{\calU}^{(\ell)}$. By Lemmas~\ref{lem:transversal}, \ref{lem:Fredholm}, we see that the fiber of $\pi$ over each regular value $(\xi^{(1)},\ldots,\xi^{(3k)})$ is a 0-manifold. Here, one may check that $(\xi^{(1)},\ldots,\xi^{(3k)})$ is a regular value of $\pi$ if and only if the restriction $\Phi_{(\xi^{(1)},\ldots,\xi^{(3k)})}:\calE_1(\xi^{(1)})\times_O\calE_2(\xi^{(2)})\times_O\cdots\times_O\calE_{3k}(\xi^{(3k)})\to F_0^{6k}$ of $\Phi$ on $\pi^{-1}(\{(\xi^{(1)},\ldots,\xi^{(3k)})\})$ is transversal to $\Delta$. Since $C^\infty$ sections are $C^r$ dense in the space of $C^r$ sections, one can find a $C^\infty$ tuple in the image of $\pi$ that is arbitrarily $C^r$ close to a regular value $(\xi^{(1)},\ldots,\xi^{(3k)})$, by which one can approximate $\calE_1(\xi^{(1)}),\calE_2(\xi^{(2)}),\ldots,\calE_{3k}(\xi^{(3k)})$ by $C^\infty$ submanifolds arbitrarily $C^r$ closely. Such approximations inherit the transversality property of $\Phi_{(\xi^{(1)},\ldots,\xi^{(3k)})}$, and it follows that we may assume after a small perturbation that the regular value $(\xi^{(1)},\ldots,\xi^{(3k)})$ consists of $C^\infty$ $v$-gradients. 

One may see that the argument above also works even if $\calE_\ell(\xi^{(\ell)})$ are replaced with their compactifications. This completes the proof that $\acalM_\Gamma(\vec{\xi},\eta)$ is a compact 0-manifold, for a single $\vec{C}$-graph $\Gamma$. Since there may be only finitely many $\vec{C}$-graphs with $2k$ black vertices and $\deg(\Gamma)=(1,\ldots,1)$, and finite intersection of residual subsets is residual in $\prod_\ell\widetilde{\calU}^{(\ell)}$, the assertion follows for all the relevant $\vec{C}$-graphs.
\end{proof}

%%%%%
\subsection{Proof of the analogue of Lemma~\ref{lem:0-mfd} for $v$-gradients that are generic with respect to $\vec{V}_G$}\label{ss:transversality-surgery}

Here, we assume that the open $k$-disk $O$ in the proof of Lemma~\ref{lem:0-mfd} is $O_1\times O_2\times\cdots\times O_{2k}$, where $O_i$ is a point or an open interval. We take $\xi_0^{(\ell)}$ to be one that is adapted to $\vec{V}_G$, as in \S\ref{ss:fiberwise-MF}, and we replace $\widetilde{\calU}^{(\ell)}$ with the following one.
\[ \widetilde{\calU}^{(\ell)}=\calU_+^{(\ell)}\times \widetilde{\calU}^{(\ell)}_{1-}\times\cdots\times \widetilde{\calU}^{(\ell)}_{2k-} \]
Here, $\calU_+^{(\ell)}$ is a small neighborhood of $\xi^{(\ell)}_0|_{\R^4-\mathrm{Int}(V_1\cup\cdots\cup V_{2k})}$ in the space of $v$-gradients on $\R^4-\mathrm{Int}(V_1\cup\cdots\cup V_{2k})$ that agree with $\xi^{(\ell)}_0$ near the boundary. We denote by $\calU_{j-}^{(\ell)}$ a small neighborhood of $\xi^{(\ell)}_0|_{V_j}$ in the space of $v$-gradients on $V_j$ that agree with $\xi^{(\ell)}_0$ near the boundary, and define $\widetilde{\calU}_{j-}^{(\ell)}$ to be the space of $C^r$ maps $O_j\to \calU_{j-}^{(\ell)}$. In this situation, $\widetilde{\calU}^{(\ell)}$ still has enough freedom so that the proofs of Lemmas~\ref{lem:transversal}, \ref{lem:Fredholm} are almost the same as above.

\section*{\bf Acknowledgments.} 
I would like to thank Tatsuro Shimizu for explaining to me his work and like to thank Takuya Sakasai for giving me information about Kontsevich's graph complex. I would also like to thank Fran\c{c}ois Laudenbach, Andrew Lobb, Syunji Moriya, Kaoru Ono, Osamu Saeki, Keiichi Sakai, Makoto Sakuma and Masatoshi Sato for helpful discussions, comments and encouragements. This work was partially supported by JSPS Grant-in-Aid for Scientific Research 17K05252 and 26400089 and by the Research Institute for Mathematical Sciences, a Joint Usage/Research Center located in Kyoto University.

%%%%%%%%%%%%%%%%%%%%%%%%%%%%%%%
%%%%%%%%%%%%%%%%%%%%%%%%%%%%%%%

\end{document}